\documentclass[reqn,11pt]{amsart}
\usepackage{amsmath,amsthm,amsfonts,color,graphicx}
\usepackage{pdfsync,float}

\definecolor{lightGray}{RGB}{235,235,235}
\definecolor{orange}{RGB}{255,128,0}
\definecolor{ucib}{RGB}{0,36,105}
\definecolor{mygreen}{RGB}{0,128,0}
\definecolor{lightBlue}{RGB}{102,153,204}

\newtheorem{thm}{Theorem}[section]
\newtheorem{lem}[thm]{Lemma}

\newtheorem{prop}[thm]{Proposition}
\newtheorem{conj}[thm]{Conjecture}
\newtheorem{rems}[thm]{Remarks}
\newtheorem{rem}[thm]{Remark}

\newtheorem{deff}[thm]{Definition}

\numberwithin{equation}{section}

\begin{document}
\bibliographystyle{plain}

\title[A Parabolic Oscillator]{On Wiener's Violent Oscillations,
  Popov's curves and Hopf's Supercritical Bifurcation for a Scalar
  Heat Equation}

\author{Patrick Guidotti}
\author{Sandro Merino}
\address{University of California, Irvine\\
Department of Mathematics\\
340 Rowland Hall\\
Irvine, CA 92697-3875\\ USA }

\address{Basler Kantonalbank \\
Brunngaesslein 3\\ 
CH-4002 Basel\\
Switzerland}
\email{gpatrick@math.uci.edu and sandro.merino@bkb.ch}

\begin{abstract}
A parameter dependent perturbation of the spectrum of the scalar
Laplacian is studied for a class of nonlocal and non-self-adjoint rank
one perturbations.  
A detailed description of the perturbed spectrum is obtained both for 
Dirichlet boundary conditions on a bounded interval as well as for the
problem on the full real line.
The perturbation results are applied to the study of a related
parameter dependent nonlinear and nonlocal  
parabolic equation. The equation models a feedback system that
e.g. can be interpreted as a thermostat device  
or in the context of an agent based price formation model for a market. 
The existence and the stability of periodic self-oscillations of the
related nonlinear and nonlocal heat equation that  
arise from a Hopf bifurcation is proved. The bifurcation and stability
results are obtained both for the nonlinear parabolic  
equation with Dirichlet boundary conditions and for a related problem
with nonlinear Neumann boundary conditions that  
model feedback boundary control. 
The bifurcation and stability results follow from a Popov criterion
for integral equations after reducing the  
stability analysis for the nonlinear parabolic equation to the study
of a related nonlinear Volterra integral equation. 
While the problem is studied in the scalar case only it can be
extended naturally to arbitrary euclidean dimension and  
to manifolds.           
\end{abstract}


\maketitle

\section{Introduction}
In \cite{GM97} the authors consider a simple model of a
one-dimensional temperature control system given by
\begin{equation}\label{tsEq}
  \begin{cases}
    u_t-u_{xx}=0&\text{in }(0,\infty)\times(0,\pi),\\
    u_x(t,0)=\tanh \bigl( \beta u(t,\pi) \bigr)&\text{for }t\in(0,\infty),\\
    u_x(t,\pi)=0&\text{for }t\in(0,\infty),\\
    u(0,\cdot)=u_0&\text{in }(0,\pi),
  \end{cases}
\end{equation}
where heat is injected/removed from the interval $[0,\pi]$ at the left
endpoint $x=0$ based on a temperature measurement taken at the other
endpoint $x=\pi$. The system is controlled by the parameter $\beta>0$
which models the intensity of the heat injection/removal. The trivial
solution $u\equiv 0$ represents the desired equilibrium state of the
system. As it turns out, it can only be obtained with certainty
(and independently of the initial state) up to
the critical parameter value $\beta _0\approx 5.6655$, at which a Hopf
bifurcation occurs causing the loss of linear stability of the trivial
steady-state and the appearance of periodic solutions. 
The problem first introduced in \cite{GM97} was inspired by a remark in 
N. Wiener's book ``Cybernetics''  on possible violent temperature
oscillations for a badly designed thermostat device that is quoted in
\cite{GM97}. Subsequently the problem received attention in a series
of papers \cite{CBOP03}, \cite{DG18}, \cite{DGK16}, \cite{GeWe06(1)}, 
\cite{GeWe06(2)}, \cite{KmcK04}, \cite{L18} mainly due to its novelty and the 
interesting properties hidden behind its apparent simplicity.\\ 
The Hopf bifurcation phenomenon engendered by the nonlocal nature of the 
boundary condition has also inspired an application, presented in \cite{GoGuSo13},  
to a market price formation model introduced by J.M. Lasry and P.L. Lions. 
In particular, in that specific context a similar Hopf bifurcation scenario 
shows that ``demand" and ``supply" do not simply lead to unique equilibrium prices 
but can produce price oscillations. The phenomenon emerges on the basis of the 
modelled behaviour of the population densities of buyers and sellers 
positioned in a liquid market over a continuum of prospective transaction prices. \\
Problem \eqref{tsEq} can be
conventiently weakly formulated as the abstract Cauchy problem
$$
\begin{cases}
  \dot u +Au = -\tanh\bigl(\beta u(t,\pi)\bigr)\delta_0,&t>0,\\
  u(0)=u_0,&
\end{cases}
$$
in the space $\operatorname{H}^{-1}=\operatorname{H}^1(0,\pi)'$,
where the unbounded operator $A$ defined on $\operatorname{H}^{-1}$ and with
domain $\operatorname{dom}(A)=\operatorname{H}^1(0,\pi)$ is the one
induced by the Dirichlet form $a(u,v)=\int_0^\pi u_xv_x\, dx$ defined
on the product space $\operatorname{H}^1(0,\pi)\times
\operatorname{H}^1(0,\pi)$. We refer to \cite{GM97} and \cite{GM99} for additional
details.

In this paper, taking inspiration from that model we consider the
following heat conduction problem
\begin{equation}\label{lineTs}
  \begin{cases}
    u_t+A_Lu=-f\bigl(\beta \langle \delta_{x_0},u \rangle \bigr)\delta_0, &
    t>0,\\
    u(0)=u_0,&
  \end{cases}
\end{equation}
for the unbounded operator $A_L:\operatorname{H}^1_L\subset
\operatorname{H}^{-1}_L\to\operatorname{H}^{-1}_L$,
where it holds that $\operatorname{H}^1_L:=\operatorname{H}^1_0\bigl(
(-L,L)\bigr)$ and that $\operatorname{H}^{-1}_L:=\operatorname{H}^1_0\bigl(
(-L,L)\bigr)'=\operatorname{H}^{-1} \bigl( (-L,L)\bigr)$, and where $A_L$
is the operator induced by the Dirichlet form
$$
 a:\operatorname{H}^1_L\times\operatorname{H}^1_L\to
 \mathbb{R},\: (u,v)\mapsto a(u,v)=\int_{-L}^L u_xv_x\, dx
$$
on the interval $(-L,L)$ with $L\in(0,\infty]$ and for a smooth
bounded globally Lipschitz non-linearity $f$ satisfying the conditions
$f(0)=0$, $f'(0)=1$, and
$\operatorname{sign}(f)=\operatorname{sign}(\operatorname{id}_\mathbb{R})$. We
also assume, without loss of any generality, that $x_0\in(0,L)$.
The Cauchy problem \eqref{lineTs} can be thought of as a heat conduction
model with a source placed in the origin which is controlled by a
temperature measurement at another point $x_0$ in the domain.\\
The rest of the paper is organized as follows. Sections 2 and 3 are
devoted to the study of the linear problems for $L=\infty$ and
$L<\infty$, respectively. In particular, a detailed understanding of
the dependence of the spectrum of $A_{\beta,L}=A_L+\beta \delta_0
\delta_{x_0}^\top$ on the parameter $\beta$ is obtained. The main
results of this paper require the preparatory ground work of Section
4 on the instrumental Volterra integral equation associated with
\eqref{lineTs}. They are given in Section \ref{SectionVIE}, Theorem
\ref{vie_L_result} for the Volterra integral equation and, in Section
\ref{SecHopfRes}, Theorems \ref{Main_Result_L} and
\ref{Stability_Periodic_L}, for the nonlinear heat equation
\eqref{lineTs}. In Section \ref{SecHopfRes} we also state Theorem
\ref{Stability_Periodic_Neumann} that settles a 
conjecture of \cite{GM97} and that motivates the approach described in
\cite{GM20} and the analysis performed in the present article. The main
results are valid for $L<\infty$ only. The case when $L=\infty$ poses
additional difficulties and may be the subject of further
research. One difficulty incurred when $L=\infty$ is the fact that the
continuous spectrum of the linearization is not bounded away from the
imaginary axis. Nevertheless a partial investigation of the case $L=\infty$
is included since it is simpler, in certain aspects, and contributes
to the understanding of the case $L<\infty$.


\section{The Linear Problem on the Real Line}
We first consider the linearized problem obtained by choosing
$f=\operatorname{id}_\mathbb{R}$. This amounts to understanding the
operator
\begin{equation}\label{abeta}
A_{L,\beta}=A_L+\beta \delta_0 \delta_{x_0}^\top:\operatorname{H}^1_L\subset
\operatorname{H}^{-1}_L\to \operatorname{H}^{-1}_L,
\end{equation}
where we use the suggestive notation $\delta_{x_0}^\top$ for the
trace/evaluation operator $\gamma_{x_0}$ at the point $x_0$. The
operator $A_{L,\beta}$ is a relatively bounded rank one perturbation of
$A_L$ by $B=\beta \delta_0 \delta_{x_0}^\top$, for which it is
well-known that $-A_L$ generates an analytic
$c_0$-semigroup $T_L(t)=e^{-tA_L}$ on $\operatorname{H}^{-1}_L$. When
the case $L=\infty$ is considered, the index will be dropped for simplicity so
that, e.g., $A$ and $\operatorname{H}^{\pm 1}$ will be used instead of
$A_\infty$ and $\operatorname{H}^{\pm 1}_\infty$, if a more consistent
notation were to be applied. When $L=\infty$, it is well-known that
$$
 e^{-tA}u_0=\frac{1}{\sqrt{4\pi t}}e^{-\frac{|\cdot|^2}{4t}}*u_0,\:
 u_0\in \operatorname{L}^1(\mathbb{R}),
$$
whereas the case $L<\infty$ will be discussed in more detail
later. The perturbation $B$ satisfies
$$
 B\in \mathcal{L} \bigl(\operatorname{H}^{\frac{1}{2}+\varepsilon}_L,
 \operatorname{H}^{-\frac{1}{2}-\varepsilon}_L\bigr),
$$
for any $\varepsilon\in(0,\frac{1}{2}]$ and any $L\in(0,\infty]$, due
to the embedding $\operatorname{H}^{\frac{1}{2}+\varepsilon}_L\hookrightarrow
\operatorname{BUC(-L,L)}$ and due to $\delta_0\in
\operatorname{H}^{-\frac{1}{2}-\varepsilon}_L$.
The notation $\operatorname{BUC}(-L,L)$ refers to the space of
bounded and uniformly continuous real-valued functions defined on
$(-L,L)$. The following simple remark is useful for the case $L=\infty$.
\begin{rem}
It holds that $\operatorname{H}^1\hookrightarrow
\operatorname{C}_0(\mathbb{R})$. This follows from the fact that
$(1+|\cdot|^2)^{\frac{1}{2}}\hat u\in \operatorname{L}^2(\mathbb{R})$
by the definition of $\operatorname{H}^1$, which, in turn yields
$$
 \int |\hat u(\xi)|\, d\xi\leq \bigl[ \int (1+|\xi|^2)|\hat u(\xi)|^2\,
 d\xi\bigr]^{\frac{1}{2}}\bigl[\int\frac{1}{1+|\xi|^2}\, d\xi\bigr]
 ^{\frac{1}{2}} \leq c\| u\| _{\operatorname{H}^1}
$$
The Riemann-Lebesgue Lemma then gives the claimed embedding since
$$
 u=\mathcal{F} ^{-1}(\hat u)=\widetilde{\mathcal{F}(\hat u)},
$$
where $\tilde f(x):=f(-x)$ and 
$\hat f(\omega)\equiv\mathcal{F}(f)(\omega):=\int e^{-ix\omega}f(x)\, dx$ is the
standard Fourier transform. 
\end{rem}
Returning to the operator $B$, we see that it is indeed a relatively
bounded perturbation thanks to the interpolation inequality for Bessel
potential spaces which yields
$$
\| Bu\|_{\operatorname{H}^{-1}_L}\leq c\, |u(x_0)|\leq
c\, \| u\| _{\operatorname{H}^{\frac{1}{2}+\varepsilon}_L}\leq 
c\,\| u\|_{\operatorname{H}^{-1}_L}^{\frac{1}{4}+2 \varepsilon }
\| u\|_{\operatorname{H}^1_L}^{\frac{3}{4}-2 \varepsilon}
\leq\delta \| u\|
 _{\operatorname{H}^{1}_L}+ c_\delta \| u\|_{\operatorname{H}^{-1}_L},
$$
which is valid for any $\delta>0$ by appropriate choice of
the constant $c_\delta>0$. It then follows from a classical perturbation result for
generators of analytic semigroups (see \cite[Theorem 2.4 on page 499]{Ka80}) that $A_\beta$ also
generates such a semigroup on $\operatorname{H}^{-1}_L$ for any
$\beta\in \mathbb{R}$. In \cite{DS88}, Desch and Schappacher show directly
that relatively bounded rank one perturbations of generators of
analytic $c_0$-semigroups preserve the generation property. They also
show that this is not the case for non-analytic semigroups and, in fact,
leads to an alternative characterization of analyticity of a
semigroup. Later in \cite{AR91}, Arendt and Randy show that positive rank
one perturbations of the generator of a holomorphic semigroup
preserve not only the generation property but also positivity. They
approach the problem via resolvent positivity which, for a given linear
operator $C:\operatorname{dom}(C)\subset E\to E$ amounts to the
validity of 
$$
(\lambda-C):\operatorname{dom}(C)\to E\text{ is bijective and }
(\lambda-C)^{-1}\text{ is positive for }\lambda>\omega,
$$
for some $\omega\in \mathbb{R}$ and characterizes positivity of the
corresponding semigroup $T_C(t)$. This clearly requires $E$ to be a
Banach lattice, see \cite{AR91}.
As it is known that $-A$ generates a positive semigroup, we see that the same
remains true for $-A_\beta$ for any $\beta<0$. We are, however,
interested in the parameter range $\beta>0$. It is therefore natural to
ask whether the semigroups remains positive for any parameter value in
this regime.
\begin{prop}
Let $\beta>0$. Then $-A_\beta$ is not resolvent positive and,
consequently, the corresponding semigroup $T_{A_\beta}$ is not
positive.
\end{prop}
\begin{proof}
The space $\operatorname{H}^1$ becomes a Banach lattice if one defines
$$
 u\geq 0\text{ iff }u(x)\geq 0\text{ for }x\in \mathbb{R}.
$$
This follows from the continuity of any $u\in \operatorname{H}^1$. One
can then make $\operatorname{H}^{-1}$ into a Banach lattice as well by
defining
$$
 T\geq 0 \text{ iff } \langle T,u\rangle\geq 0\text{ for every }u\geq 0,
$$
for any given $T\in \operatorname{H}^{-1}$. Next notice that the
resolvent equation for $-A_\beta$, given by
$$
 (s+A_\beta)u=(s+A)u+\beta u(x_0)\delta_0=f,
$$
can be solved for $u$ by observing that
$$
 u=(s+A)^{-1}f-\beta u(x_0)(s+A)^{-1}\delta_0\,.
$$
Then, evaluating the last expression at $x_0\,$, solving for $u(x_0)$, and reinserting 
the result back into the formula above, one obtains
$$
 u=(s+A_\beta)^{-1}f=(s+A)^{-1}f-\beta
 \frac{\bigl[(s+A)^{-1}f\bigr](x_0)}{1+\beta\bigl[(s+A)^{-1}\delta_0\bigr](x_0)}
 (s+A)^{-1}\delta_0,
$$
for any $s>0\,$. More precisely, this holds for
$s\,\in\,\rho(-A)\cap\rho(-A_\beta)\,$, where $\rho(-A)$ and
$\rho(-A_\beta)$ denote the resolvent set of $-A$ and $-A_\beta\,$,
respectively. It will be shown later that
$$
\rho(-A)\cap\rho(-A_\beta) = \rho(-A_\beta)\supset (0,\infty).
$$
Also observe that
$(s+A)^{-1}=\text{``}(s-\partial_{xx})^{-1}\text{''}$ is given by convolution with the kernel
\begin{equation}\label{resKernel}
 G_s(x)=\frac{1}{2\sqrt{s}}\;e^{-\sqrt{s}\, |x|}\;,
\end{equation}
whenever the convolution makes sense. Now take $f=\delta_y$ for $y\in
\mathbb{R}$ to be determined later. Then the solution of $(s+A_\beta)u=f$ is
given by
\begin{equation}\label{green}
 u(x)=G_s(x-y)-\beta \frac{G_S(x_0-y)}{1+\beta G_s(x_0)}G_s(x),\: x\in
 \mathbb{R}\;,
\end{equation}
so that
$$
u(0)=\frac{1}{2\sqrt{s}} e^{-\sqrt{s} |y|}-\beta
\frac{\frac{1}{2\sqrt{s}} e^{-\sqrt{s}|x_0-y|}}{1+\frac{\beta}{2\sqrt{s}}e^{-\sqrt{s}|x_0|}}\;.
$$
Setting $y=x_0$ one gets that
$$
 u(0)=\frac{1}{2\sqrt{s}}e^{-\sqrt{s}|x_0|}\Bigl( 1-\beta
 \frac{e^{\sqrt{s}|x_0|}}{1+\frac{\beta}{2\sqrt{s}} e^{-\sqrt{s}|x_0|}}\Bigr).
$$
As long as $\beta>0$, it follows that $u(0)<0$ for $s\geq s_0>0$ and some $s_0>0$ and, since
$u\in \operatorname{H}^1$, also that $u\not\geq 0$, showing that
$$
 (s+A)^{-1}\delta_{x_0}\not\geq 0\text{ for }s\geq s_0,
$$
and the claim follows since $\delta_{x_0}\geq 0$ in $\operatorname{H}^{-1}$. 
\end{proof}
\begin{rem}
We will analyze the operator $A_{L,\beta}$ ($L<\infty$) later, in which case the
above proposition remains valid. In that case, however, a weaker
positivity property holds up to a critical value $\beta_+>0$.
\end{rem}
By providing a careful spectral analysis of the operator $A_\beta$, it
will be shown below that, not only positivity is lost but, in fact,
\eqref{lineTs} possesses oscillatory solutions.
\begin{rem}
While $-A_{L,\beta}$ ($L\in(0,\infty]$)
generates a holomoprhic semigroup, the solutions of the linear Cauchy
problem are not smooth, since any solution $u$ will clearly have
non-differentiable derivatives whenever $u(x_0)\neq 0$, as follows
from the fact that
$$
 u_{t}-u_{xx}=-\beta u(x_0)\delta _0.
$$
Analyticity of the semigroup entails that
$e^{-tA_{L,\beta}}\bigl( \operatorname{H}^{-1}_L\bigr)\subset
\operatorname{dom} \bigl( A_{L,\beta}^n\bigr)$ for $t>0$ and $n\in
\mathbb{N}$ (see \cite{Pa83,Gol85}). This shows that the singularity of a
solution $e^{-tA_{\beta,L}}u_0$ does not deteriorate as more derivatives
are taken in the sense that
\begin{gather*}
 u\in \operatorname{H}^1_L,\: A_Lu+\beta u(x_0)\delta_0\in
 \operatorname{H}^1_L,\\ A_L \bigl[ A_Lu+\beta u(x_0)\delta_0\bigr]+\beta 
\bigl[ A_Lu+\beta u(x_0)\delta_0\bigr](x_0)\delta_0\in
\operatorname{H}^1_L,\:\dots  
\end{gather*}
and that $u\in \operatorname{H}^m(\mathbb{R}\setminus\{0\})$ for any
$m\in\mathbb{N}$. Thus $u(t,\cdot)\in
\operatorname{C}^\infty(\mathbb{R}\setminus\{0\})$ for any $t>0$ and
for any $u_0\in \operatorname{H}^{-1}$, and, consequently also $u\in
\operatorname{C}^\infty\bigl((0,\infty)\times (\mathbb{R}\setminus\{
0\})\bigr)$.
\end{rem}
For the case $L=\infty$ we obtain the following result on the 
sprectrum of the perturbed operator. Again the case $L<\infty$ will be
considered later. 
However, for finite $L$ the results of our analysis will not be equally 
explicit as for $L=\infty\,$. We shall use the notation $\sigma_p$ and
$\sigma_c$ for the point and continuous spectrum, respectively.
\begin{prop}\label{spec}
There is a critical value $\beta_0=\pi$ such that
$$
\sigma (-A_\beta)=\sigma_{c}(-A_\beta)=(-\infty,0],\:
\sigma_p(-A_\beta)=\emptyset
$$
for $\beta\in[0,\beta_0)$. There is a further critical value
$\beta_1>\beta_0$, whose value can be  determined explicitely as
$\beta_1=\frac{3\pi}{\sqrt{2}}e^{3\pi/4}$ such that for any
$\beta\in(\beta_0,\beta_1)$ the continous spectrum remains unchanged,
i.e.,
$$
\sigma_{c}(-A_\beta)=(-\infty,0],
$$
whereas the point spectrum is genuinely complex 
$$
 \sigma_p(-A_\beta)=\big\{
 \lambda_{1,\beta},\dots,\lambda_{n_\beta,\beta}\} \,\subset\
 \mathbb{C}\backslash\mathbb{R}
$$
and varies with $\beta\in(\beta_0,\beta_1)$. The point spectrum is
never empty for $\beta\in(\beta_0,\beta_1)$ and consists of finitely
many, genuinely complex, isolated eigenvalues that form conjugate
pairs in the interior of the left complex half plane. \\
For $\beta=\beta_1$, a first pair of complex conjugate eigenvalues
reaches the imaginary axis. The pair crosses into
the right complex half plane for $\beta>\beta_1$, yet never reaches
the positive real axis as $\beta\to\infty$.\\
As $\beta$ increases beyond $\beta_1$, additional pairs of complex
conjugate eigenvalues are ejected from the real continuous spectrum
into the left complex half-plane and migrate towards the imaginary
axis, eventually crossing it, pair after pair.
For any finite $\beta>0$, there is only a finite number of conjugate
eigenvalue pairs. None of the pairs ever reunites on the positive real
axis as $\beta\to\infty$.
\end{prop}

\begin{proof}
As previously mentioned, it holds that
$$
 (s+A_\beta)^{-1}f=(s+A)^{-1}f-\beta
 \frac{\bigl[(s+A)^{-1}f\bigr](x_0)}{1+\beta\bigl[(s+A)^{-1}\delta_0\bigr](x_0)}
 (s+A)^{-1}\delta_0.
$$
This shows that, if $s\in
\rho(-A)=\rho(-A_{\beta=0})=\mathbb{C}\setminus(-\infty,0]$, then $s\in
\rho(-A_\beta)$ unless it so happens that
$1+\beta\bigl[(s+A)^{-1}\delta_0\bigr](x_0)=0$. The latter equation is
equivalent to
\begin{equation}\label{poleEq}
 2\sqrt{s}+\beta e^{-x_0\sqrt{s}}=0
\end{equation}
thanks to \eqref{resKernel}. Zeros of this equation in 
$\mathbb{C}\setminus(-\infty,0]$ are
simple poles of the resolvent and, as such, are eigenvalues of
$-A_\beta$. This follows from a classical result found e.g. in
\cite[Theorem 3 on page 229]{Y74}.\\
Before tracing the path of the complex conjugate pairs in more detail,
we provide a qualitative description of the consequences of varying 
$\beta$. The function \eqref{poleEq} is holomorphic in the open domain
$G:=\mathbb{C}\setminus(-\infty,0]$ and can be written as
$f+g$ for  $f=\beta e^{-x_0\sqrt{\cdot}}$ and $g=2\sqrt{\cdot}$. Since
$f$ never vanishes for $\beta\neq0$ and since $g$ is bounded on any
compact subset $K$ of $G$, it is clear that $\lvert g \rvert$ can be
dominated  by $\lvert f \rvert$ on any such $K$ by making
$\lvert \beta \rvert$ sufficiently large. It follows from Rouch\'e's
Theorem that, for any compact $K\subset G$ with smooth boundary, there
exists a $\beta(K)>0$ such that \eqref{poleEq} has no zeros in $K$ for
any $\beta\geq\beta(K)$. An analogous statement for $\beta<0$ clearly
also holds. Thus, increasing or decreasing $\beta$, all eigenvalues
of $A_{\beta}$ exit from any given compact subset of $G$. \\
The solutions of \eqref{poleEq} with $s=-\lvert s \rvert \in
(-\infty,0]$ and, consequently, with $\operatorname{Re}(\sqrt{s})=0$
can immediately be obtained from the validity of
$$
\operatorname{Re}(e^{-x_0\sqrt{s}} )=0\text{ and }2\operatorname{Im}(
\sqrt{s})+\beta \operatorname{Im}( e^{-x_0\sqrt{s}} )=0. 
$$
A separate discussion for positive and negative values of $\beta$
produces all negative real solutions of \eqref{poleEq} for
$\beta\neq0\,$. They are given by
$$
s^+_k=-\frac{(4k+1)^2\pi^2}{4x_0^2}\text{ for }
\beta^+_k=\frac{(4k+1)\pi}{x_0}\text { and }k=0,1,2, \dots
$$
and 
$$
s^-_k=-\frac{(4k+3)^2\pi^2}{4x_0^2}\text{ for
}\beta^-_k=-\frac{(4k+3)\pi}{x_0}\text{ and } k=0,1,2, \dots
$$
In Section \ref{SectionVIE}, the discussion of the Popov
criterion will again reveal this pattern, however along the imaginary axis, where zeros are 
found in an alternating order and they induce a sequence of increasing positive and decreasing negative
values of the parameter $\beta$. 
This reflects the fact, that for increasingly positive values of 
$\beta$ or decreasingly negative values of $\beta$, complex conjugate
solution pairs of \eqref{poleEq} migrate away from the negative real axis,
where they originate at specific locations, towards the imaginary axis.\\ 
We now return to a more precise account of the trajectory traced by
the (genuinely) complex conjugate solutions of equation \eqref{poleEq}
in terms of the parameter $\beta$. To that end, we write
$\sqrt{s}=\alpha+i \gamma$, where $\alpha>0$ and $\gamma\geq 0$. We
can restrict our search in this way since we know that $\alpha -i
\gamma$ is also a solution and since we are interested in solutions
such that $s\not\in(-\infty,0]$, in which case
$\operatorname{Re}(\sqrt{s})\geq 0$. Equation \eqref{poleEq} can then
be rewritten as the system
$$
\begin{cases}
  2 \alpha +\beta e^{-\alpha}\cos(\gamma)=0,&\\
  2 \gamma -\beta e^{-\alpha}\sin(\gamma)=0.&
\end{cases}
$$
We now fix $x_0=1$ in the rest of the calculation. If $0\neq x_0\neq 1$,
the same qualitative behavior is observed for any $\beta\neq 0$ simply with
different numerical values. It follows from the above system that
$$
 \frac{\gamma}{\alpha}=-\tan(\gamma),
$$
and we look for solutions on lines of the form $\gamma/\alpha =m$,
i.e. on lines $\alpha+i m \alpha$ with parameter $\alpha$ where
$m\in[0,\infty]$. In the extreme case when $m=0$, the equation reads
$2\alpha +\beta e^{-\alpha}=0$ and has no solutions for any
$\beta>0$. Next let's fix $\frac{\gamma}{\alpha}=m$, in which case
$$
 \gamma=-\arctan(m)+k\pi
$$
for any integer $k$ such that $\gamma\geq 0$. As $m\in(0,\infty]$, one has
that $-\arctan(m)\in[-\frac{\pi}{2},0)$ and thus
$$
 \gamma=\underset{\gamma_0}{\underbrace{-\arctan(m)+\pi}},
 -\arctan(m)+2\pi,\dots =\gamma_0,\gamma_0+\pi, \gamma_0+2\pi,\dots=
 \gamma_0,\gamma_1,\gamma_2,\dots
$$
where $\gamma_0\in[\frac{\pi}{2},\pi)$. With $\gamma$ in hand and
$\alpha=\gamma/m$ we arrive at
$$
 2 \gamma -\beta e^{-\gamma/m}\sin(\gamma )=0,
$$
from which we see that we need only to consider even $k\geq 0$ due to
$\sin(\gamma_0+\pi)<0$ and the periodicity of $\sin$. This shows that
no solution exists unless $\beta\geq \beta _0>0$, where
$\beta_0=2 \gamma_0e^{\gamma_0/m}/\sin(\gamma_0)$. Notice that, if
$\gamma_0$ is not a solution, then so aren't $\gamma_{2k}$ for $k\geq
1$ since
$$
2 \gamma_{2k}>2 \gamma_0>\beta e^{-\gamma_{0}/m}\sin(\gamma_0)>
\beta e^{-\gamma_{2k}/m}\sin(\gamma_{2k}),
$$
and since $\sin(\gamma_{2k})=\sin(\gamma_0)$. 

The limiting case $m=\infty$ corresponds to looking for real negative
solutions of $\eqref{poleEq}$ and was discussed above separately. In
that case the equation for $\gamma_0$ reduces to $2
\gamma_0=\beta\sin(\gamma_0)$ and it has no solution unless $\beta\geq
2 \gamma_0$, i.e. unless $\beta\geq\pi$ since $\gamma_0=\pi/2$ for
$m=\infty$. We can also conclude that no solution exists for any
$m<\infty$ unless $\beta>\pi$. Next take the case when $m=1$, which
corresponds to looking for purely imaginary eigenvalues. The equation
then reads
$$
 2 \gamma_0=\beta e^{-\gamma_0}\sin(\gamma _0)
$$
for $\gamma_0=3\pi/4$, which requires
\begin{equation}\label{beta1first}
 \beta\geq \beta _1=\frac{3\pi}{\sqrt{2}}e^{3\pi/4}\simeq 70.3134
\end{equation}
for a solution to exist. For $\beta=\beta_1$ only one solution is
found on the line $\gamma=\alpha$. Let us finally consider the case
$m\geq m_*>1$ for some fixed $m_*$. The equation is
$$
 2 \gamma_0=\beta e^{-\gamma_0/m}\sin(\gamma_0)=0,
$$
where $\gamma_0=\pi-\arctan(m)\leq
\pi-\arctan(m_*)=\gamma_*$ satisfies
$\gamma_0\in(\frac{\pi}{2},\gamma_*]$ and
$\gamma_*<\frac{3\pi}{4}$. Under these circumstances, there is no
solution until $\beta$ becomes larger or equal than $2
\gamma_0e^{\gamma_0/m}/\sin(\gamma_0)$ for $m$ fixed. Clearly
$\gamma_0$ can be thought of as a function $\gamma_0(m)$ of $m$, which
is decreasing. It follows that the function
$$
 \Phi(m)=2\frac{\gamma_0(m)e^{\gamma_0(m)/m}}{\sin \bigl(
   \gamma_0(m)\bigr)} 
$$
is also decreasing in $m$. Thus, when considering the equation $\beta
=\Phi(m)\leq \Phi(m_*)$, we see that, for any given
$\beta\leq\Phi(m_*)$, there exists a unique
$m=m(\beta)=\Phi^{-1}(\beta)$. It can be verified that
$\Phi'(1)\approx -254$ and that
$$
 \lim_{m\to\infty}\Phi(m)=\pi.
$$
We conclude that $\sigma(-A_\beta)=(-\infty,0]$ for $\beta<\pi\,$. 
We observe that all negative real solutions are also recovered in this
more detailed discussion of the case of interest ($\beta>0$). Indeed,
for $\beta=\pi$ and $m=\infty$, one has the appearance of the
solution $s^+_1=(i\frac{\pi}{2})^2=-\frac{\pi^2}{4}$ of
\eqref{poleEq} on the negative real axis (note that
$\gamma_0=\frac{\pi}{2}$).
The next solution to appear from $m=\infty$ satisfies
$2(\gamma_0+2\pi)=\beta \sin(\gamma_0)$ yielding $\beta=5\pi$ and the
solution $s^+_2=-\frac{25\pi^2}{4}$ of \eqref{poleEq}.
It follows that more and more solutions of \eqref{poleEq} appear on
the negative real axis (with increasing absolute value) as $\beta$
increases, and, due to the monotonicity properties of the function
$\Phi$, they all migrate towards the imaginary axis along complex
conjugate curves which cross and move beyond it.\\
It remains to verify that the continuous spectrum
persists. This follows from general spectral results which are found
in Kato's book \cite[Theorem 5.35 in Chapter IV]{Ka66}. For the
specific operator of interest here, it is also possible to give a 
direct proof, which also produces generalized eigenfunctions.\\

Consider $\lambda=0$ first and notice
that $G=A-|x|/2$ is, for any $A\in \mathbb{R}$, a fundamental solution
for $-\partial_{xx}$ and therefore it holds that
$$
 -\partial _{xx}G+\beta G(x_0)\delta_0=\bigl( 1+\beta A-|x_0|/2\bigr)
 \delta_0.
$$
Setting $A=|x_0|/2-1/\beta$ one obtains $G\in N(-\partial_{xx}+\beta
\delta_0 \delta_{x_0}^\top)$. While $G\notin \operatorname{H}^{1}$,
it can be approximated by such functions, showing that $\lambda=0$ is
indeed still in the spectrum of $A_\beta$ when
$\beta>0$. For $\lambda =-\alpha ^2$ and $\alpha>0$, one similarly
observes that $\widetilde{G}_\alpha =G_\alpha +Ae^{i \alpha x}+Be^{-i \alpha x}$ is a
fundamental solution of $-\partial_{xx}-\alpha^2$ provided
$$
 G_\alpha (x)=
 \begin{cases}
   \frac{1}{2i \alpha}e^{i \alpha x},&x<0,\\
   \frac{1}{2i \alpha}e^{-i \alpha x},&x\geq 0,
 \end{cases}
$$
since $Ae^{i \alpha x}+Be^{-i \alpha x}\in
N(-\partial_{xx}-\alpha^2)$. One computes that
$$
A_\beta \widetilde{G}_\alpha = \alpha^2 \widetilde{G}_\alpha +\bigl[
1+\beta \widetilde{G}_\alpha (x_0)\bigr] \delta_0.
$$
Since it is always possible to choose $A$ and $B$ so that $\widetilde{G
}_\alpha =-1/\beta$, the claim follows as for $\lambda=0$.
\end{proof}
The asymptotic behavior of the semigroup generated by $-A_\beta$ and the long time 
behavior of solutions to the Cauchy
problem \eqref{lineTs} is the focus of the remainder of this
section. As in the rest of the section we consider the linear
case and, again, postpone the discussion of the case when $L<\infty$
to a later section. Equation \eqref{green} gives an explicit formula for the Green's
function $G_{s,\beta}$ of the operator $s+A_\beta$, so that the Laplace transform
$\mathcal{L}(u)=\hat u$ of a solution $u$ of the linear version of \eqref{lineTs} is given by
\begin{equation}\label{uHat}
 \hat u(s,x)=\frac{1}{2\sqrt{s}} \int e^{-\sqrt{s}|x-y|}u_0(y)\,
 dy-\beta 
 \frac{\int e^{-\sqrt{s}|x_0-y|}u_0(y)\, dy}{2\sqrt{s}+\beta e^{-\sqrt{s}|x_0|}} \;
 \frac{1}{2\sqrt{s}}e^{-\sqrt{s}|x|}\;.
\end{equation}
It holds in particular that
$$
 \hat
 u(s,x_0)=\frac{1}{1+\frac{\beta}{2\sqrt{s}}e^{-\sqrt{s}|x_0|}}
 \frac{1}{2\sqrt{s}}\int e^{-\sqrt{s}|x_0-y|}u_0(y)\, dy.
$$
Also notice the classical fact that 
$\mathcal{L}\bigl( \frac{1}{\sqrt{4\pi t}}e^{-\frac{|x|^2}{4t}}\bigr)(s)=
\frac{1}{2\sqrt{s}} e^{-\sqrt{s}|x|}$ for $s\in \mathbb{C}\setminus
(-\infty,0]$. It is a well-known fact of semigroup theory \cite{ABHN01} that
$$
 \Bigl(\int _0^\infty e^{-st}T_{A_\beta}(t)u_0\,
 dt\Bigr)(x)=(s+A_\beta)^{-1}u_0=
 \int G_{s,\beta}(\cdot,y)u_0(y)\, dy,
$$
so that the kernel $k_{A_\beta}(t)$ of $T_{A_\beta}(t)$ is given by
$$
 k_{A_\beta}(t)(x,y)=\mathcal{L}^{-1}\bigl( G_{\cdot,\beta}(x,y)\bigr).
$$
\begin{prop}
For any $u_0\in \operatorname{L}^1(\mathbb{R})$, so in particular for
any $u_0\in \operatorname{H}^1$, it holds, for any $\beta<\beta_1$, that
$$
 u(t,x_0)=O(\frac{1}{t})\text{ as }t\to\infty,
$$
and that
$$
 u(t,x)=O(\frac{1}{\sqrt{t}}) \text{ as }t\to\infty,
$$
for the corresponding solution of the linear Cauchy problem and for
$x\neq x_0$.
\end{prop}
\begin{proof}
Define
$$
 F(s):=\hat u(s,x_0)=\frac{1}{2\sqrt{s}+\beta e^{-\sqrt{s}|x_0|}}\int
 e^{-\sqrt{s}|x_0-y|}u_0(y)\, dy
$$
and observe that the abscissa of convergence $abs \bigl[ u(\cdot,x_0)\bigr]$ of
$u(\cdot,x_0)$ is 0, i.e. the integral defining the Laplace transform
converges for $\operatorname{Re}(s)>0$,  by the explicit representation of $\hat u$.
Then the well-known inversion theorem for the Laplace transform yields that
$$
 u(t,x_0)=\frac{1}{2\pi i}\int _{\delta-i \infty}^{\delta+i\infty} e^{zt}F(z)\, dz,\: t>0,
$$
where $\delta>0$. Since $\beta<\beta_1$, $F$ is holomorphic in a sector
$$
 [ \, \left| \theta \right|\leq \theta_\beta \,]\,\setminus \{0\}\text{ for }\theta_\beta
 >\frac{\pi}{2}+\gamma \text{ and some }\gamma>0,
$$
as follows from Proposition \ref{spec}. The path of integration can therefore be deformed into
$$
 \Gamma_\varepsilon
 =(-\infty,-\varepsilon)e^{-i(\frac{\pi}{2}+\gamma)}\cup
 \underset{C_\varepsilon}{\underbrace{\Big\{ \varepsilon e^{i\theta}\, \Big |\, \theta\in
 \bigl[-\frac{\pi}{2}-\gamma, \frac{\pi}{2}+\gamma\bigr]\Big\}}}\cup
 (\varepsilon, \infty)e^{i(\frac{\pi}{2}+\gamma)}
$$
without changing the value of the integral. The contribution from the
integration over the circular arc $C_\varepsilon$ is easily seen to vanish as
$\varepsilon\to 0+$, so that we can simply integrate along the rays
$(-\infty,0)e^{\mp i(\frac{\pi}{2}+\gamma)}$. The estimates of the
integrals along both rays can be handled similarly and we therefore
only consider one of them. Let $z=r\, e^{i(\frac{\pi}{2}+\gamma)}$ for
$r\in(0,\infty)$, so that
$$
 \sqrt{z}=\sqrt{r} \Bigl[ \cos(\frac{\pi}{4}+\frac{\gamma }{2})+
 i\sin(\frac{\pi}{4}+\frac{\gamma }{2})\Bigr]
$$
and therefore that
$$
 \big |e^{zt}\big |\simeq e^{-r\gamma t} \,,
$$
since $\cos(\frac{\pi}{2}+\gamma)\simeq -\gamma$. Next notice that\\
$$
 \big |F(z)\big |=\frac{1}{\big |2\sqrt{z}+\beta e^{-\sqrt{z}|x_0|}\big
 |}\Big |\int e^{-\sqrt{z}|x_0-y|}u_0(y)\, dy\Big |\leq C\int
|u_0(y)|\, dy
$$
since $2\sqrt{z}+\beta e^{\sqrt{z}|x_0|}$ has zeros which are a
positive distance away from the path of integration and that $-\sqrt{z}\leq
-\frac{\sqrt{2r}}{2}$. The assumption that $u_0\in
\operatorname{L}^1(\mathbb{R})$ therefore yields that
$$
 \big | u(t,x_0) \big | \leq C\int _0^\infty e^{-rt \gamma}\, dr=\frac{C}{t\gamma },\: t>0.
$$
Notice that the decay is slower, i.e. like $\frac{1}{\sqrt{t}}$ for $\beta=0$, where we have an 
explicit representation of the kernel. It therefore follows from \eqref{uHat} that
$$
 u(t,x)\,=\,O(\frac{1}{\sqrt{t}})\text{ for }x\neq x_0,
$$
as claimed.
\end{proof}
The above proof shows that the decay of solutions varies with
location. It is easily seen that the decay is slowest for $x=0$.


\section{The Linear Dirichlet Problem on an Interval}\label{SecLinear}
We now focus our attention on the case of a finite interval $[-L,L]$
with $L>x_0$ with homogeneous Dirichlet condition
\begin{equation*}
  \begin{cases}
    u_t+A_Lu=-\beta \langle \delta_{x_0},u \rangle\delta_0 &\text{in
    }\operatorname{H}^{-1}_L \text{ for }t>0,\\
    u(0)=u_0,&
  \end{cases}
\end{equation*}
where $\operatorname{H}^{-1}_L$ was defined in the precending section
as the dual of $\operatorname{H}^1_L=\operatorname{H}^1_0 \bigl(
(-L,L)\bigr) $. This captures the problem with homogeneous
Dirichlet data $u(\mp L)=0$ in weak form. 
Using the orthonormal basis of eigenfuctions of $A_L$ that, for
$k=1,2,3,\dots$, is given by   
$$
\varphi_{k,L}=\frac{1}{\sqrt{L}}\sin \bigl(k\pi\,\frac{x+L}{2L}\bigr),  
$$
it is seen that
\begin{equation}\label{sgL}
 e^{-tA_L}=\sum _{k=1}^\infty e^{-t \lambda^2_k}\,\big\langle \cdot,
 \varphi_{k,L} \big\rangle \,\varphi_{k,L},\text{ for }\lambda_{k,L}=\frac{k\pi}{2L},
\end{equation}
and therefore that
\begin{equation}\label{Fundamental_L}
 e^{-tA_L}\delta_0=\sum_{k=0}^\infty
 (-1)^k\exp\Bigl(-t\frac{(2k+1)^2\pi^2}{4L^2} 
 \Bigr)\frac{1}{L}\sin \Bigl((2k+1)\pi\frac{x+L}{2L}\Bigr).
\end{equation}
This series can also be written in terms of classical functions by
reducing the Dirichlet problem to the $4L$-periodic one by extension
\begin{equation}\label{perExtension}
\operatorname{ext}(u)(x)=\begin{cases}
-u(-2L-x),& x\in (-2L,-L),\\u(x),& x\in [-L,L],\\ -u(2L-x),& x\in
(L,2L].\end{cases}
\end{equation}
For the periodic problem it is know that the heat kernel can be
described by the theta function
\begin{equation}\label{thetaFct}
 \theta(z,q)=\sum_{k\in \mathbb{Z}}q^{n^2}e^{2i
   nz}=1+2\sum_{k=1}^\infty q^{n^2}\cos(2nz)
\end{equation}
and the Dirichlet heat kernel takes the form
\begin{equation}\label{dirHeatKernel}
  k_L(t,x)=\frac{1}{4L}\Big\{ \theta \bigl( \frac{\pi
    x}{L},e^{-\frac{\pi^2}{4L}t}\bigr) -\theta \bigl(\frac{\pi
    (x-2L)}{L},e^{-\frac{\pi^2}{4L}t}
  \bigr) \Big\}.
\end{equation}
Using the variation of constant formula for the new
operator $A_L$ and evaluating it at $x=x_0$, the initial boundary
value problem is therefore reduced to the integral equation
\begin{equation}\label{linDirVie}
  y(t)=\bigl( e^{-tA_L}u_0\bigr) (x_0)-\beta \int _0^t
  y(\tau)k_L(t-\tau,x_0)\, d\tau,
\end{equation}
where $y$ plays the role of $u(\cdot,x_0)$. As is the case on the
line, the problem can actually be solved by Laplace transform
methods. Reproducing the calculation of the previous section, one
arrives at
$$
 \hat u(s)=(s+A_L)^{-1}u_0-\beta \hat u(s,x_0)(s+A_L)^{-1}\delta_0,
$$
from which one deduces that
$$
 \hat u(s,x_0)=\frac{\bigl( (s+A_L)^{-1}u_0\bigr)(x_0)}{1+\beta \bigl(
   (s+A_L)^{-1}\delta _0\bigr)(x_0)}.
$$
The Green's function $G^L_s$ of the Dirichlet problem which is given by
$(s+A_L)^{-1}\delta _y$ can be obtained explicitly by
computing the general solution of the ODE $sz-z''=\delta_y,\:
y\in(-L,L)$, given by
$$
 z(x)=\sinh\bigl((y-x)\sqrt{s}\bigr)H(x-y)+Ae^{-x\sqrt{s}}+Be^{x\sqrt{s}},
$$
where $H$ is the Heaviside function, and determining the
coefficients $A,B$ by imposing the boundary conditions $z(\pm L)=0$.
Doing so yields
\begin{equation}\label{greensL}
G^L_s(x,y)=
\begin{cases}
  \frac{\sinh \bigl( \sqrt{s}(L-y)\bigr) \sinh \bigl(
    \sqrt{s}(L+x)\bigr)}{2\sqrt{s}\cosh(\sqrt{s}L) \sinh(\sqrt{s}L)},&
  -L\leq x\leq y,\\
  \frac{\sinh \bigl( \sqrt{s}(L+y)\bigr) \sinh \bigl(
    \sqrt{s}(L-x)\bigr)}{2\sqrt{s}\cosh(\sqrt{s}L) \sinh(\sqrt{s}L)},&
  y\leq x\leq L,  
\end{cases}
\end{equation}
for $y\in (-L,L)$. From this, it is seen that, as $L\to\infty$,
$$
G^L_s(x,y)\longrightarrow G_s^\infty (x,y)=
\frac{1}{2\sqrt{s}}e^{-\sqrt{s}|x-y|}=G_s(x-y)
$$
for $y\in (-\infty,\infty)$. The resolvent of
$A_{L,\beta}=A_L+\beta \delta_0 \delta_{x_0}^\top$ is given by
$$
 (s+A_{L,\beta})^{-1}\bullet=(s+A_L)^{-1}\bullet-\beta \frac{\bigl[
   (s+A_L)^{-1}\bullet\bigr](x_0)}{1+\beta\bigl[
   (s+A_L)^{-1}\delta_0\bigr](x_0)}(s+A_L)^{-1}\delta_0 \,,
$$
where $\bullet$ is a stand-in for the argument and has kernel
\begin{equation}\label{Lkernel}
  G^{L,\beta}_s(x,y)=G^L_s(x,y)-\beta \frac{G_s^L(x_0,y)}{1+\beta
    G_s^L(x_0,0)}G_L(x,0).
\end{equation}
The operator $A_{L,0}$ has positive spectrum and a principal
eigenvalue with positive eigenfunction. This remains true for the
non-selfadjoint operator $A_{L,\beta}$ up to a critical value
$\beta_+>0$.
\begin{prop}
The operator $-A_{L,\beta}$ generates an analytic
$c_0$-semigroup. This semigroup is positive if and only if $\beta\leq
0$. There is, however, a value $\beta_+>0$, below which the first
eigenfunctions of the operator and of the adjoint operator both remain
positive. In the parameter range $(0,\beta_+)$, the semigroup is
individually eventually positive in the sense of \cite{DGK16,DG18}.
\end{prop}
\begin{proof}
We compute the first eigenvalue of the operator $A_{L,\beta}$ by
observing that its eigenfunction $\varphi$ is smooth away from $x=0$. We can
therefore assume that
$$
 \varphi (x)=A_\pm \sin(\lambda x)+B_\pm \cos( \lambda x),\: \pm x>0.
$$
The function $\varphi$ needs to satisfy the boundary conditions
$\varphi (\pm L)=0$, is continuous in the origin $\varphi(0-)=\varphi
(0+)$, where it enjoys the jump condition
$$
 -\varphi_x(0-)+\varphi_x(0+)=\beta \varphi(x_0),
$$
in order for the eigenvalue equation $-\varphi_{xx}+\beta
\varphi(x_0)\delta_0=\lambda^2 \varphi$ to hold. Continuity across the origin
implies that $B_-=B_+$, whereas the other conditions lead to the
system
$$
\begin{bmatrix} -\sin(\lambda L)&0&\cos(\lambda L)\\
  0&\sin(\lambda L)&\cos(\lambda L)\\ -\lambda &\lambda-\beta
  \sin(\lambda x_0)& -\beta \cos(\lambda x_0)\end{bmatrix}\, 
  \begin{bmatrix} A_- \\A_+\\B_-\end{bmatrix}=
  \begin{bmatrix} 0 \\0\\0\end{bmatrix}
$$
A necessary condition for the existence of nontrivial solutions is
given by the vanishing of the determinant which yields the equation
\begin{equation}\label{eigEq}
 \sin(\lambda L)\Bigl\{ 2\cos(\lambda L)+\beta
 \frac{\sin\bigl(\lambda(L-x_0)\bigr)}{\lambda}\Big\}=0.
\end{equation}
For $\beta=0$, the first zero is $\lambda^1_{L,0}=\frac{\pi}{2L}$ and
yields the eigenvalue $\mu^1_{L,0}=\bigl(\lambda^1_0\bigr)^2
=\frac{\pi^2}{4L^2}$. The associated eigenfunction $\varphi^1_{L,0}$
is given by
$\varphi^1_{L,0}(x)=\frac{1}{\sqrt{L}}\sin(\pi\frac{x+L}{2L})$. Continuous
dependence on $\beta$ of the equation \eqref{eigEq}, shows that the
first eigenvalue $\mu^1_{L,\beta}$ will be located near $\mu^1_{L,0}$
and that the associated eigenfunction $\varphi^1_{L,\beta}$ will be
close to $\varphi ^1_{L,0}$. Due to the heat sink at $x=0$, it will
develop a kink, which, with increasing $\beta$, will eventually make
the eigenfunction negative in and near $x=0$. The eigenfunction
$\varphi^1_{L,\beta}$ is depicted in Figure \ref{1stEfct} for several
values of the parameter $\beta$. The eigenfunctions are obtained
numerically by a spectral discretization that is presented in Section
\ref{SecNumerical}.
\begin{figure}
  \centering
  \includegraphics[scale=.75]{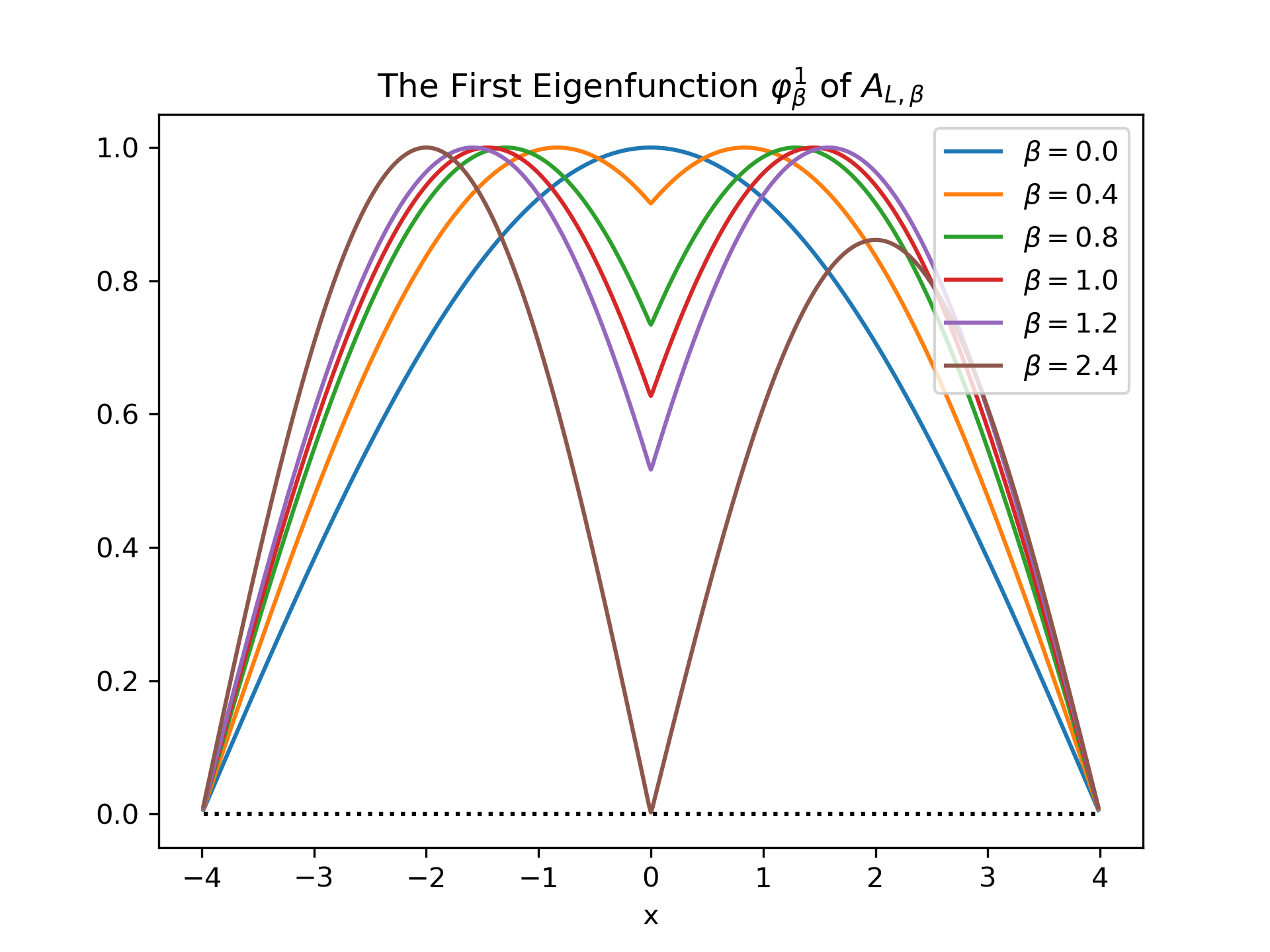}
  \caption{The first eigenfunction of $A_{L,\beta}$ as the parameter $\beta$
    increases for $L=4$ and $x_0=1$.}
  \label{1stEfct}
\end{figure}

Next we observe that the operator $A'_{L,\beta}$ adjoint to $A_{L,\beta}=A_L+\beta
\delta_0 \delta_{x_0}^\top$ is given by $A_L+\beta
\delta_{x_0} \delta_0^\top$ as immediately follows from
$$
 \langle A_{L,\beta}u,v \rangle =a_L(u,v)+\beta u(x_0)v(0)=
 a_L(v,u)+\beta v(0)u(x_0)=\langle u, A_{L,\beta}'v \rangle
 ,\: u,v\in\operatorname{H}^1_L. 
$$
These operators share their eigenvalues and, if we denote their
eigenfunctions by $\varphi ^k_{L,\beta}$, for $A_{L,\beta}$, and by
$\psi ^k_{L,\beta}$, for the adjoint operator, we obtain the spectral
resolution given by
$$
A_{L,\beta}=\sum_{k=1}^\infty \mu^k_{L,\beta}\big\langle
\psi^k_{L,\beta} ,\cdot \big\rangle \varphi^k_{L,\beta},
$$
and the associated semigroup is explicitly given by
\begin{align*}
e^{-tA_{L,\beta}}&=\sum_{k=1}^\infty \exp \bigl( -t\mu^k_{L,\beta}
\bigr) \big\langle \psi^k_{L,\beta} ,\cdot \big\rangle\varphi^k_{L,\beta}\\&=
\exp \bigl( -t\mu^1_{L,\beta}\bigr) \varphi^1_{L,\beta}\Big\{
\big\langle \psi^1_{L,\beta},\cdot \big\rangle  +
\sum_{k=2}^\infty \exp \Bigl( -t\bigl[\mu^k_{L,\beta}
-\mu^1_{L,\beta}\bigr] \Bigr)\big\langle \psi^k_{L,\beta} ,\cdot
  \big\rangle \frac{\varphi^k_{L,\beta}}{\varphi ^1_{L,\beta}}\Big\},
\end{align*}
where the second equality holds provided $\varphi^1_{L,\beta} >0$.
Notice that, in that case, the quotients
${\varphi^k_{L,\beta}}/{\varphi^1_{L,\beta}}$
are well defined up to the boundary thanks to L'H\^opital's rule and to
$\bigl(\varphi^1_{L,\beta}\bigr)'(\pm L)\neq 0$. The latter is seen either by
using the maximum principle or by direct inspection of the form of the
eigenfunctions.
Now the first eigenfunction $\psi^1_{L,\beta}$ of $A'_{L,\beta}$ is
also positive for small $\beta$. This can be seen either by a direct
computation similar to the one we preformed above for
$\varphi^1_{L,\beta}$ or by observing that the adjoint operator has the
same structure as the original one. It follows that, given any
positive initial datum $u_0\in \operatorname{H}^1_L$, or even in
$\operatorname{H}^{-1}_L$, one necessarily has that $\big \langle u_0,
\psi ^1_{L,\beta}\big \rangle>0$ 
and the corresponding solution will eventually be positive in
$(-L,L)$. The actual time at which this happens will depend on $u_0$,
leading to individual eventual positivity. This positivity holds as long as
both $\varphi ^1_{L,\beta}$ and $\psi_{L,\beta} ^1$ are positive, which
is the case for $\beta<\beta_+$ and some $\beta_+>0$. Figure
\ref{efctT} depicts the first eigenfunction of $A_{L,\beta}'$ for
several values of $\beta$.
\end{proof}
\begin{figure}
  \centering
  \includegraphics[scale=.75]{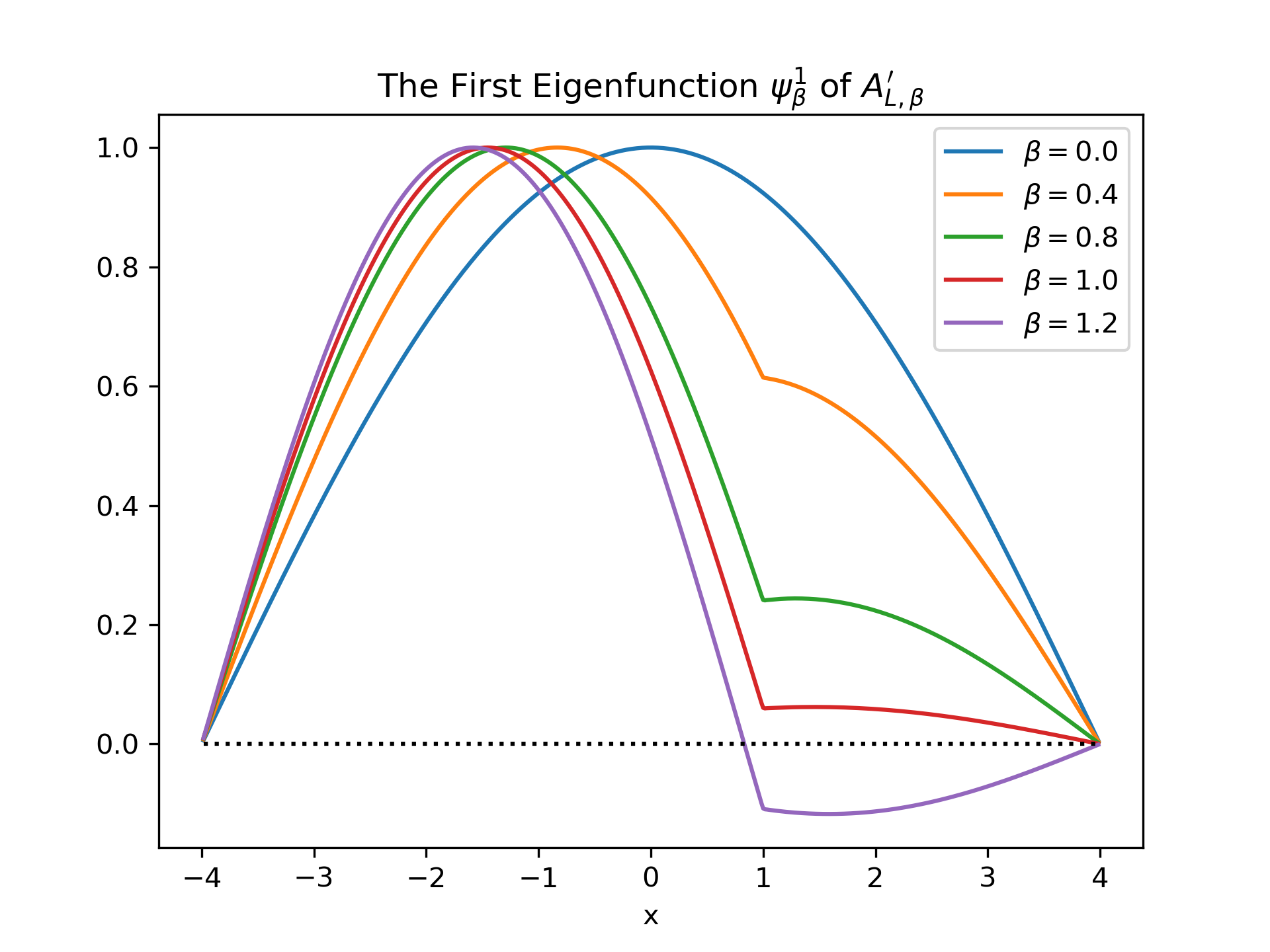}
  \caption{The first eigenfunction of $A_{L,\beta}'$ as the parameter $\beta$
    increases for $L=4$ and $x_0=1$.}
  \label{efctT}
\end{figure}
\begin{rem}
Notice that equation \eqref{eigEq}, which determines the eigenvalues
of $A_{L,\beta}$ shows that ``half'' of the eigenvalues, those arising
as zeros of $\sin(\lambda L)$, do not in fact depend on $\beta$ at
all. In the limit as $L\to\infty$ they contribute to the continuous
spectrum of $A_\beta$, which we already observed remains unchanged as
$\beta$ increases.
\end{rem}
The eigenvalues of $A_{L,\beta}$ generated by the zeros of the second
factor in \eqref{eigEq} are partly responsible for the onset of
complex spectrum, but mostly contribute to the real spectrum.
\begin{prop}
The zeros of the second factor of \eqref{eigEq} located in
$\mathbb{C}\setminus(-\infty,0]$ coincide with those of 
$1+\beta G^L_s(x_0,0)$ appearing in \eqref{Lkernel}. For any finite
$\beta>0$ and any $L>x_0>0$ large enough, there is only a finite
number of them and they are close to the zeros of $1+\beta
G_s(x_0)$.
\end{prop}
\begin{proof}
First notice that the function $\cosh(\lambda L)$ only vanishes for
$\lambda = (\frac{\pi}{2L}+\frac{\pi}{L}k) i$, $k\in \mathbb{Z}$. This
means that, when looking for zeros of
$$
 K_L(\lambda)=1+\beta\frac{\sinh\bigl(\lambda (L-x_0)\bigr)}{2
   \lambda\cosh(\lambda L)}
$$
leading to eigenvalues $\lambda^2\in \mathbb{C} \setminus
(-\infty,0]$, we can safely consider the equation $J_L(\lambda)=0$
instead, where
$$
 J_L(\lambda)=2\cosh( \lambda L)+\beta \frac{\sinh\bigl(
 \lambda(L-x_0)\bigr) }{\lambda}=2\cosh( \lambda L)K_L(\lambda),
$$
when looking for eigenvalues with non-trivial imaginary part.
Zeros of $J_L$ in $\mathbb{C}\setminus(-\infty,0]$ therefore account
for all and any non-real eigenvalues of $A_{L,\beta}$. We already know
that the second factor in \eqref{eigEq} is the only possible source of
non-real eigenvalues of $A_{L,\beta}$, as well. We use the notation
$$
H_L(\lambda)=2\cos( \lambda L)+\beta \frac{\sin\big(
  \lambda(L-x_0)\bigr) }{\lambda}
$$
for that factor. Direct computation shows that, for these functions,
it holds that
$$
 H_L(\overline{\lambda})=\overline{H_L(\lambda)}, \:
 J_L(\overline{\lambda})=\overline{J_L(\lambda)},\: \lambda\in
 \mathbb{C},
$$
and that $H_L(-\lambda)=H_L(\lambda)$, $J_L(-\lambda)=J_L(\lambda)$.
This shows, unsurprisingly, that complex zeros come in complex
conjugate pairs. Well-known trigonometric (or hyperbolic) identities
show that
$$
 J_L(\lambda)=H_L(i \lambda),\: J_L(i
 \lambda)=H_L(-\lambda)=H_L(\lambda).
$$
It follows that
\begin{equation}\label{HJrelation}
 J_L(\alpha +im \alpha)=H_L(i \alpha -m \alpha ) = \overline{H_L(-i
   \alpha -m \alpha)}=\overline{ H_L(m \alpha +i \alpha)},\: \alpha\in
 [0,\infty).
\end{equation}
Varying $m\in[0,\infty)$ allows for the search of complex zeros
on rays emanating from the origin covering the first quadrant (with the
exception of the positive imaginary axis), and leads to the
determination of all complex eigenvalues in the upper-half plane. In
view of the stated properties of the functions of interest, this is
sufficient in order to locate all eigenvalues in $\mathbb{C} \setminus (-\infty,0]$. Identity
\eqref{HJrelation} readily implies that eigenvalues on $i(0,\infty)$,
which are obtained searching for zeros with $m=1$, correspond to the shared
zeros of $J_L$ and $K_L$ on the ray $(1+i)(0,\infty)$. For the
other rays in the first quadrant, i.e. for $m\in (0,\infty) \setminus
\{ 1\}$, zeros of $J_L$ on $(1,m\, i)(0,\infty)$ correspond to zeros of
$H_L$ on $(m,i)(0,\infty)$ and vice-versa. We conclude that, while the
equations for the zeros of $J_L$ and of $H_L$ are not equivalent,
these two functions have identical zero sets in the open first quadrant.\\
Next observe that $K_L(\lambda)=1+\beta
G_\lambda^L(x_0,0)$ and that $K_L \longrightarrow 1+\beta
G_\lambda(x_0) =:K(\lambda)$  as $L\to\infty$,
uniformly in subsets which are a positive distance away from
$\mathbb{C}\setminus (-\infty,0]$. Uniform convergence holds also for
the first derivative of these functions. The zeros with non-trivial
imaginary part of the limiting function have been fully characterized
in Proposition \ref{spec}. It therefore follows that, for any fixed
$\beta>0$ and for $L$ large enough, the zero set of $K_L$ in the
interior of the first quadrant is close to that of $K$, which was fully
understood in Proposition \ref{spec}. This is true due to the fact
that these zeros are non-degenerate, a fact that will follow from a
later more detailed discussion (see the proof of Proposition
\ref{betaCritical-Lcase} below). To be more precise, in the limit,
the countable simple eigenvalues on the imaginary axis do not accuumulate and
are non-degenerate as they are generated by the zeros of the function
$z_\infty$ appearing in \eqref{z-fct}. As the parameter $\beta$ is
dialed back down, these zeros move on smooth curves that do not cross
until they reach the real line for $L=\infty$. Due to the uniform
convergence mentioned above the same has to remain true away from the
real line for any large $L$, as well.  
\end{proof}
\begin{rem}
While it is not possible to carry out calculations as explicitly as
it was the case for $A_\beta$, i.e. for the full line, the fast
convergence of the resolvent/kernel as $L\to\infty$ allows one to
conclude that the zeros of $1+\beta G_s^L(x_0,0)$ located in the
interior of the first quadrant are very close to those of $1+\beta
G_s(x_0,0)$ already for modest values of $L$ (even for $L=2$ and
$x_0=1$). In particular, the complex eigenvalues of $A_{L,\beta}$
(situated outside a neighborhood of the origin, and they all are) do
behave in a manner very close to those of $A_\beta$. The eigenvalues
on the negative real axis essentially only contribute to the continuous spectrum
in the limit. This is even true for discretizations of $A_{L,\beta}$
for the first few crossings, which can be captured with relatively few
grid points. We refer to Figure \ref{crossing} for a plot of the curve
traced by the first pair of complex conjugate eigenvalues parametrized
by $\beta$ from the moment they leave the real line (for $L=4,8,16$
and $x_0=1$) and to the last section for details about the numerical discretization
used in the computations. Notice that the imaginary axis is crossed at
$\beta\simeq 70\simeq \beta_1$, regardless of the value of $L$. 
\end{rem}

\begin{prop}\label{betaCritical-Lcase}
For $L$ large enough, there are critical values $\beta_{0,L}>0$ and
$\beta_{1,L}>\beta _{0,L}$, so that at $\beta_{0,L}>0$ genuinely complex eigenvalues 
appear in the spectrum of $A_{\beta,L}$ and so that at $\beta_{1,L}$ 
a complex conjugate eigenvalue pair crosses the imaginary axis.
\end{prop}

\begin{proof}
This follows again from the uniform convergence of corresponding
functions determining the non-real eigenvalues of $A_{L,\beta}$ and
the complete knowledge of the limiting case $L=\infty$. 
\end{proof}

\begin{rem}
The parameter
value at which pairs of real eigenvalues merge and become complex
conjugate with non-trivial imaginary part appears to have a
non-straightfoward relation to the parameter $L$. As for the parameter
$\beta_{1,L}$, more can be said analyzing the equation $1+\beta
G^L_s(x_0,0)$ more closely. To shorten the formul\ae , we use the
notation $\operatorname{c}$, $\operatorname{s}$, 
$\operatorname{ch}$, $\operatorname{sh}$, and $\operatorname{th}$ for
the functions $\cos$, $\sin$, $\cosh$, $\sinh$, and $\tanh$,
respectively. Morever $L_0$ will denote $L-x_0$. As observed earlier,
looking for eigenvalues on the imaginary axis amounts to looking for zeros
of the form $\sqrt{s}=\lambda=\alpha +i\alpha$,
$\alpha>0$. Decomposing the function
$G^L_s(x_0,0)=\frac{\operatorname{sh}(L_0 \lambda)}{2 \lambda
  \operatorname{ch}(L \lambda)}$ into real and complex parts yields
\begin{multline*}
 \operatorname{Re}\bigl( G^L_s(x_0,0)\bigr)=\frac{1}{4
 \alpha}\frac{e^{-\alpha}+e^{(1-2L)\alpha}}{1+e^{-2L \alpha}}
 \frac{1}{\operatorname{c}^2(L \alpha)+\operatorname{th}^2(L
 \alpha)\operatorname{s}^2(L \alpha)}\cdot\\
 \Big\{ \operatorname{c}(L \alpha)
 \operatorname{th}(L_0 \alpha) \operatorname{c}(L_0 \alpha)+
 \operatorname{s}(L_0 \alpha) \operatorname{c}(L
 \alpha)+\operatorname{s}(L_0 \alpha) \operatorname{th}(L \alpha)
 \operatorname{s}(L \alpha)-\operatorname{th}(L_0 \alpha)
 \operatorname{c}(L_0 \alpha) \operatorname{s}(L \alpha)
 \operatorname{th}(L \alpha)\Big\},
\end{multline*}
and
\begin{multline*}
 \operatorname{Im}\bigl( G^L_s(x_0,0)\bigr)=\frac{1}{4
 \alpha}\frac{e^{-\alpha}+e^{(1-2L)\alpha}}{1+e^{-2L \alpha}}
 \frac{1}{\operatorname{c}^2(L \alpha)+\operatorname{th}^2(L
 \alpha)\operatorname{s}^2(L \alpha)}\cdot\\
 \Big\{\operatorname{s}(L_0 \alpha) \operatorname{c}(L
 \alpha)-\operatorname{c}(L_0 \alpha)\operatorname{th}(L_0 \alpha)
 \operatorname{c}(L \alpha)-\operatorname{th}(L_0 \alpha)
 \operatorname{c}(L_0 \alpha) \operatorname{s}(L \alpha)
 \operatorname{th}(L \alpha)-\operatorname{s}(L_0 \alpha)
 \operatorname{th}(L \alpha)\operatorname{s}(L \alpha)\Big\}.
\end{multline*}
Since the term $\operatorname{c}^2(L \alpha)+\operatorname{th}^2(L
\alpha)\operatorname{s}^2(L \alpha)$ never vanishes as follows from
the fact that it takes the value 1 in $\alpha=0$ and that zeros would 
otherwise ($\alpha\neq 0$) satisfy $\tanh^2(L \alpha)=-\cot^2(L
\alpha)$, the imaginary part of $1+\beta G^L_s(x_0,0)$ can only vanish
if the term in the curly brackets vanishes, equivalently iff
$$
 z_L(\alpha)=\operatorname{s}(L_0 \alpha)\bigl[ \operatorname{c}(L
 \alpha)-\operatorname{th}(L \alpha) \operatorname{s}(L \alpha)
 \bigr]- \operatorname{th}(L_0 \alpha) \operatorname{c}(L_0 \alpha)
 \bigl[ \operatorname{s}(L \alpha)+\operatorname{th}(L \alpha)
 \operatorname{s}(L \alpha)\bigr] =0.
$$
Now, for $\alpha\geq \alpha_0>0$ and $L>>1$, using the trigonometric
addition formul{\ae} to expand the terms with argument $L_0=L-x_0$ and
observing that $\tanh(L_0 \alpha)\simeq 1\simeq \tanh(L \alpha)$ in
this regime, it can be verified that
\begin{equation}\label{z-fct}
 z_L(\alpha)\simeq -\cos( \alpha x_0)-\sin (\alpha
 x_0)=z_\infty(\alpha).
\end{equation}
The convergence is quite fast as can be seen in Figure
\ref{L2infty}. For the first zero of interest, the curves are almost
identical even for small $L$, and even more so for subsequent
zeros. Once the zeros of the imaginary part are known (the first one is
the one we care about), the corresponding value of $\beta$ can be
recovered by setting
$$
 \operatorname{Re}\bigl( 1+\beta G^L_s(x_0,0)\bigr)=0, 
$$
and solving for $\beta$. A similar asymptotic analysis of the behavior
of the real part of $G^L_s(x_0,0)$ as that performed for the imaginary part,
reveals that
$$
 \operatorname{Re}\bigl( G^L_s(x_0,0) \bigr) \simeq
 \frac{e^{-\alpha}}{4 \alpha} \bigl[ \cos( \alpha x_0) -\sin( \alpha
   x_0)\bigr] \text{ for }L\simeq \infty.
$$
Again the convergence is very fast and the above approximation
delivers a good estimate of the critical value for moderately sized
$L$. It is interesting to observe that $\beta _{1,L}$ does not exhibit
monotone behavior in $L$, 
see Figure \ref{L2infty}.
\end{rem}

\begin{figure}
  \centering
  \includegraphics[scale=0.5]{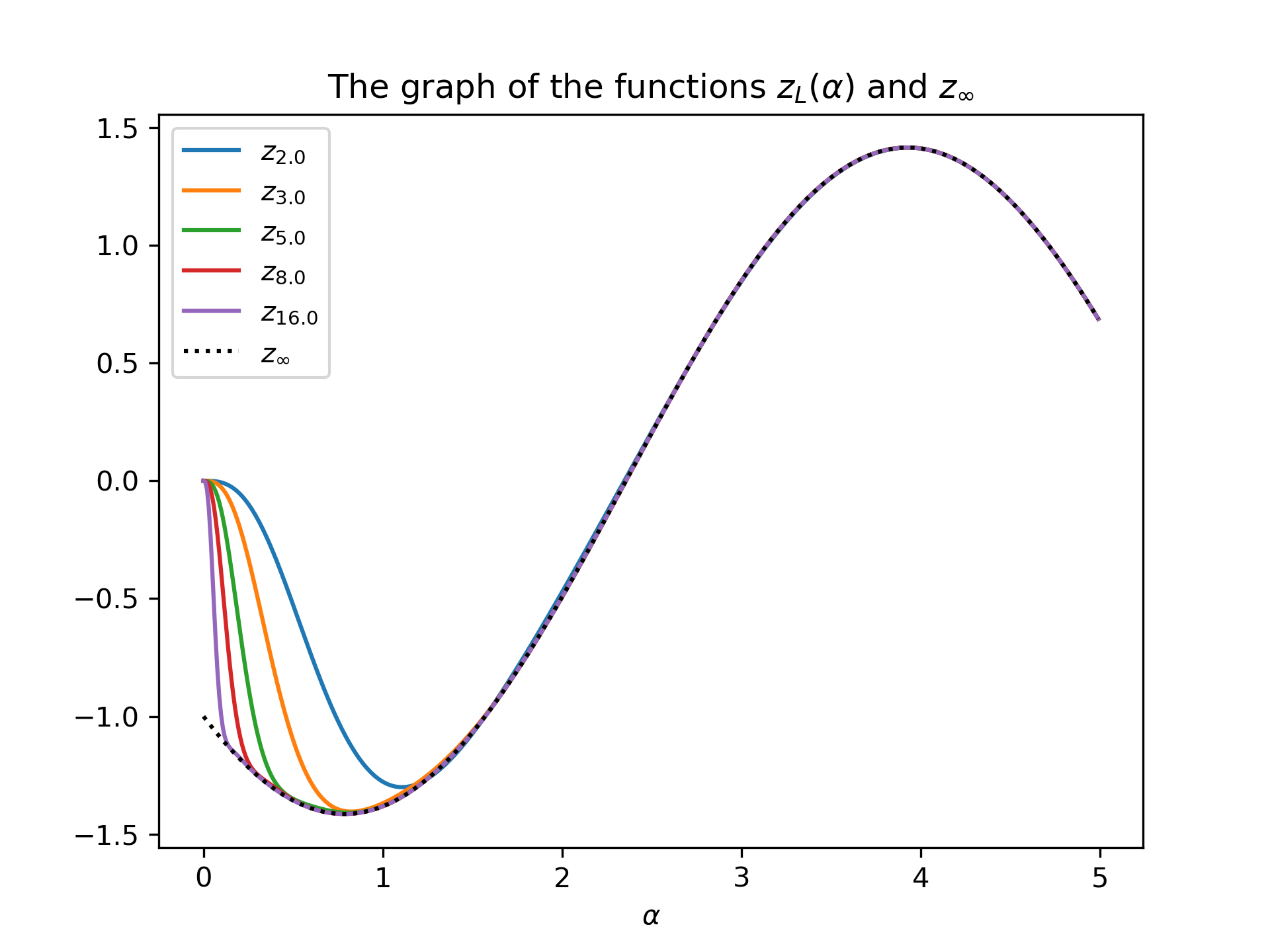}
  \includegraphics[scale=0.5]{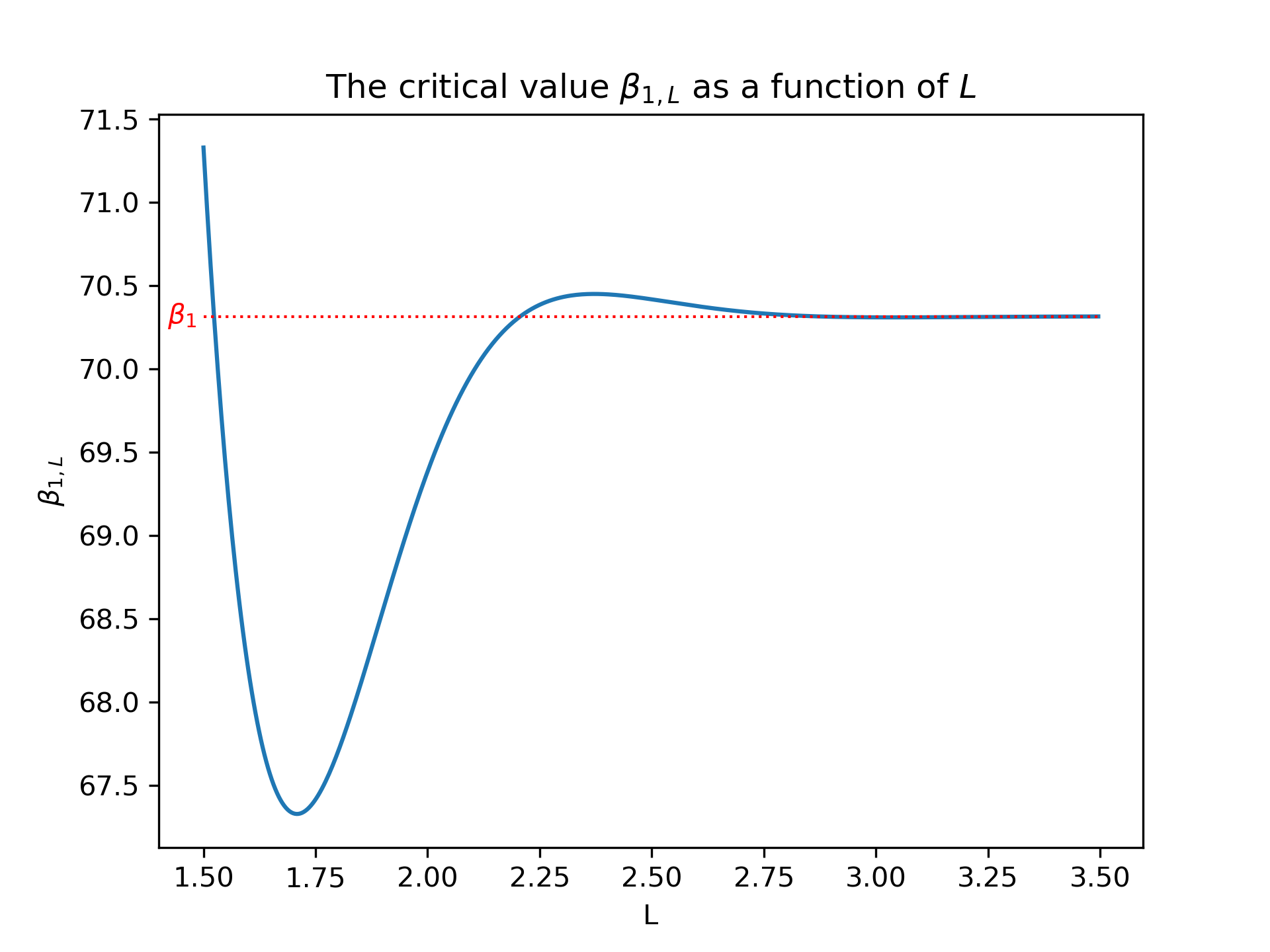}
  \caption{The behavior of the function $z_L(\alpha)$ as $L$ grows for $x_0=1$.}
  \label{L2infty}
\end{figure}

\begin{rem}\label{L_finite_Rem}
The behavior of the real spectrum for $L<\infty$ as a function of $L$
is harder to pinpoint analytically. ``Half'' the eigenvalues do not
depend on $\beta$ and just ``fill'' the negative real axis in the limit as
$L\to\infty$. As for the other half, an increasingly small fraction
(as $L$ increases) of them merge and become complex as $\beta$ gets
larger. This we know since only a finite number of non-real simple
eigenvalues appears with increasing $\beta$ (and for large $L$). A
numerical calculation of the zeros of the second term in the explicit equation
\eqref{eigEq} confirms this theoretical prediction and can be seen in
Figure \ref{merging}.
\end{rem}
\begin{figure}
  \centering
  \includegraphics[scale=.75]{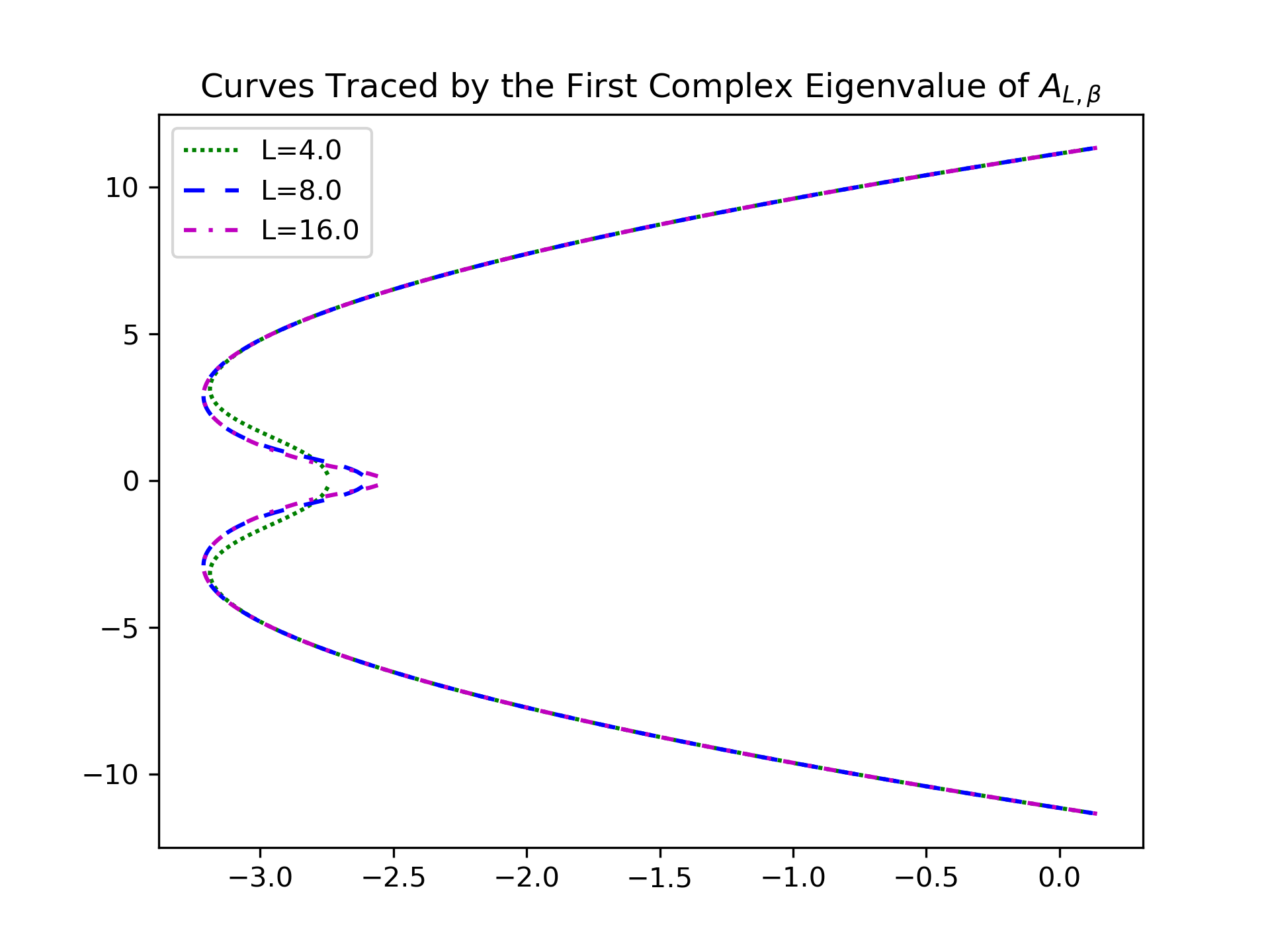}
  \caption{The curve traced by the first pair of complex eigenvalues
    of $A_\beta^L$ for interval half-lenghts $L=4,8,16$ and $\beta\in[3.0,73.0)$.}
  \label{crossing}
\end{figure}

\begin{figure}
  \centering
  \includegraphics[scale=.5]{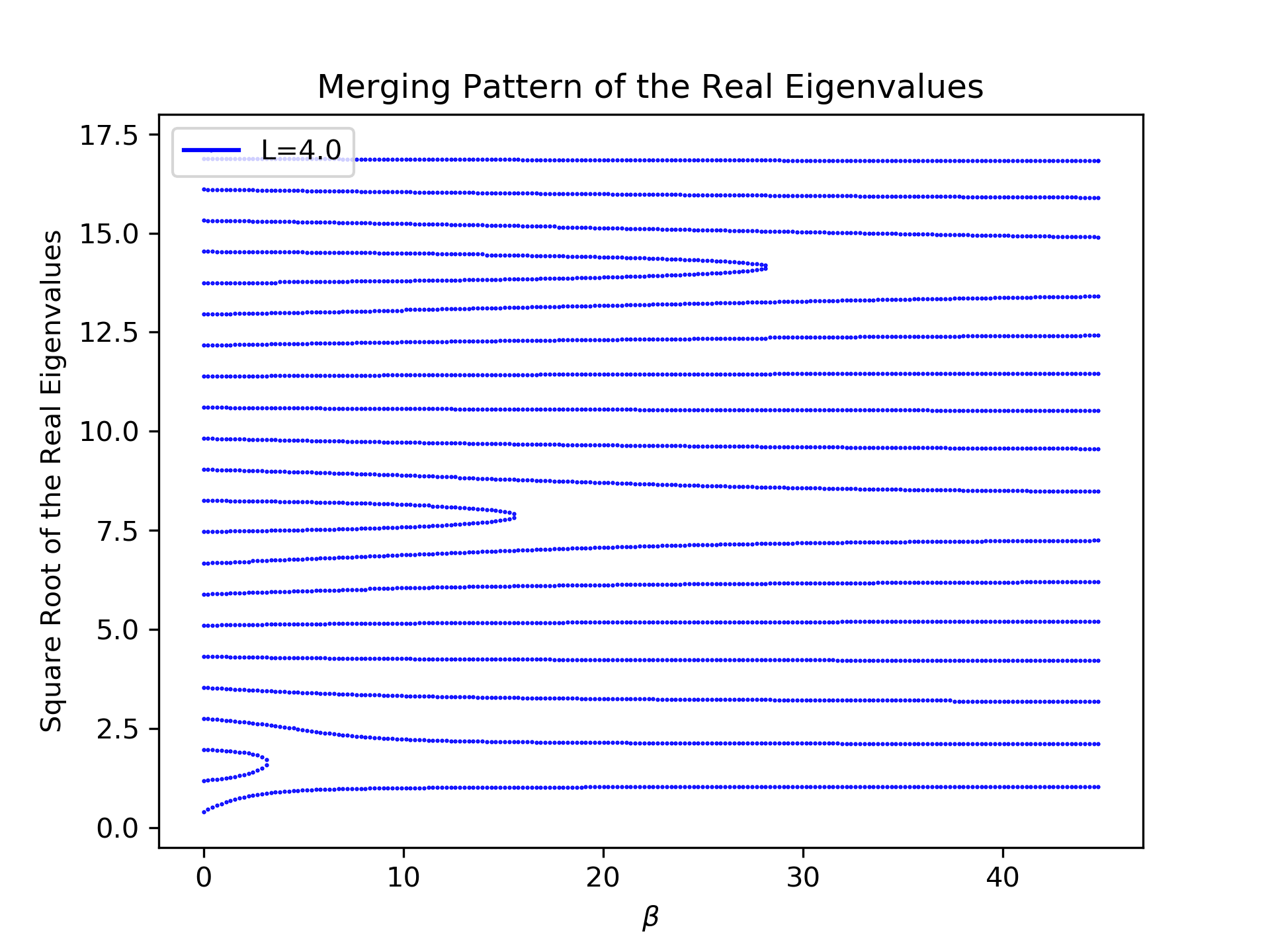}
  \includegraphics[scale=.5]{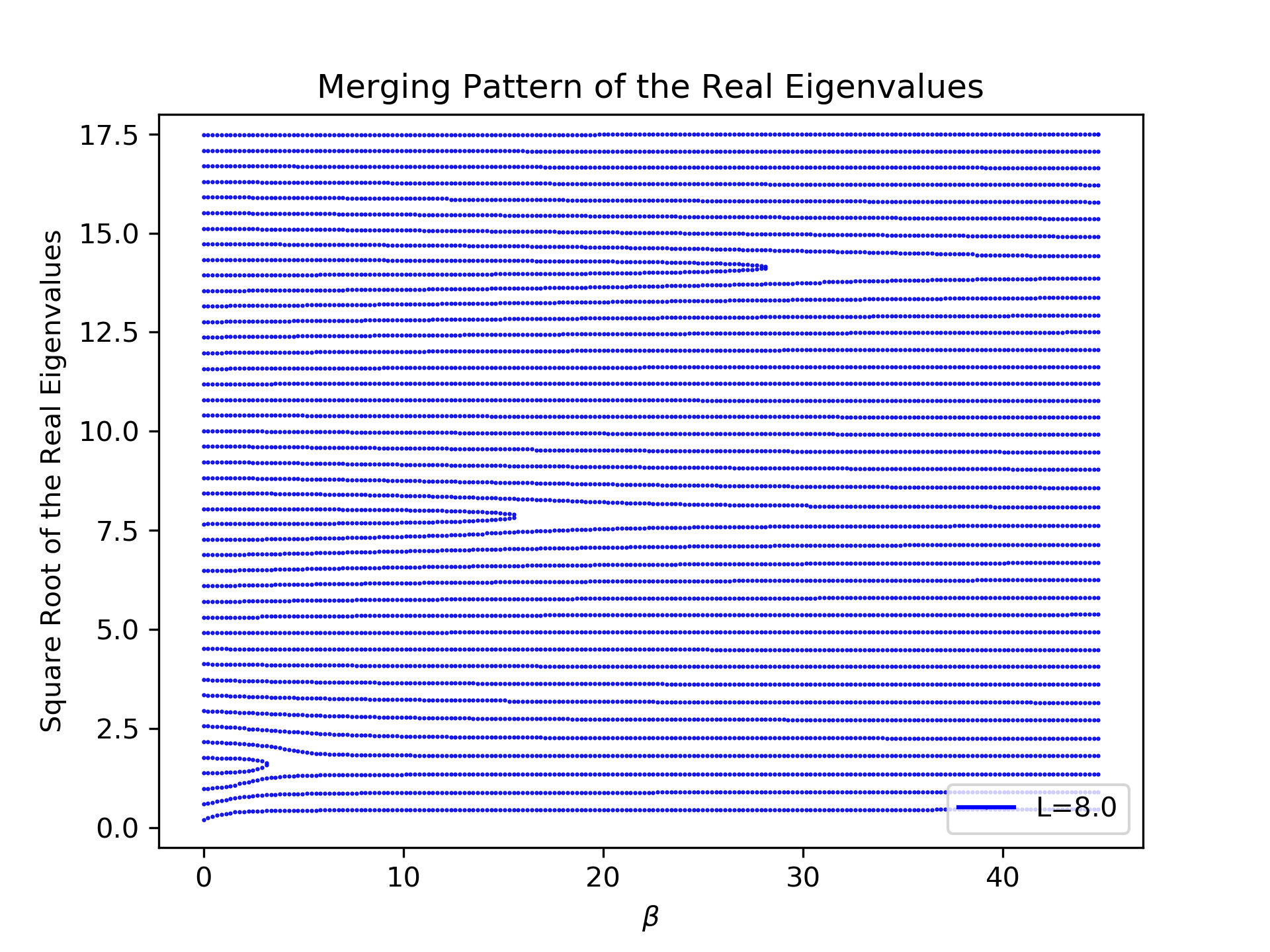} 
  \caption{The merging pattern of ``half'' of the real eigenvalues of
    $A_{\beta,L}\label{fig:re}$ observed when $\beta$ is 
    increased for different values of $L$: left $L=4$, right $L=8$.}  
  \label{merging}
\end{figure}
\section{The Nonlinear Equation}
Using the analytic semigroup generated by $A$ on
$\operatorname{H}^{-1}$, solutions of the nonlinear equation
\eqref{lineTs} can be looked for as fixed points of the equation
\begin{equation}\label{vcfEq}
  u(t)=e^{-tA}u_0-\int _0^t f \bigl(\beta u(t,x_0)\bigr)e^{-(t-\tau)A}\delta_0\,
  d\tau.
\end{equation}
We consider $u_0\in \operatorname{H}^1$ and look for a solution
$u\in \operatorname{C}\bigl( [0,\infty),\operatorname{H}^1\bigr)$. 
Evaluating at $x=x_0$ yields the Volterra integral equation
\begin{align}\label{vie}\notag
  u(t,x_0)&=\frac{1}{\sqrt{4\pi t}}\int e^{-\frac{|x_0-y|^2}{4t}}u_0(y)\,
  dy-\int _0^t f \bigl(\beta u(\tau,x_0)\bigr)\frac{1}{\sqrt{4\pi
            (t-\tau)}}e^{-x_0^2/4(t-\tau)}\, d\tau\\
  &=: g_\infty(t,x_0)-\int_0^t f \bigl(\beta u(\tau,x_0)\bigr)
    k_\infty(t-\tau,x_0)\, d\tau.
\end{align}
An analogous equation can be obtained for the solution of the
nonlinear problem on the interval $[-L,L]$ for $L>x_0$ simply
replacing the forcing function $g_\infty$ and the kernel $k_\infty$ with
\begin{equation}\label{vieData}
 g_L(t,x_0)=\bigl(e^{-tA_L}u_0 \bigr)(x_0)\text{ and } k_L(t,x_0),
\end{equation}
respectively, where $k_L$ was defined in \eqref{dirHeatKernel}. The
main difference between these two kernels is that $k_L\in
\operatorname{L}^1 \bigl([0,\infty)\bigr)$, due to the exponential
decay of the semigroup while $k$ only decays like $1/\sqrt{t}$ for
large $t$ and initial data in $\operatorname{H}^1$. We will denote by
\eqref{lineTs}$_L$ the corresponding nonlinear equation on the
interval $[-L,L]$ with homogeneous Dirichlet boundary condition with
the same nonlinearity $f$ and initial condition in
$\operatorname{H}^1_L$. To simplify the combined treatment of
the Dirichlet problem on $[-L,L]$ and the problem on the line we will
stipulate again that $\operatorname{H}^1_L=\operatorname{H}^1$ for
$L=\infty$. Existence and uniqueness are a straightforward
application of classical results about nonlinear Volterra integral
equations.
\begin{prop}
Let $L\in (x_0,\infty]$ and $u_0\in\operatorname{H} ^1_L$ be
given. Then\\
(i) The Volterra integral equation with forcing term $g_L$ and
kernel $k_L(\cdot,x_0)$ with $L\in(x_0,\infty)$ has a unique global
solution.\\
(ii) The initial value problem \eqref{lineTs}$_L$ has a unique
global solution $u\in \operatorname{C}\bigl(
[0,\infty),\operatorname{H}^1_L\bigr)\subset \operatorname{C}\bigl(
[0,\infty),\operatorname{H}^1\bigr) $ to any given $u_0\in
\operatorname{H}^1_L$ and, therefore, generates a
global continuous semiflow on $\operatorname{H}^1_L$.
\end{prop}
\begin{proof}
(i) First notice that $g_L(t,\cdot)=e^{-tA_L}u_0$ is continuous with
values in $\operatorname{H}^1$ for all $L$ in the given range since
$\operatorname{H}^1_L\hookrightarrow \operatorname{H}^1$ for any
$L$ by simply extending functions trivially. Since
$\operatorname{H}^1 \hookrightarrow \operatorname{C}$, it follows that
$g_L(\cdot,x_0)\in \operatorname{C} \bigl( [0,\infty)\bigr) $, and, in
view of the decay properties of the semigroups,
$\lim_{t\to\infty}g_L(t,x_0)=0$.
As we are keeping $x_0$ fixed in this argument, we remove it from the
notation from now on.
Existence is obtained by the standard iterative procedure starting
with $y_0=g_L(t)$ and recursively defining
$$
 y_n(t)=g_L(t)-\int _0^t f \bigl( \beta y_{n-1}(\tau) \bigr)
 k_L(t-\tau)\, d\tau.
$$
Setting $\varphi_n=y_n-y_{n-1}$ and $\varphi _0=g_L$, we can
write $y_n=\sum_{k=0}^n \varphi _k$ and a simple use of the global
Lipschitz continuity of the nonlinearity $f$ yields
$$
 \big | \varphi_n (t) \big |\leq C \int _0^t  \big | \varphi _{n-1}
 (\tau) \big |\, d\tau\leq \cdots \leq \| g_L\|
 _{\infty,[0,T]}\frac{(CT)^n}{n!},\: t\in [0,T].
$$
It follows that $y(\cdot)=\sum _{k=0}^\infty \varphi _k$ exists and is
continuous on $[0,T]$ for any $T>0$. Writing
$$
y=y_n+\sum_{k=n+1}^\infty \varphi_k=: y_n+\Delta _n,
$$
it is easily seen that
$$
 y_n(t)=y(t)-\Delta _n(t)=g_L-\int _0^t k_L(t-\tau) f \Bigl( \beta
 \bigl[ y(\tau)-\Delta _n(\tau)\bigr]\Bigr) ,
$$
from which one obtains that
\begin{align*}
 \Big | y(t)-g_L(t)+\int _0^t k_L(t-\tau) f \bigl( \beta
 y(\tau)\bigr)\, d\tau\Big | &\leq \big |\Delta _n(t)\big |+C\int _0^t
 \big |\Delta _{n-1}(\tau)\big |\, d\tau\\
 &\leq \big |\Delta _n(t)\big | +CT\|\Delta _{n-1}\| _{\infty,[0,T]}.
\end{align*}
The terms after the last inequality converge to zero uniformly in $t\in [0,T]$
for any fixed $T>0$, showing that $y$ indeed satisfies the Volterra
integral equation. Uniqueness follows from similar estimates for the
difference of two solutions. We refer to \cite[Chapter 4]{Linz85} for
missing details.\\
(ii) Once a solution $y_L\in \operatorname{C} \bigl( [0,\infty)\bigr)$ is
known, the right-hand-side of \eqref{vcfEq} is completely determined
and the unique mild solution of \eqref{lineTs} is obtained. It follows
from semigroup theory (see e.g. \cite{Pa83,Gol85}) that the right-hand-side of
\eqref{vcfEq} is in $\operatorname{C}
\bigl([0,\infty),\operatorname{H}^1_L\bigr) $. The equation
\eqref{lineTs}$_L$ therefore generates  a global continuous semiflow on the
space $\operatorname{H}^1_L$ for any $L>x_0$.
\end{proof}
\section{Asymptotic stability for solutions of the nonlinear Volterra integral equation}\label{SectionVIE}
In this section we will adapt the stability analysis
presented in \cite[Section 4]{GM20} that builds on results in \cite{Mi71}
to the integral equation  
obtained from the nonlinear thermostat problem $\eqref{lineTs}_{L}$ with Dirichlet boundary
conditions. 
The nonlinear Volterra integral equation is obtained from the global
continous semiflow $\bigl(\Phi_\beta,\operatorname{H}^1_L\bigr)$
induced by $\eqref{lineTs}_{L}$ for arbitrary fixed 
parameters $L,\beta\in(0,\infty)$ and $x_0\in(0,L)\,$.
In fact, let $\Phi_\beta(\cdot, u_0)$ be any orbit of the continuous
semiflow $\bigl(\Phi_\beta,\operatorname{H}^1_L\bigr)\,$, then
$$
 u(t):=\Phi _\beta (t,u_0)(x=x_0)
$$ 
solves the nonlinear, convolution-kernel, Volterra integral equation of the 
second kind
\begin{equation}\label{vie_L}
y(t)=g_L(t)+\int _0^t a_{L}(t-\tau)f\bigl(\beta y(\tau)\bigr)\, d\tau \,,\, t >0,
\end{equation}
where the forcing function $g_L\equiv g_L(u_{0})$ and the convolution 
kernel $k_L=-a_L$ are defined in \eqref{vieData} in the discussion at
the beginning of the previous section.

\begin{rem}
In this and the following sections we always assume that the
nonlinearity $f: \mathbb{R} \to \mathbb{R}$ has the following
properties: 
\begin{itemize}
    \item[(i)] $f$ is bounded, has a continuous derivative, and is
      globally Lipschitz continous. 
    \item[(ii)] $f(0)=0$ and $f'(0)=1$.
    \item[(iii)] $f(\beta w) (w - \frac{f(\beta w)}{\beta})>0$ for
      $w\neq 0\,$ and $\beta\in\mathbb{R}\backslash \{0\}$.
    \item[(iv)] For the statements on the bifurcation and
      the stability of periodic solutions we  additionally assume that
      $f\in \operatorname{C}^{\infty}(\mathbb{R},\:\mathbb{R})$.
\end{itemize}
It may be helpful to think of $f(w)$ as the specific example
$\tanh(w)$ that satisfies the above conditions.
\end{rem}
The main objective of this section is to prove the following
result. Its proof will be given at the end of the section. 

\begin{thm}\label{vie_L_result}
Assume that either \\
$$
 L>x_0>0\text{ and }\beta\in\bigl( 0,\hat{\beta}_1(x_0,L)\bigr),
 \text{ for some constant } \hat{\beta}_1(x_0,L)>0 
$$ 
or 
$$
\beta\in\bigl( 0,\beta_1(x_0)\,\bigr)  \text{ with
}\beta_1(x_0):=\frac{c_\pi}{x_0}\,,\,c_\pi:=\frac{3\pi}{\sqrt{2}}e^{\frac{3\pi}{4}} 
\text{ and }L>C(x_0)\text{ for some constant } C(x_0)>x_0>0
$$ 
holds. Then for arbitrary $u_0\in \operatorname{H}^1_L$ any solution
$y\in\operatorname{BC}\bigl( (0,\infty),\mathbb{R}\bigr)$ 
of the integral equation \eqref{vie_L} with parameters $\beta,L,x_0$
satisfies $\lim_{t\to\infty}y(t)=0\,$. 
\end{thm}
The above constants $\hat{\beta}_1(x_0,L)$ and $C(x_0)$ will be
constructed in the proof of the theorem.  
We proceed by adapting the steps of the proof of the analogous result
in \cite{GM20} to the present situation.  
First we introduce the following slightly modified auxiliary function.
\begin{deff}
For $\beta,q\in(0,\infty)$ and $y\in\operatorname{BC}\bigl(
[0,\infty),\mathbb{R}\bigr)$ set
$$
 W_{\beta,q}(y)(t):=\sum_{i=1}^2 W_i(y)(t),\: t\geq 0,
$$
where
\begin{align*}
 W_1(y)(t)&:=\int _0^t f\bigl(\beta y(\tau)\bigr)\Bigl[
 y(\tau)-\frac{f\bigl(\beta y(\tau)\bigr)}{\beta}\Bigr]\, d\tau\;,
 \\W_2(y)(t)&:=q\,F_{\beta} \bigl( y(t)\bigr)\text{ for }F_{\beta}(z):=\int _0^z
 f(\beta\zeta)\, d\zeta\;.
\end{align*}
Note that, in the sequel, we will at times suppress the dependence on $\beta$, $q$,  
and on the function $y$ in the notation and simply write $W$ and $W_i$
for $i=1,2$.
\end{deff}
The proof of the following lemma is identical to the one given in
\cite[Lemma 4.3]{GM20} and is thus omitted.
\begin{lem}\label{W_nonneg}
Let $\beta,q\in(0,\infty)$ and
$y\in\operatorname{BC}\bigl([0,\infty),\mathbb{R}\bigr)$.
Then
$$
 W_i(t)\geq 0,\: t\geq 0,\: i=1,2.
$$
It therefore also holds that $W_{\beta,q}(t)\geq 0$ for $t\geq 0$.
\end{lem}
The following decomposition of the auxiliary function W(y)(t) along
solutions $y(t)$ of the integral equation \eqref{vie_L} follows by a
simple verification that is carried out in \cite[Lemma 4.4]{GM20} in
full detail. We therefore do not reproduce the proof here.    
\begin{lem}\label{VandR}
Fix $L>0, x_0\in(0,L)$ and $\beta,q\in(0,\infty)\,$. Let $y\in\operatorname{BC}\bigl(
(0,\infty),\mathbb{R}\bigr)$ be a solution of the integral equation
\eqref{vie_L} with parameters $L,x_0,\beta$ and $u_0\in \operatorname{H}^1_L\,$. Then
\begin{equation}\label{Lyap}
  W_{\beta,q}(t)=V_{\beta,q}(t)+R_{\beta,q}(t),\: t\geq 0\,,
\end{equation}
where
$$
 V_{\beta,q}(t):=\int _0^t f \bigl(\beta y(\tau)\bigr)\bigl[
 g_L(\tau)+qg'_L(\tau)\bigr]\, d\tau+qF_{\beta} \bigl( y(0)\bigr) 
$$
and
$$
 R_{\beta,q}(t)\equiv R_{\beta,q,L}(t) :=\int _0^t f\bigl(\beta y(\tau)\bigr)\Big\{
 \int _0^\tau \bigl[ a_L(\tau-\sigma)+q\, a_L'(\tau-\sigma) \bigr]
 f\bigl(\beta y(\sigma)\bigr)\, d\sigma -\frac{f
 \bigl(\beta y(\tau)\bigr)}{\beta}\Big\}\, d\tau.
$$
Using convolutions, the last expression can be written more concisely as
$$
 R_{\beta,q}(t)=\int _0^t f \bigl(\beta
 y(\tau)\bigr)J_{\beta,q}(\tau)\, d\tau,
$$
where $J_{\beta,q}$ is defined for $\tau\geq 0$ by
\begin{equation}\label{J}
  J_{\beta,q}(\tau):=\Bigl[ \bigl[ a_L+q a_L'\bigr]*f \bigl(\beta y(\cdot)\bigr)\Bigr] (\tau)-
  \frac{f\bigl(\beta y(\tau) \bigr)}{\beta} \,.
\end{equation}
\end{lem}
To apply the Parseval-Plancherel Theorem as in \cite[Lemma 4.5]{GM20} we first collect some properties of the 
Fourier transform of the kernel $a_L(t)$ and of its derivative $a'_L(t)\,$.


\begin{rems}\label{kernelrems}
{\bf (a)} For $L>0$ and $x_0\in (0,L)$ the kernel $a_L$ of the
Volterra integral equation \eqref{vie_L} is given by
$a_L(t)=-k_L(t,x_0)$, i.e. it holds that
\begin{equation}\label{aRep}
  a_L(t)=-\frac{1}{L}\sum_{k=0}^\infty (-1)^k \sin \Bigl[
  \frac{(2k+1)\pi}{2L}(x_0+L) \Bigr] \, e^{-t\,\lambda^2_{2k+1,L}},\:
  t>0,
\end{equation}
for $\lambda_{k,L}=\frac{k\pi}{2L}$ according to \eqref{sgL}, 

The kernel $a_L$ can be extended to a $\operatorname{C}^\infty$-function
on $\mathbb{R}$ by setting  $a_L(t):=0$ for $t\leq 0$. When no
confusion seems likely we will not use a different notation for this extension.\\
Besides the convolution kernel $a_L$, the forcing term $g_L$ can also be
expressed in terms of the basis of eigenfunctions to give
\begin{equation}\label{g_L_Rep}
 g_L(t)=\sum_{k=1}^\infty \langle u_0,\varphi _{k,L} \rangle
 \varphi_{k,L}(x_0) e^{-t\,\lambda^2_{k,L}} \;.
\end{equation}
\\[0.1cm]
{\bf (b)} For $u_0\in \operatorname{H}^1$ the forcing function $g_L$ and its
$n$-th derivative are in $\operatorname{BUC}^\infty \bigl(
(0,\infty)\bigr)$.
\\[0.1cm]
{\bf (c)} As already observed in \eqref{dirHeatKernel} the Dirichlet
heat kernel can be obtained from the heat kernel on the whole real
line and can be expressed in terms of the theta function.  
The (general) theta function $\theta_1$ is defined as  
$$
\theta_1(\tau,z):=\sum_{k\in\mathbb{Z}}e^{\pi i k^2 \tau}\,e^{2\pi i
  kz} \:,\, \tau, z \in \mathbb{C}\,,\, \operatorname{Im}(\tau)>0, 
$$
or in less symmetric form by modifying the order of summation  
$$
\theta_1(\tau,z):=2 \sum_{k=0}^{\infty} (-1)^k q^{(\frac{k+1}{2})^2}\,
\sin((2k+1)z),\: \tau, z \in
\mathbb{C},\:\operatorname{Im}(\tau)>0,\: q:=e^{i\pi\tau}.
$$ 
By setting $\alpha(L, x_0)= \frac{\pi}{2L} (x_0 + L)$, it is
directly verified from \eqref{aRep} and the definition of $\theta_1$
that  
$$
 -a_L(t)=k_L(t,x_0)=\Bigl[e^{-tA_L}\delta_0\Bigr](x_0)=\frac{1}{2L}\theta_1\Bigl(
 i\frac{4\pi}{L^2}t ,\alpha(L, x_0)\Bigr).
$$
{\bf (d)} The kernel $a_L(t)$ satisfies 
$$
 \lim _{t\to\infty} a_L(t)=0 \text{ and } \lim _{t\searrow 0} a(t)=0\:.
$$
The limit for $t\to\infty$ is obtained directly from \eqref{aRep}. 
To determine the one-sided limit we observe that    
$$
-a_L(t)=\frac{1}{2L}\theta_1\Bigl( i\frac{4\pi}{L^2}t ,\alpha(L, x_0)\Bigr),
$$
and therefore the limit for ${t\searrow 0}$ follows from known
properties of the heat kernel and the theta function.\\[0.1cm]
{\bf (e)} The kernel $a_L(t)$ (extended by 0 for $t\leq0$) satisfies
$$
 a_L\in \operatorname{BC}^\infty(\mathbb{R},\mathbb{R})
$$
and its derivatives satisfy
$$
 a^{(n)}_L\in \operatorname{L}^p(\mathbb{R}),\: p\in [1,\infty],\: n\geq1.
$$
This again follows from the kernel's representation \eqref{aRep}.\\[0.1cm]
%
{\bf (f)} The series representation of the Fourier transforms of $a_L$ and its derivate $a'_L$ 
will be needed in the sequel to recover the Popov stability criterion for the Volterra integral equation. 
They are given by
\begin{equation}\label{aFourier}
  \hat a_L(\omega) = -\frac{1}{L}\sum_{k=0}^\infty (-1)^k 
  \frac{\sin \bigl[ \frac{(2k+1)\pi}{2L}(x_0+L) \bigr]}{i\omega + \lambda_{2k+1,L}},	
\end{equation}
and by
\begin{equation}\label{aprimeFourier}
  \widehat{a'_L}(\omega)=i\omega\hat a_L(\omega) =-\frac{i\omega}{L}\sum_{k=0}^\infty (-1)^k 
  \frac{\sin \bigl[ \frac{(2k+1)\pi}{2L}(x_0+L) \bigr]}{i\omega + \lambda_{2k+1,L}},
\end{equation} 
respectively.\\[0.1cm]
This follows from elementary properties of the Fourier transform
similarly as discussed in \cite[Remarks 2.2 (g)]{GM20}. \\ 
%
{\bf (g)} The Laplace transform $\mathcal{L}(a_L)$ of $a_L(t)$ is given by
$$
\mathcal{L} (a_L)(s)=-\frac{1}{L} \sum _{k=0}^\infty (-1)^k \frac{\sin
  \bigl[ \frac{(2k+1)\pi}{2L}(x_0+L) \bigr] }{s + \lambda_{2k+1,L}}
$$
for $s\in\big\{ z\in \mathbb{C} \, \big |\,
\operatorname{Re}(z)>0\big\}$. It also has the explicit representation
\begin{equation}\label{expltransL}
 \mathcal{L} (a_L)(s)=-\frac{\sinh\bigl(\sqrt{s}(L-x_0)\bigr)}{2\sqrt{s}\cosh(\sqrt{s}L)}\,.
\end{equation}
The series representation of the Laplace transform is obtained from \eqref{aRep} by 
elementary integrations. The explicit representation follows from \eqref{greensL}. 
A direct verification that the explicit form is represented by the above series is also 
given in \cite{CM09} by a partial fraction expansion and by
determination of the residuals of the poles of \eqref{expltransL}. \\[0.1cm] 
%
%
{\bf (h)} We note that, for given parameters $L>x_0>0\,$, the associated transfer function   
\begin{equation}\label{transLexpl}
G_{L,x_0}(s):=-\mathcal{L}(a_L)(s)= \frac{\sinh\bigl(\sqrt{s}(L-x_0)\bigr)}{2\sqrt{s}\cosh(\sqrt{s}L)}\,,
\end{equation}
can be expressed in terms of the Fourier transforms of $a_L$ and that of $a_L'\,$.  
In fact, for $\omega\in \mathbb{R}\backslash\{0\}$, it holds that
$$
 \operatorname{Re}\bigl( G_{L,x_0}(i \omega)\bigr)=-\operatorname{Re}
 \bigl( \hat a_L(\omega)\bigr)
$$
and 
\begin{equation*}
 \omega \operatorname{Im}\bigl( G_{L,x_0}(i \omega)\bigr)=
 \operatorname{Re}\bigl( \widehat{a'_L}(\omega)\bigr) .
\end{equation*}
The inequality
\begin{equation}\label{M}
 \operatorname{Re}\bigl( \hat a_L(\omega)\bigr)+
 q\operatorname{Re}\bigl(\widehat{a'_L}(\omega)\bigr)-\frac{1}{\beta}\leq 0
 ,\:\omega\in\mathbb{R}\backslash\{0\},
\end{equation}
is then equivalent to
\begin{equation}\label{P}
  \operatorname{Re}\bigl( G_{L,x_0}(i \omega)\bigr)-q \omega 
  \operatorname{Im}\bigl( G_{L,x_0}(i \omega)\bigr)\geq-\frac{1}{\beta}
  ,\:\omega\in \mathbb{R}\backslash\{0\}. 
\end{equation}
We will use this relationship between $G_{L,x_0}(s)$ and $\hat a_L$ and
$\widehat{a'_L}$ to verify that the stability condition \eqref{M}
(which we will obtain below from the analysis of the integral equation
\eqref{vie_L}) is equivalent to the well-known Popov stability
criterion \eqref{P} when applied to the transfer function $G_{L,x_0}$.
\end{rems}  
In the next lemma we apply the Parseval-Plancherel Theorem to derive an alternative
representation of $R_{\beta,q}\,$. It will reveal a condition,
expressed in terms of the Fourier transforms of $a_L(t)$
and $a'_L(t)\,$, which implies $R_{\beta,q}(t)\leq0\;$. 
The nonnegativity of the quantity $W_{\beta,q}$ will then allow to bound $V_{\beta,q}$ from above.  
The proof is straightforward and can be found in \cite[Lemma 4.5]{GM20}.

\begin{lem}\label{ParPlanCriterion}
Let $\beta,q,L\in(0,\infty)$ and $x_0\in(0,L)\,$. Let $y\in\operatorname{BC}\bigl(
(0,\infty),\mathbb{R}\bigr)$ be a solution of the integral equation
\eqref{vie_L} with parameters $L,x_0,\beta$ and $u_0\in \operatorname{H}^1_L$. Then
$$
 R_{\beta,q}(t)=\frac{1}{2\pi}\int _{-\infty}^\infty
 {\hat{f}^2_{\beta,\theta_t}}(\omega)\bigl[
 \hat a_L(\omega)+q\,\widehat{a_L'}(\omega)-\frac{1}{\beta}\bigr]
 \, d \omega,\: t\geq 0,
$$
where
$$
 f_{\beta,\theta_t}(\tau):=f \bigl(\beta
 y(\tau)\bigr)\theta_t(\tau),\:\tau\in \mathbb{R},\: t\geq 0,
$$
with 
$$
 \theta _t(\tau):=\begin{cases} 1,&\tau\in [0,t],\\ 0,&\tau\in
   \mathbb{R}\setminus[0,t]\, \end{cases}
$$
and $y(\tau):=0$ for $\tau<0\,$.
\end{lem}
The Fourier transforms of $a_L$ and $a_L'$ were discussed in Remarks
\ref{kernelrems}. Since $f_{\beta,\theta_t}\in
\operatorname{L}^1(\mathbb{R})\cap \operatorname{L}^2(\mathbb{R})$ for each 
$t\geq0\,$, its Fourier transform  is defined classically.

After these clarifications we can formulate a lemma showing that
controlling the sign of $R_{\beta,q,L}(t)$ for $t\geq0$ leads to a
bound for $V_{\beta,q,L}(t)$.  
The proof of this lemma is simpler than its counterpart in
\cite[Lemma 4.10]{GM20} for the case of Neumann boundary conditions and
boundary control. This is due to the fact that the semigroup
associated with the heat equation on $(-L,L)$ subject to Dirichlet
boundary conditions decays to the trivial solution exponentially for
any initial state $u_0\in H^1_L$.
 
\begin{lem}\label{lembound_V}
Let $\beta,L\in(0,\infty)$ and $x_0\in(0,L)\,$. Let
$y\in\operatorname{BC}\bigl( (0,\infty),\mathbb{R}\bigr)$ be a
solution of the integral equation \eqref{vie_L} with parameters
$L,x_0,\beta$ and $u_0\in \operatorname{H}^1_L$. Then, if for some
$q>0$ it holds that   
$$
R_{\beta,q,L}(t)\leq 0 \text{ for } t\geq 0,
$$ 
then the function $V(t)\equiv V_{\beta,q,L}(t)$ defined in Lemma
\ref{VandR} satisfies 
$$
 0\leq V(t)\leq c,\:t\geq 0,
$$
for a constant $c>0$ independent of $t\geq 0$.
\end{lem}
\begin{proof}
As $R(t)\leq0$ and by the definition of $V$ and $W\,$, it holds that 
$$
0\leq V(t)=\int _0^t f \bigl(\beta y(\tau)\bigr)\bigl[ g_L(\tau)+q\,
g'_L(\tau)\bigr]\, d\tau+W_2(y)(0),\: t\geq 0,
$$
and therefore 
$$
V(t) \leq \int _0^t \big |f \bigl(\beta y(\tau)\bigr)\big | \, \big |g_L(\tau)\big | d\tau +
q\int _0^t \big |f \bigl(\beta y(\tau)\bigr)\big | \, \big |g'_L(\tau)\big |\, d\tau 
+ W_2(y)(0),\: t\geq 0\,.
$$
The assertion now follows from the assumed boundedness of
$f(\beta\cdot)$ and from the exponential decay and analyticity
of the semigroup
$e^{-tA_L}$ associated with the heat equation on $(-L,L)$ subject to
Dirichlet boundary conditions. Namely, in order to
bound $V(t)$ by a constant independent of $t$ we use the estimate
$$
 \big |g_L(\tau)\big |= \big |\bigl( e^{-\tau A_L}u_0\bigr)(x_0)
 \big| \leq c\|  e^{-\tau A_L}u_0\| _{\operatorname{H}^1_L}\leq c\,
 e^{-\alpha _L \tau}\| u_0\| _{\operatorname{H}^1_L}\:,
$$
and the estimate
$$
\big |g'_L(\tau)\big |= \big |\bigl( A_L e^{-\tau A_L}u_0\bigr)(x_0)
 \big| 
\leq 
c\| e^{-\tau A_L}A_L u_0\|
 _{\operatorname{H}^{\frac{1}{2}+\varepsilon}_L}
 \leq 
\frac{c}{\tau^{\frac{3}{4}+\frac{\varepsilon}{2}}}
 e^{-\alpha_L\tau} \| A_Lu_0\| _{\operatorname{H}^{-1}_L}\:,
$$
which are valid for $\tau>0$, $\varepsilon\in(0,\frac{1}{2})$ and by the choice of an appropriate 
constant $\alpha_L>0\,$. We refer to \cite{Gol85} or \cite{Pa83} for these 
standard estimates on analytic semigroups in interpolation spaces.  


\end{proof}

Next we verify that bounded and continuous solutions of the Volterra
integral equation \eqref{vie_L} are also uniformly continuous.
As for the previous lemma, the proof turns out to be simpler than its
counterpart \cite[Lemma 4.11]{GM20} in the case of Neumann boundary
conditions. 
\begin{prop}\label{uniformcont}
Let $\beta,L\in(0,\infty)$ and $x_0\in(0,L)\,$. Let
$y\in\operatorname{BC}\bigl((0,\infty),\mathbb{R}\bigr)$ be a solution
of the integral equation \eqref{vie_L} with parameters $L,x_0,\beta$
and $u_0\in \operatorname{H}^1_L$. Then
$y\in\operatorname{BUC}\bigl((0,\infty),\mathbb{R})\bigr)$.
\end{prop}

\begin{proof}
Since $y$ solves \eqref{vie_L}, we have that

$$
 y(t)=g_L(t)+\int _0^t a_{L}(t-\tau)f\bigl(\beta y(\tau)\bigr)\,
 d\tau,\: t\geq 0,
$$
where $g_L(t)=\bigl( e^{-tA_L}u_0 \bigr)(x_0)$ is the forcing function
induced by $u_0$. It suffices to verify that both terms in the above
sum are uniformly continuous. Uniform continuity of the first term
holds on any finite interval and the derivative of $y$ is bounded for
$t\geq 1$ by the standard smoothing effect of analytic semigroups. 
Hence $g_L$ is uniformly continuous on $(0,\infty)\,$.\\
The second term can be written as a convolution
$$
 \int _0^t a_L(t-\tau) f \bigl(\beta y(\tau)\bigr)\, d\tau=\Bigl[
 a_L*\bigl( f\circ (\beta y) \bigr)\Bigr](t),\: t\geq 0. 
$$
Since $a_L\in \operatorname{L}^1 \bigl( [0,\infty)\bigr)$ and
$f\circ (\beta y) \in \operatorname{L}^\infty \bigl( [0,\infty)\bigr)$ 
by the assumed boundedness of $f$, well-known results on the regularity 
properties of convolutions imply the uniform continuity of the second term 
(see e.g. \cite{Ama95} or \cite{Fo84}). 
\end{proof}
We will now show the following statement: if, for a given fixed choice
of the parameters
$\beta,L, x_0$ and $u_0\in \operatorname{H}^1_L$,  a constant $q>0$
can be determined, such that $R_{\beta,q,L}(t)\leq 0$ along the
solution $y(t)$ of \eqref{vie_L}, then this implies the convergence of
$y(t)$ to zero.\\
By Lemma \ref{ParPlanCriterion} a suitable constant $q>0$ is found if
it verifies the inequality
$$
\hat a_L(\omega)+q\,\widehat{a_L'}(\omega)-\frac{1}{\beta} \leq 0 \text{
  for } \omega\in\mathbb{R}\backslash\{0\}. 
$$ 
If such a $q>0$ can be found, then it does not depend on 
the initial state $u_0\in \operatorname{H}^1_L\,$, since $u_{0}$ does
not appear in the above inequality. However, the choice of such a
suitable $q>0$ may and does depend on the choice of the parameters
$\beta,L, x_0$, as will be discussed below.  
By the relationship between the transfer function $G_L(s)$ and the 
Fourier transform of the kernel $a_L(t)$ discussed in Remarks \ref{kernelrems} 
the search of $q$ can be reinterpreted as the task of finding a
straight line in the complex plane with positive slope $\frac{1}{q}$ that
intersects the real axis at $-\frac{1}{\beta}$ such that the so-called
Popov curve associated with the kernel $a_L$ lies to the right of that
straight line.
We refer to \cite[Section 3]{GM20} for a sketch of this relationship
to feedback control problems and the celebrated Popov criterion. The
following proposition is a slight adaptation of the proof of
\cite[Proposition 4.1]{GM20}.
\begin{prop}\label{Prop_stab_vie}
Fix $\beta,L\in(0,\infty)$ and $x_0\in(0,L)\,$. Let
$y\in\operatorname{BC}\bigl((0,\infty),\mathbb{R}\bigr)$ be a
solution of the integral equation \eqref{vie_L} with parameters
$L,x_0,\beta$ and $u_0\in \operatorname{H}^1_L$.
If for some $q>0$ it holds that  
$$
R_{\beta,q,L}(t)\leq 0 \text{ for } t\geq 0,
$$ 
then 
$$
\lim_{t\to\infty}y(t)=0.
$$
\end{prop}
\begin{proof}
By assumption there exists $q>0$ such that
$$
 R_{\beta,q,L}(t)\leq 0\text{ for }t\geq 0.
$$
By Lemma \ref{lembound_V} ther exists $c>0$ such that
$$
 c\geq V_{\beta,q,L}(t)\geq W_{\beta,q}(t)\geq
 W_{1,\beta}(t)\geq 0,
$$
for $t\geq 0$, i.e. such that
$$
 W_{1,\beta}(t)=\int _0^t f\bigl(\beta y(\tau)\bigr)
 \bigl[ y(\tau)-\frac{f \bigl(\beta y(\tau)\bigr)}{\beta}\bigr]
 \, d\tau=:\int _0^t H\bigl( y(\tau)\bigr)\, d\tau\leq c<\infty,
$$
for $t\geq 0$. As shown in \cite[Lemma 4.12]{GM20} the function $H(y)$
is non-negative, only vanishes if $y=0$, and is uniformly
continuous. Assume next by contradiction that $y(t)\not\to 0$ 
as $t\to\infty$. Then, since $H(\xi)>0$ for $0\neq \xi\in \mathbb{R}$,
there is a sequence $(t_m)_{m\in \mathbb{N}}$ in $\mathbb{R}$ with
$t_m\to\infty$ and a constant $\varepsilon>0$ such that
$$
 H \bigl( y(t_m)\bigr)\geq 2 \varepsilon\text{ for all }m\in \mathbb{N}. 
$$
Since $H\circ y\in \operatorname{BUC}\bigl([0,\infty)\bigr)$, a
$\delta>0$ can be found such that
$$
 H \bigl( y(t)\bigr)\geq c\text{ for }t\in [t_m-\delta,t_m+\delta ],
$$
and all $m\in \mathbb{N}$. It follows that
$$
 W_1(t_m)=\int _0^{t_m}H \bigl( y(\tau)\bigr)\, d\tau\geq
 \sum_{k=1}^{m-1}\int _{t_k-\delta}^{t_k+\delta}H \bigl(
 y(\tau)\bigr)\, d\tau\geq (m-1)2 \delta \varepsilon, 
$$
which contradicts the boundedness of $W_1$ on $[0,\infty)$
since $m$ can be chosen arbitrarily large. 
\end{proof}
In order to complete the proof of Theorem \ref{vie_L_result} it
remains to show that, for suitably constructed constants
$\hat{\beta}_1(x_0,L)>0$ and $C(x_0)$, either the assumption
$$
\beta\in\bigl(0,\hat{\beta}_1(x_0,L)\bigr)\text{ and arbitrary }L>x_0>0,
$$ 
or the assumption 
$$
\beta\in\bigl(0,\beta_1(x_0)\bigr)\text{ and }L>C(x_0)>x_0>0,
$$ 
are sufficient to find a $q\equiv q(x_0,\beta)>0$ such that 
$$
\hat a_L(\omega)+q\,\widehat{a_L'}(\omega)-\frac{1}{\beta} \leq 0
\text{ for } \omega\in\mathbb{R}\backslash\{0\}.
$$
We note that, by symmetry, it suffices to verify the above inequality
for $\omega>0$. The following discussion of the limit of the transfer
function $G_{L,x_0}(s)$ for large $L$ is an important element of the
proof.
\begin{rem}\label{Ltoinfinity}
Fix $x_0>0$ and consider the transfer function $G_{L,x_0}(s)$ for
$L\in(x_0,\infty)$. Then 
\begin{equation}\label{LtoinfinityEq}
G_{L,x_0}(s)=\frac{\sinh\bigl(\sqrt{s}(L-x_0)\bigr)}{2\sqrt{s}\cosh(\sqrt{s}L)} 
\to G_{x_0}(s)=\frac{e^{-\sqrt{s}\,x_0}}{2\sqrt{s}} \text{ as } L\to \infty,
\end{equation}
uniformly on compact subsets of $\mathbb{C} \backslash
(-\infty,0]$. The convergence is also uniform for $s$ 
in (unbounded) subsets of the imaginary axis of the form $i\,\bigl(
\mathbb{R}\setminus (-\varepsilon,\varepsilon)\bigr)$.
\end{rem}
\begin{proof}
Note that by expanding $\sinh(z_1+z_2)$ we obtain 
$$
G_{L,x_0}(s)\,=\,\frac{1}{2\sqrt{s}}
\bigl[\cosh(\sqrt{s}x_0)\,\tanh(\sqrt{s}L) - \sinh(\sqrt{s}x_0)\bigr].
$$
Also note that whenever $\operatorname{Re}(\sqrt{s}) > 0
\Leftrightarrow s \in \mathbb{C} \backslash (-\infty,0]$ it holds that
$$
\lim\limits_{L\to\infty} \tanh(\sqrt{s}L) =1,  
$$
and the convergence is uniform on compact subsets of $\mathbb{C}
\backslash (-\infty,0]$. Since $\frac{1}{\sqrt{s}}$ is bounded on such
subsets and $\cosh(\sqrt{s}x_0)-\sinh(\sqrt{s}x_0)=e^{-\sqrt{s}\,x_0}$
the first assertion follows. Observe that we write $\sqrt{z}$
for the principal branch of the complex square root. Hence $\sqrt{\pm
  i}=\frac{1}{\sqrt{2}}(1 \pm i)$ and $\operatorname{Re}(\sqrt{\pm
  i})=\frac{1}{\sqrt{2}} > 0$. Therefore $\tanh(\sqrt{i\omega}L)$
converges to 1 as $L\to\infty$ uniformly for $\omega$ over any set of
the form  $\mathbb{R}\setminus (-\varepsilon,\varepsilon)$. Since
$\frac{1}{\sqrt{s}}$ is bounded on sets of that form the second
assertion follows.
\end{proof}


In analogy to \cite[Proposition 4.2.]{GM20} we introduce the Popov set
corresponding to a given transfer function. The set contains the frequencies
$\omega$ at which the Popov curve in the complex plane 
$$
\operatorname{Re}[ G_{x_0}(i\omega)] +i\,\omega\, \operatorname{Im}[
G_{x_0}(i\omega)],\:\omega>0, 
$$
intersects the imaginary axis. Describing the structure of the Popov
set will help finding a parameter range for $\beta$ that guarantees
the asymptotic stability of the trivial solution of \eqref{lineTs} for
given parameters $L>x_0>0$.
\begin{deff}\label{Popovsetdefs}  
For $L > x_0 >0$ we set 
$$
 \Omega^{\text{Pop}}_{x_0}:=\Big\{\omega>0 \,\Big |\, \operatorname{Im}
 \Bigl(\:\operatorname{Re}[ G_{x_0}(i\omega)] +i\omega
 \operatorname{Im}[G_{x_0}(i\omega)]\Bigr)=0\Big\}=\big\{\omega>0
 \,\big |\,\operatorname{Im}\bigl[ G_{x_0}(i\omega)\bigr]=0\big\}
$$
and
$$
 \Omega^{\text{Pop}}_{L,x_0}:=\Big\{\omega>0 \,\Big |\,
 \operatorname{Im}\Bigl(\operatorname{Re}[G_{L,x_0}(i\omega)] +
 i\omega\operatorname{Im}[G_{L,x_0}(i\omega)]\Bigr)=0\Big\}
 =\big\{\omega>0 \,\big |\, \operatorname{Im}\bigl[G_{L,x_0}(i\omega)
 \bigr]=0\big\}. 
$$
\end{deff} 
In the next proposition we show that the Popov set of the limiting
transfer function $G_{x_0}$ can be described explicitly.
By determining the first zero $\omega_1 \approx 11.1033$
of $1+\tan(\sqrt{\frac{\omega}{2}}x_0)\,$ for the sensor location
$x_0=1$ we recover the constant $\beta_1\approx 70.3134$, found
above in \eqref{beta1first}, by using the relationship 
\begin{equation}\label{critbetaEq}
1+\beta_1\, G_{L,x_0}(i\omega_1)=0.
\end{equation}
This relationship between the Popov set and the critical
parameter values of $\beta$ is understood by observing that
$$
\operatorname{Im}\bigl[
1+\beta\,G_{L,x_0}(i\omega)\bigr]=\operatorname{Im}\bigl[
G_{L,x_0}(i\omega) \bigr]. 
$$ 
This entails that the locations where the imaginary part of
\eqref{critbetaEq} vanishes are independent of $\beta$. Once the zeros
$\omega$ of the imaginary part of the transfer function $G_{L,x_0}$
are found, one recovers the corresponding critical values for $\beta$
by simply equating the real part to zero, i.e. by solving 
$$
\operatorname{Re}\bigl[ 1+\beta\,G_{L,x_0}(i\omega)\bigr] =0,
$$
for $\beta$. The smallest positive solution arising in the above
procedure is precisely $\beta_1$. This relationship between the Popov
set $\Omega^{\text{Pop}}_{x_0}$ and the corresponding parameter values
for $\beta$ that correspond to the existence of a pair of complex
conjugate eigenvalues of the operator $-A_{\beta}$ lying on the
imaginary axis is also discussed in more detail in the remarks
following \cite[Proposition 4.9]{GM20}.  
\begin{prop}\label{explpopline}
For $x_0>0$, the Popov set of $G_{x_0}$ is given as the solution set
\begin{equation}\label{explpopEq}
 \Omega^{\text{Pop}}_{x_0} = \big\{\omega>0 \,\big |\,
 1+\tan\bigl(x_0\sqrt{\omega/2}\bigr)=0\big\} = \big\{
 \omega_k=\frac{(4k-1)^2 \pi^2}{8x_0^2} \,\big |\,
 k=1,2,3,\dots\big\}.
\end{equation}
Hence $\Omega^{\text{Pop}}_{x_0}$ is an infinite countable set that
consists of positive, non-degenerate (simple) roots of the function
$1+\tan(x_0\sqrt{\omega/2})$.
\end{prop}
\begin{proof}
Setting $r:=-x_0\sqrt{\omega}<0$ we observe that
$$
\operatorname{Im}[ \, G_{x_0}(i\omega)\:\bigr]=0  \Leftrightarrow
\operatorname{Im}[\,\frac{e^{r\,\sqrt{i}}}{\sqrt{i}}\:\bigr]=0,
$$
and since it holds that
$e^{r\sqrt{i}}=e^{r\alpha}\cos(r\alpha)+ie^{r\alpha}\sin(r\alpha)$, with
$\alpha:=\sin(\frac{\pi}{4})=\cos(\frac{\pi}{4})=\frac{1}{\sqrt{2}}$, 
we find that
$$
 \operatorname{Im}\Bigl[ \frac{e^{r\sqrt{i}}}{\sqrt{i}}\Bigr]=0
 \Longleftrightarrow\operatorname{Im}\Bigl[\overline{\sqrt{i}}\cos(r\alpha)
 +\sqrt{i}\sin(r\alpha)\Bigr]=0\Longleftrightarrow\cos(r\alpha)=\sin(r\alpha),  
$$
and thus, for $\omega>0$, the assertion
$$
 \operatorname{Im}[ G_{x_0}(i\omega)\:\bigr]=0  \, \Longleftrightarrow
 1+\tan\bigl(x_0\sqrt{\omega/2}\bigr) =0
$$
follows. The other statements follow from elementary properties of $\tan(x)\,$.
\end{proof}
We note that the Popov set $\Omega^{\text{Pop}}_{x_0}$ also contains values 
$\omega_{k}$ that lead to positive values of $G_{x_0}(i\omega_k)$ and
thus to negative critical values $\beta(\omega_{k})$. More precisely,
if we set 
$$
\Omega^{\text{Pop}\,+}_{x_0} :=
\big\{\omega\in\Omega^{\text{Pop}}_{x_0} \,\big |\, G_{x_0}(i\omega)
<0\big\}\text{ and } \Omega^{\text{Pop}\,-}_{x_0} :=
\big\{\omega\in\Omega^{\text{Pop}}_{x_0} \,\big |\, G_{x_0}(i\omega)
>0\big\},
$$
then
$$
\Omega^{\text{Pop}\,+}_{x_0} = \big\{\omega_k=\frac{(4k-1)^2
  \pi^2}{8x_0^2} \,\big |\, k\in 2\mathbb{N}-1\big\}\text{ and } 
\Omega^{\text{Pop}\,-}_{x_0}
=\big\{\omega_k=\frac{(4k-1)^2\pi^2}{8x_0^2}\,\big |\, k\in 2
\mathbb{N}\big\},,
$$
where $\mathbb{N}:=\{1,2,3,...\}\,$. This is analogous to the discussion in \cite{GM20} and in Section \ref{SecLinear} and also captures that
pairs of conjugate complex eigenvalues in the spectrum of $A_{\beta}$ do cross the imaginary axis for certain negative values of
$\beta$, which are determined by
$$
 -\frac{1}{G_{x_0}(i\omega)} \text{ for }
 \omega\in\Omega^{\text{Pop}\,-}_{x_0}. 
$$
The positive values of $\beta$ where a crossing occurs are found by
$$
-\frac{1}{G_{x_0}(i\omega)} \text{ for }
\omega\in\Omega^{\text{Pop}\,+}_{x_0}.
$$ 
Thus the Popov set
$\Omega^{\text{Pop}}_{x_0}=\Omega^{\text{Pop}\,+}_{x_0} \,\cup\
\Omega^{\text{Pop}\,-}_{x_0}$ captures
both the positive and negative values of $\beta$ where complex
conjugate eigenvalue pairs of $A_{\beta}$ cross the imaginary axis.  
By the positivity of the semigroup for $\beta\leq 0$ the stability of
the trivial solution is determined by the (real) principal eigenvalue
and not by a Hopf bifurcation induced by a complex conjugate pair of
eigenvalues first crossing into the unstable complex half plane. 
In that sense, for negative values of $\beta$, the problem has
positivity properties that lead to a more familiar behavior which is 
well-studied in the context of semilinear parabolic equations. We also
refer to \cite{GM99} for a discussion of positivity aspects.

\subsection{The Popov criterion in the Limit $L=\infty$}
Next we show that the stability criterion \eqref{P} can be verified
for the limiting transfer function $G_{x_0}(s)$.
\begin{prop}\label{PopCritLim}
For any $x_0>0$ there exists $\beta_1(x_0)=\frac{c_\pi}{x_0} > 0$ for
$c_\pi=\frac{3\pi}{\sqrt{2}}e^{\frac{3\pi}{4}}$, such that, for $\beta
\in \bigl(0,\beta_1(x_0)\bigr)$, there exists $q(x_0)>0$ which
satisfies the Popov criterion, i.e. such that the inequality
\begin{equation}\label{PLim}
  \operatorname{Re}\bigl[ G_{x_0}(i \omega)\bigr]-q(x_0) \omega 
  \operatorname{Im}\bigl[ G_{x_0}(i \omega)\bigr] \geq -\frac{1}{\beta} 
\end{equation}
holds for all $\omega\in \mathbb{R}\setminus\{0\}$. 
\end{prop}
\begin{proof}
It is sufficient to show that the Popov curve parametrized as
$$
\Gamma_{x_0}(\omega):= \bigl( x(\omega),y(\omega)\bigr):=
\Bigl( \operatorname{Re}\bigl( G_{x_0}(i \omega)\bigr),\:
\omega\,\operatorname{Im}\bigl( G_{x_0}(i \omega)\bigr)\Bigr)
$$
lies in the half-plane
$$ 
 H_{q,\beta}:=\big\{  (x,y)\in\mathbb{R}^2 \,\big |
 \,F_{q,\beta}(x,y)\leq 0 \big\}
$$
that is defined by the functional 
$$
F_{q,\beta}(x,y):=y-\frac{1}{q}\,x-\frac{1}{q\,\beta}
$$
for given $q,\beta>0$. Thus verifying the Popov criterion \eqref{PLim}
is equivalent to showing that
\begin{equation}\label{FuncIneq}
F_{q,\beta}\bigl(\Gamma_{x_0}(\omega)\bigr)\leq 0,\:\omega>0.
\end{equation}
A somewhat tedious but elementary computation yields
\begin{align}\label{ExplF}
F_{q,\beta}\bigl(\Gamma_{x_0}(\omega)\bigr)=
&-\frac{\alpha\sqrt{\omega}}{2} e^{-\alpha x_{0}\sqrt{\omega}} \: 
\left[\sin(\alpha x_{0}\sqrt{\omega})+\cos(\alpha
x_{0}\sqrt{\omega})\right]\\ 
&-\frac{\alpha}{2q\sqrt{\omega}} e^{-\alpha x_{0}\sqrt{\omega}} \,
\left[\cos(\alpha x_{0}\sqrt{\omega})-\sin(\alpha
x_{0}\sqrt{\omega})\right] -\frac{1}{q\beta}\:, \notag
\end{align}
where $\alpha:=\frac{1}{\sqrt{2}}\,$. Next, for each $x_0>0$, we fix
$\beta_1(x_0)$ as follows 
$$
\beta_1(x_0):=-\frac{1}{\operatorname{Re}\bigl(
  G_{x_0}(i\omega_1)\bigr)},
$$
where, using \eqref{explpopEq}, we can express $\omega_1$ explicitly as
a function of $x_0$ as
$$
\omega_1(x_0):=\min \Omega^{\text{Pop}}_{x_0}=
\frac{2r_1^2}{x_0^2},\:r_1:=\arctan(-1)+\pi=\frac{3}{4}\pi\;, 
$$
i.e.
\begin{equation}\label{omega1explicitinfty}
\omega_1(x_0)=\frac{b_\pi}{x_0^2}\;,\;b_\pi:=\frac{9 \pi^2}{8}\;. 
\end{equation}
We also note that inserting the explicit expression for
$\omega_1(x_0)$ into $G_{x_0}$ we easily obtain
$$
\beta_1(x_0)=\frac{c_\pi}{x_0}\;,\;c_\pi=\frac{3\pi}{\sqrt{2}}\;e^{\frac{3\pi}{4}}.
$$ 
By the definition of $F_{q,\beta}\,$, $\Gamma_{x_0}$ and
$\omega_1(x_0)$ or by a simple direct verification using \eqref{ExplF},
it follows that 
$$
F_{q,\beta_1(x_0)}\bigl( \Gamma_{x_0}(\omega_1(x_0)\bigr)=0,\: q>0,
$$
holds for any $x_0>0$. Given $x_0$, we set $\frac{1}{q(x_0)}$ to be
the slope of the curve $\Gamma_{x_0}(\omega)$ at  
the point where it intersects the real axis for $\omega=\omega_1(x_0)\,$. 
The first such intersection occurs for the parameter value
$\omega=\omega_1(x_0)\,$. In other words, we use the parametrization
$\Gamma_{x_0}(\omega)=\bigl( x(\omega),y(\omega)\bigr)$ of the curve
to define   
$$
\frac{1}{q(x_0)}:=\frac{\dot{y}}{\dot{x}} \Big | _{\omega=\omega_1(x_0)}\,.
$$
In order to find an explicit formula for $\frac{1}{q(x_0)}$, we first
rewrite the coordinates $x(\omega)$ and $y(\omega)$ as follows
$$
x(\omega)= \operatorname{Re}\bigl( G_{x_0}(i \omega)\bigr)=
\frac{\alpha}{2\sqrt{\omega}}e^{-x_0 \alpha\sqrt{\omega}}\Bigl[
\cos\bigl( x_{0}\alpha\sqrt{\omega}\bigr)-\sin\bigl(
x_{0}\alpha\sqrt{\omega}\bigr)\Bigr]
$$
and 
$$
y(\omega)= \omega\operatorname{Im}\bigl( G_{x_0}(i \omega)\bigr)=
-\frac{\alpha \sqrt{\omega}}{2}e^{-x_0 \alpha\sqrt{\omega}}\Bigl[
\cos\bigl( x_{0}\alpha\sqrt{\omega}\bigr)+\sin\bigl(
x_{0}\alpha\sqrt{\omega}\bigr)\Bigr].
$$
When differentiating and evaluating these expressions at
$\omega=\omega_1(x_0)$ in order to compute $\dot{x}$
and $\dot{y}\,$, we use that
$$
\sin\bigl( x_{0}\alpha \sqrt{\omega_{1}(x_0)}\bigr) + \cos\bigl( x_{0} \alpha
\sqrt{\omega_{1}(x_0)}\bigr)=0 
$$ 
and that
$$
\frac{d}{d\omega}\Big | _{\omega=\omega_1(x_0)} \bigl[ \sin\bigl(
x_{0}\alpha\sqrt{\omega}\bigr)-\cos\bigl(
x_{0}\alpha\sqrt{\omega}\bigr) \bigr]
=0\,. 
$$
An elementary calculation then shows that
\begin{equation}\label{Explq}
  \frac{1}{q(x_0)}=\frac{2r_1^2}{x_0^2(1+\frac{1}{r_1})}=
  \frac{d_\pi}{x_0^2}\;,\;d_\pi:=\frac{9\pi^3}{8\pi+\frac{32}{3}}\,. 
\end{equation}
To complete the proof we need to show that for arbitrary $x_{0}>0$ the
inequality 
$$
F_{q(x_{0}),\beta}\big( \Gamma_{x_0}(\omega)\bigr) \leq 0
$$
holds for $\omega>0\,$. To that end, note that, for $\beta\in\bigl(
0,\beta_1(x_0)\bigr)$, it clearly holds that $-\frac{1}{\beta} <
-\frac{1}{\beta_{1}(x_0)}$ and, therefore, making use of
\eqref{ExplF}, we see that 
$$
F_{q(x_{0}),\beta}\bigl(\Gamma_{x_0}(\omega)\bigr)<F_{q(x_{0}),\beta_{1}(x_0)}\bigl(
\Gamma_{x_0}(\omega)\bigr), \:\omega\in(0,\infty). 
$$
Hence it remains to prove that the inequality
$$
F_{q(x_{0}),\beta_{1}(x_0)}\bigl(\Gamma_{x_0}(\omega)\bigr)\leq 0
$$
is satisfied for $\omega\in(0,\infty)$. To do so, we
first use \eqref{Explq} to get 
$$
\frac{d}{d\omega}\Big | _{\omega=\omega_1(x_0)}
F_{q(x_{0}),\beta_{1}(x_0)}\bigl( \Gamma_{x_0}(\omega)\bigr)=0, 
$$
and
$$
\frac{d^{2}}{d^{2}\omega} \Big | _{\omega=\omega_1(x_0)}
F_{q(x_{0}),\beta_{1}(x_0)}\bigl( \Gamma_{x_0}(\omega)\bigr)<0. 
$$
From \eqref{ExplF} we also see immediately that 
$$
\lim_{\omega\to 0}F_{q(x_{0}),\beta_{1}(x_0)}\bigl(\Gamma_{x_0}(\omega)\bigr)=-\infty
$$ 
and
$$
\lim_{\omega\to\infty}F_{q(x_{0}),\beta_{1}(x_0)}\bigl(\Gamma_{x_0}(\omega)\bigr)
=-\frac{1}{q(x_0)\,\beta_1(x_0)}< 0. 
$$
This shows that $F_{q(x_{0}),\beta_{1}(x_0)}\bigl(
\Gamma_{x_0}(\omega)\bigr)$ achieves its maximum in the interior of a
compact subset of $(0,\infty)$. Now we prove the assertion by showing that 
$$
\max_{\omega>0} F_{q(x_{0}),\beta_{1}(x_0)}\bigl( \Gamma_{x_0}(\omega)\bigr)\leq 0\;.
$$
We do this by verifying that $F_{q(x_{0}),\beta_{1}(x_0)}\bigl(
\Gamma_{x_0}(\omega)\bigr)$ is nonpositive in any of its critical
points. In other words, we show that for any $\widehat\omega>0$ where 
$$
\frac{d}{d\omega}\Big |_{\omega=\widehat\omega}F_{q(x_{0}),\beta_{1}(x_0)}\bigl(
\Gamma_{x_0}(\omega)\bigr)=0,
$$
it holds that
$F_{q(x_{0}),\beta_{1}(x_0)}\bigl(\Gamma_{x_0}(\widehat\omega)\bigr)\leq0$.
An elementary differentiation shows that a critical point $\widehat\omega$
needs to be a solution of
\begin{equation}\label{CritOmegaCond}
  \tan\bigl({y(\widehat\omega)}\bigr)=T\bigl(y(\widehat\omega)\bigr),\:
  y(\widehat\omega):=x_0\sqrt{\widehat\omega/2}, 
\end{equation}  
where, for $y>0$, the function $T$ is given by 
\begin{equation}\label{T_fun}
T(y):= \frac{y^2 - d_\pi y -\frac{d_\pi}{2}}{2y^3 - y^2 -\frac{d_\pi}{2}}.
\end{equation}
The cubic polynomial in the denominator of $T$ has only one real root
$y_s\approx 1.4399094$, which generates a real pole of $T$.  
A discussion of the graph of $T(y)$ for $y\in(0,y_s)$ and
$y\in(y_s,\infty)$ and the fact that $y_s<\pi/2$ yield that all
positive solutions of $\tan(y)=T(y)$ satisfy $y>y_s$. 
Hence any critical point $\widehat\omega>0$ of
$F_{q(x_{0}),\beta_{1}(x_0)}\bigl( \Gamma_{x_0}(\omega)\bigr)$
enjoys the relationship 
$$
\sin\bigl( y(\widehat\omega) \bigr) =
T\bigl( y(\widehat\omega) \bigr)\cos\bigl( y(\widehat\omega)
\bigr)\,\text{ and } y(\widehat\omega)>y_s.  
$$
Inserting this into the expression \eqref{ExplF}, the verification
that $F_{q(x_{0}),\beta_{1}(x_0)}\bigl(
\Gamma_{x_0}(\widehat\omega)\bigr)\leq 0$ is easily seen to be equivalent
to the verification that, for $y>y_s$,
$$
2\cos(y)y^2\bigl[ 1+T(y) \bigr]+d_\pi\cos(y)\bigl[ 1-T(y) \bigr]+
\frac{d_\pi}{c_\pi}ye^y\geq 0.
$$
This follows by plotting this function or by an analytic discussion
using the fact that the expression in the left-hand side of the above
inequality vanishes at
$y=\sqrt{\omega_1(1)/2}=\sqrt{b_\pi/2}=\frac{3\pi}{4}\,$. 
\end{proof} 

\subsection{The Popov criterion for the Dirichlet problem}
In contrast to the case for $L=\infty$ just discussed in Proposition
\ref{PopCritLim}, a direct rigorous verification of the Popov criterion for the
transfer function $G_{L,x_0}$ associated with the Dirichlet problem is
more involved analytically. To simplify the discussion in the case of
Dirichlet boundary conditions for a fixed but arbitrary $L>0$, we
may assume without loss of generality that $L=1$ by rescaling units of
length. For notational convenience we parametrize the location of
$x_0$ as 
$$
x_0(\delta):= 1-\delta,\text{ for } \delta\in(0,1).
$$
This leads to considering the one-parameter family of transfer functions 
$$
  G_{\delta}(s):=G_{L=1,x_0(\delta)}(s)=\frac{\sinh(\delta\sqrt{s})}{2\sqrt{s}
  \cosh(\sqrt{s})},\:\delta\in(0,1),
$$ 
and its associated one-parameter family of Popov curves 
$$
 \Gamma_{\delta}(\omega):= \bigl( x(\omega),y(\omega)\bigr)= 
 \Bigl( \operatorname{Re}\bigl[ G_{\delta}(i \omega)\bigr],\:
 \omega\,\operatorname{Im}\bigl[ G_{\delta}(i \omega)\bigr]\Bigr),\:
 \omega\in(0,\infty)\,,\, \delta\in(0,1). 
$$
In order to more conveniently deal with the limits as $\delta \to 0$
and as $\delta \to 1$ we consider $\widetilde{G}_\delta:=G_{\delta}/\delta$,
so that
$$
 \widetilde{G}_0(s)=\lim_{\delta\to 0}\widetilde{G}_\delta(s)
 =\frac{1}{2\cosh(\sqrt{s})}
$$
and
$$
 \widetilde{G}_{1}(s)=\lim_{\delta\to 1}\widetilde{G}_\delta(s)
 =\frac{\tanh(\sqrt{s})}{2\sqrt{s}}=G_1(s).
$$
Also note that, since 
$$
\lim_{\omega\to 0}
 \frac{\sinh(\sqrt{i\omega})}{\sqrt{i\omega}\cosh(\sqrt{i\omega})}
 =\lim_{\omega\to 0}\frac{1}{\cosh(\sqrt{i\omega})}=\lim_{\omega\to 0}
 \frac{\tanh(\sqrt{i\omega})}{\sqrt{i\omega}}=1,
$$ 
we find that
$$
 \lim_{\omega\to 0} {\widetilde{G}}_{\delta}(i\omega)=
 \frac{1}{2}\; \text{ and }\lim_{\omega\to 0} {G}_{1}(i\omega)=\frac{1}{2}\;.
$$
We denote the corresponding asymptotic (rescaled for $\delta\to 0$)
Popov curves accordingly by $\widetilde{\Gamma}^{0}(\omega)$ and $\Gamma^{1}(\omega)\,$. 
Without giving a proof, we note that
$$
\tilde \Gamma _\delta(\omega) =\Gamma_\delta(\omega)/\delta\longrightarrow
\widetilde{\Gamma}^0\text{ as }\delta\to 0 \text{ uniformly in } [0,\infty)  
$$
and
$$
\Gamma_\delta(\omega) \to \Gamma^1(\omega)\text{ as }
\delta\to 1\text{ uniformly in intervals of the form }(0,M)\,.
$$
Similarly as in the proof of Proposition \ref{PopCritLim} for the case
$L=\infty$, the relevant parameters that determine the stability and
bifurcation properties associated with $\Gamma_{\delta}(\omega)$ for
$\delta \simeq 0$ can be determined explicitly by studying the
corresponding properties of its rescaled limit
$\widetilde{\Gamma}^0(\omega)$.

\begin{rem}\label{explicit_small_delta}
The Popov set of $\widetilde{\Gamma}^0(\omega)$ is given by  
$$
 \widetilde{\Omega}^{\text{Pop}}_0:=\Big\{\omega>0 \,\Big |\,
 \operatorname{Im}\bigl( \widetilde{\Gamma}^0(\omega)
 \bigr)=0\Big\}=\big\{\omega_k=2k^2\pi^2 \,\big |\,
 k=1,2,\dots\big\}.
$$
In particular, the first intersection of $\widetilde{\Gamma}^0$
with the real axis occurs at the frequency
$$ 
\omega_1^0:= \min \widetilde{\Omega}^{\text{Pop}}_0=2\pi^2,
$$
and the corresponding period is given by
$T_1^0:=\frac{2\pi}{\omega_1^0}=\frac{1}{\pi}$. We obtain the
critical parameter for $G_{\delta}$ in the limit as $\delta\to 0$ from
the value of $\operatorname{Re}\Bigl(\widetilde{\Gamma}^0\bigl(
\omega_1^0\bigr)\Bigr) =\operatorname{Re}\Bigl(
\widetilde{G}_0(i\omega_1^0)\Bigr)$, i.e. 
$$ 
 \beta_1^0:=-\frac{1}{\delta \widetilde{G}_0(i\omega_1^0)}
 =\frac{e^{\pi}+e^{-\pi}}{\delta}. 
 $$
Finally, the slope of $\widetilde{\Gamma}^0$ at its first
intersection point with the real axis, which occurs at
$\omega=\omega_1^0$, can be determined explicitly. In fact, using the
parametrization $\widetilde{\Gamma}^0=\bigl(
x(\omega),y(\omega)\bigr)$ it holds that
$$
\frac{1}{q_{1}^0}=\frac{\dot{y}}{\dot{x}} \Big |
_{\omega=\omega_1^0} = 2\pi^2. 
$$
\end{rem}
\begin{proof}
The proof follows from somewhat lengthy but elementary calculations that begin
with splitting the function $\frac{1}{\cosh(\sqrt{i\omega})}$ into its real and
imaginary part. 
\end{proof}
The verification of the Popov criterion for given parameter values
$\beta>0$ and $\delta\in(0,1)$ can be interpreted geometrically. It
amounts to showing that it is possible to choose a straight line in the complex plane with
positive slope such that it intersects the real axis at $-\frac{1}{\beta}$ and such that 
the entire Popov curve lies to the right of that straight line. In our case,
the choice of a tangent to the Popov curve at its most negative
intersection point with the negative real axis is a possible choice of such a straight line. 
The choice of the tangent, as a particular separating straight line, corresponds to the critical parameter 
$\beta_1$ at which a change of stability takes place and this choice leads to a ``maximal'' interval of stability $(0,\beta_1)\,$.  
In applied problems, e.g. in electrical engineering, the verification of the Popov
stability criterion is often simply reduced to plotting the Popov curve and to checking
whether such a tangential (optimal) line, or any separating line, can
be fitted into the Popov plot. 
In Figure \ref{PopovPlots} we plot the rescaled Popov curves
$\widetilde{\Gamma}_{\delta}(\omega)$ for different choices of $\delta$. 
The two asymptotes $\Gamma^{1}(\omega)$, which is confined to
the right complex halfplane, and $\widetilde{\Gamma}^0(\omega)$, 
which orginates at $(\frac{1}{2},0)$ and spirals to the origin as $\omega\to \infty\,$, 
are both depicted as dotted lines.\\
\begin{figure}[H]
    \centering
    {\includegraphics[scale=.5]{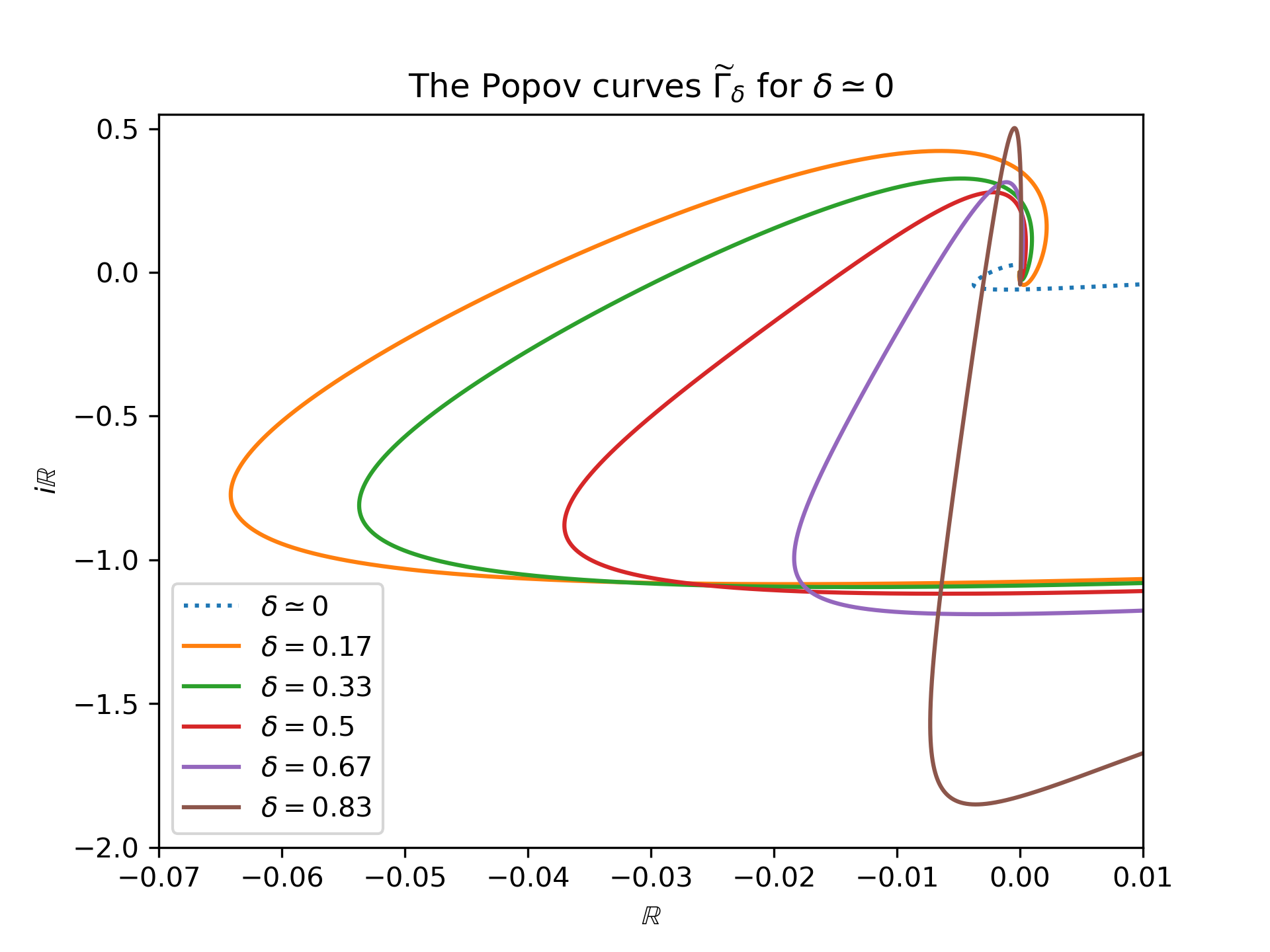} }%
    {\includegraphics[scale=.5]{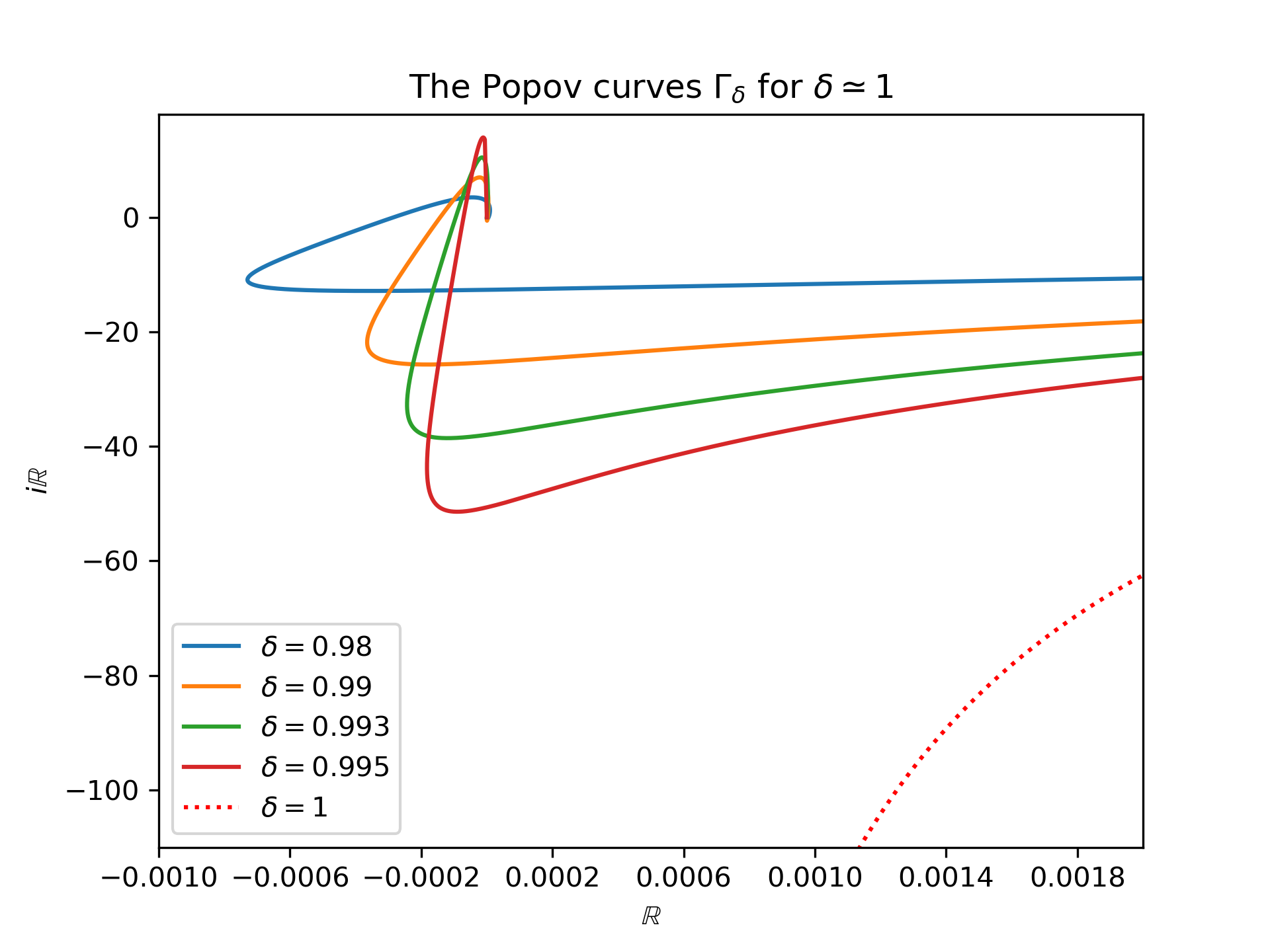} }%
    \caption{The Popov curves close to the limiting cases $\delta=0,1$.}%
    \label{PopovPlots}
\end{figure}
The shape of the Popov curves found by the parameter study shown in
Figure \ref{PopovPlots} suggests that to each $\delta\in(0,1)$ we can
associate the uniquely determined line in $\mathbb{R}^2 \cong
\mathbb{C}$ that is tangent to $\Gamma_{\delta}(\omega)$ at its most
negative intersection point with the real axis, i.e. where
$$
\operatorname{Im}\Bigl( G_{\delta}\bigl(i\omega_{1}(\delta)\bigr)\Bigr)=0.
$$ 
That line is obviously given by 
$$
F_{q(\delta),\beta(\delta)}(x,y):=y-\frac{1}{q(\delta)}x
-\frac{1}{q(\delta)\,\beta(\delta)}=0,
$$
where, using the coordinates
\begin{equation}\label{Popcurvedelcoord}
\Gamma_{\delta}(\omega):= \bigl( x(\omega),y(\omega)\bigr):=
\Bigl( \operatorname{Re}\bigl( G_{\delta}(i\omega)\bigr),\:
\omega\,\operatorname{Im}\bigl( G_{\delta}(i\omega)\bigr)\Bigr),
\end{equation}
we set
$$
\beta(\delta):=-\frac{1}{G_{\delta}(i\omega_1(\delta))} \text{ and } 
\frac{1}{q(\delta)}:=\frac{\dot{y}}{\dot{x}} \Big | _{\omega=\omega_1(\delta)}.
$$
In spite of this numerical graphical evidence, which shows the existence of an
optimal straight line satisfying the Popov criterion up to the
maximal choice for the constant
$\hat{\beta}_{1}(x_0,L)=-1/G_{L,x_0}\bigl(i\,\omega_1(L,x_0)\bigr)$,  
we chose to state Theorem \ref{vie_L_result} in a weaker form that
does not rely on any numerical or graphical verification.
\subsection{Numerical verification of the Popov criterion for
  $\beta\in\bigl(0,\beta(\delta)\bigr]$ in the case $L<\infty$}
Before giving the proof of Theorem \ref{vie_L_result} we discuss how   
a numerical verification of the Popov criterion can be performed to
see that $\bigl( 0,\beta(\delta)\bigr]$ is the maximal interval of
global stability for the trivial equilibrium of \eqref{lineTs}.
Here we again rescale units of length so that for $\delta\in (0,1)$ we
can consider the transfer function
$$
  G_{\delta}(i\omega)=\frac{\sinh(\delta\sqrt{i\omega})}{2\sqrt{i\omega}
  \cosh(\sqrt{i\omega})},\:\omega\geq 0.
$$ 
While we proceed in the spirit of the proof of Proposition
\ref{PopCritLim}, we need to resort to numerical computations to check
the sign of the resulting elementary function. In order to express the
imaginary and the real part of $G_{\delta}(i\omega)$ explicitly in a
concise manner  we set
$$
  A_1:=A_1(\delta,\omega)=\cosh\bigl(\delta\sqrt{\omega/2}\bigr)\sin\bigl(\delta\sqrt{
  \omega/2}\bigr)\text{ and }
  A_2:=A_2(\delta,\omega)=\cos\bigl(\delta\sqrt{\omega/2}\bigr)\sinh\bigl(\delta\sqrt{ 
  \omega/2}\bigr)
$$
as well as 
\begin{align}
  B_1&:=B_1(\omega)=\cos\bigl(\sqrt{\omega/2}\bigr)\cosh\bigl(\sqrt{\omega/2}\bigr)+
  \sin\bigl(\sqrt{\omega/2}\bigr)\sinh\bigl(\sqrt{\omega/2}\bigr)\\
B_2&:=B_2(\omega)=\cos\bigl(\sqrt{\omega/2}\bigr)\cosh\bigl(\sqrt{\omega/2}\bigr)-
  \sin\bigl(\sqrt{\omega/2}\bigr)\sinh\bigl(\sqrt{\omega/2}\bigr)
\end{align}
and 
$$
D:=D(\omega)=\sqrt{2\omega}\Big[\cos(\sqrt{2\omega}) + \cosh(\sqrt{2\omega})\Bigr].
$$
The one can write
$$
\operatorname{Re}\bigl[ G_{\delta}(i \omega)\bigr] = \frac{1}{D}
\langle A,B\rangle,\:
\operatorname{Im}\bigl[ G_{\delta}(i \omega)\bigr] = \frac{1}{D} \det (A,B),
$$
with
$$
\langle A,B\rangle :=A_1B_1 + A_2B_2   \text{ and }  \det (A,B):=A_1B_2 - A_2B_1. 
$$
Using the coordinate representation \eqref{Popcurvedelcoord} of the
Popov curve, an explicit representation of
$$
\frac{1}{q(\delta)}:=\frac{\dot{y}}{\dot{x}} \Big | _{\omega= \omega_1(\delta)}
$$
can be found in the form  
\begin{equation}\label{slopeDiriexplicit}
\frac{1}{q(\delta)} = \omega_1(\delta)\Big\{\frac{\det(\dot{A},B)+\det(A,\dot{B})}
{\langle\dot{A},B\rangle+\langle A,\dot{B}\rangle - \langle
  A,B\rangle\dot{D}/D}\Big\}\Big | _{\omega= \omega_1(\delta)}.
\end{equation}
The dotted quantities are differentiated with respect to $\omega$ and
evaluated at $\omega_1(\delta)$. This representation is not fully
explicit since a numerical root finding procedure needs to be used in order to
locate the first positive solution $\omega_1(\delta)$ of
$\operatorname{Im}\bigl[ G_{\delta}(i \omega)\bigr] = \frac{1}{D} \det
(A,B)=0$. In principle, for any given $\delta\in(0,1)$, the zero
$\omega_1(\delta)$ can be determined with arbitrary (finite)
precision. Therefore, for each $\delta\in(0,1)$, the verification of
the Popov criterion
\begin{equation}
F_{q(\delta),\beta(\delta)}(\omega):=y(\omega)-\frac{1}{q(\delta)}x(\omega)
-\frac{1}{q(\delta)\,\beta(\delta)}\leq 0 \text{ for } \omega >0, 
\end{equation}
up to the numerical determination of $\omega_1(\delta)$, consists in
verifying that the following combination of the elementary functions
$A_i,B_i,D$ is nonpositive, i.e., that   
\begin{equation}\label{PopovLfinite}
F_{q(\delta),\beta(\delta)}(\omega)= 
\omega\det(A,B)/D-\frac{1}{q(\delta)}\Bigl[\langle A,B \rangle/D-  
\big\langle A\bigl(\omega_1(\delta)\bigr),B\bigl(\omega_1(\delta)\bigr)\big\rangle/D
\bigl(\omega_1(\delta)\bigr)\Bigr]\leq 0\text{ for }\omega > 0,
\end{equation}
where $\frac{1}{q(\delta)}$ is given by \eqref{slopeDiriexplicit}. It
is clear by the definition of $F_{q(\delta),\beta(\delta)}$, as well as
directly by inspection of the above formula, that
$F_{q(\delta),\beta(\delta)}\bigl(\omega_1(\delta)\bigr)=0$ which
reflects the tangency condition. The verification of condition
\eqref{PopovLfinite} can thus be performed by evaluation of the above
expression over a finite range for $\omega$. This follows from the
fact that we know that the curve $\Gamma_{\delta}(\cdot)$ spirals into
the origin of the complex plane exponentially fast. 
Clearly the statement of nonpositivity requires a parametric study for
$\delta$ in $(0,1)$. It can be verified analytically that
$$
\lim_{\delta \to 0} \omega_1(\delta)q(\delta)=1 \text{ and }
\lim_{\delta \to 1}\omega_1(\delta)q(\delta)=\frac{b_{\pi}}{d_{\pi}}
=\frac{3\pi+4}{3\pi}. 
$$
Thus the limit $\delta\to 1$ corresponds to the case $L=\infty$, which
is intuitively clear. In fact, observe that, owing to
\eqref{omega1explicitinfty} and to \eqref{Explq}, the product  
$$
\omega_1(x_0)q(x_0)=\frac{b_{\pi}}{d_{\pi}}=\frac{3\pi+4}{3\pi},
$$
is an invariant of the one-parameter family of Popov curves
$\big\{ \Gamma_{x_0}\, \big |\, x_0>0\big\}$. By contrast, for
$L<\infty$, the product $\omega_1(\delta)q(\delta)$ depends on $\delta$
but has the two known limits given above. \\
Based on the parameter study in Figure \ref{popovCrit}, we formulate
the following conjecture. In order to safeguard rigour we are,
somewhat reluctantly, forced to formulate the numerical result merely as a
conjecture, since it must be conceded that any parameter study cannot
replace a rigorous proof of the validity  of \eqref{PopovLfinite} for
arbitrary $\delta\in (0,1)$ in spite of the fact that the criterion
could be checked up to arbitrary finite precision for any given
specific $\delta\in (0,1)$.

\begin{figure}
  \centering
  \includegraphics[scale=.6]{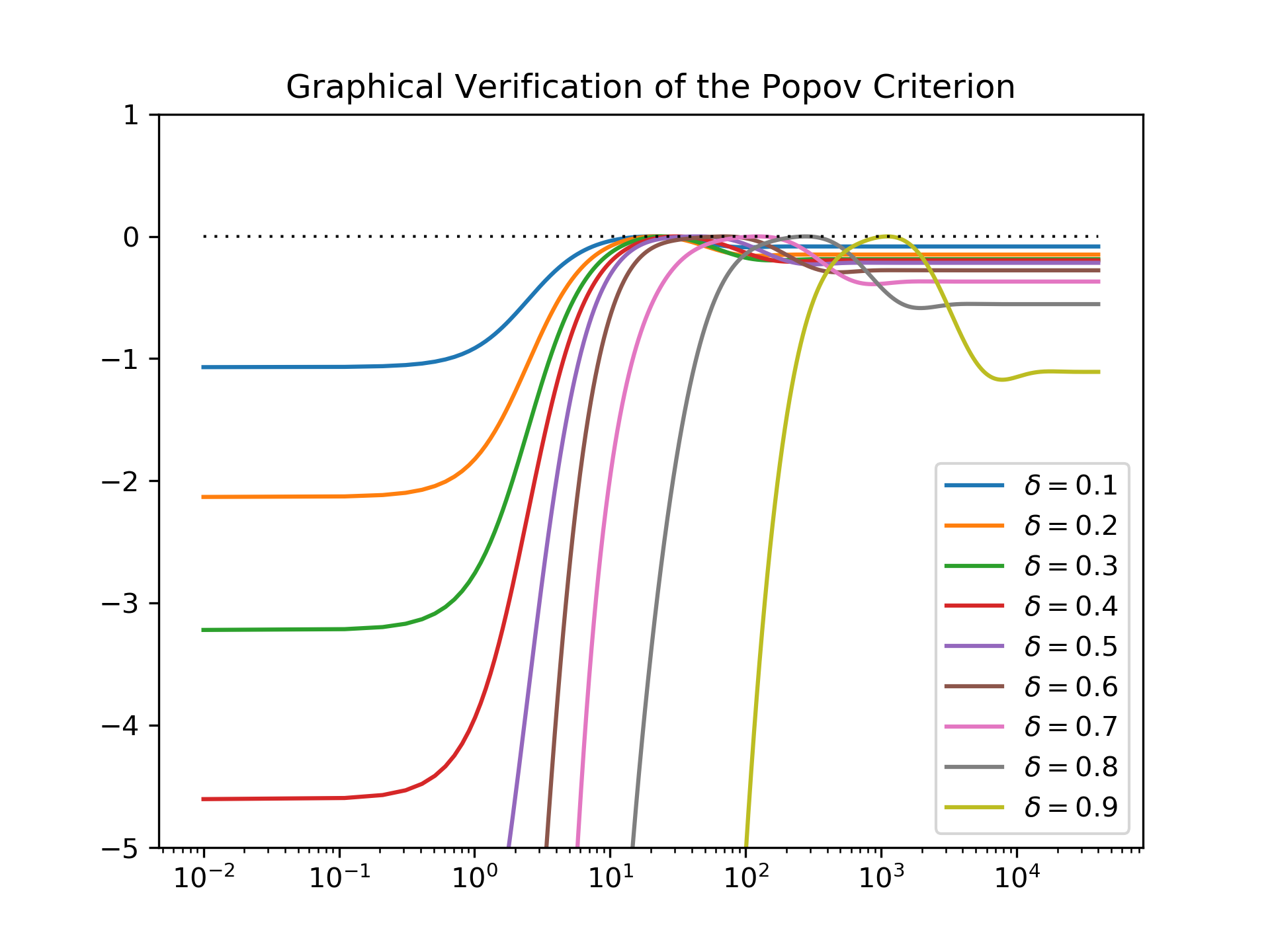}
  \caption{Depicted is the function appearing in the Popov criterion
    \eqref{PopovLfinite} for various values of $\delta\in(0,1)$}
  \label{popovCrit}
\end{figure}

\begin{conj}\label{PopDiriConj}
For any $\delta\in(0,1)$, let
$\beta(\delta):=-1/G_{\delta}\bigl( i\omega_1(\delta)\bigr)$. Then,
for any $\beta\in\bigl(0,\beta(\delta)\bigr]$, the pair $\beta$ and
$q(\delta)>0$, given by \eqref{slopeDiriexplicit}, satisfies the Popov
criterion, i.e. the inequality
\begin{equation}\label{Pdelta}
  \operatorname{Re}\bigl[ G_{\delta}(i \omega)\bigr]-q(\delta) \omega 
  \operatorname{Im}\bigl[ G_{\delta}(i \omega)\bigr] \geq -\frac{1}{\beta} 
\end{equation}
for all $\omega\in \mathbb{R}\setminus\{0\}$. 
\end{conj}

We finally prove the theorem as it was formulated at the beginning of
this section, i.e. without making any reference to the conjecture above. 

\subsection*{Proof of Theorem \ref{vie_L_result}}
The proof relies on the application of the criterion derived in
Proposition \ref{Prop_stab_vie}. For arbitrary fixed $L>x_0>0$ we look
for a parameter value $\widehat{\beta}_1(x_0,L)>0$ such that, for any
$\beta\in\bigl(0,\widehat{\beta}_1(x_0,L)\bigr)$, it is possible to
find $q(x_0)>0$ such that
$$
 \widehat{a}_L(\omega)+q(x_0)\,\widehat{a_L'}(\omega)-\frac{1}{\beta} \leq 0
 \text{ for } \omega>0. 
$$
This then entails that all solutions of the Volterra integral equation
\eqref{vie_L} with arbitrary parameters $x_0>0$ and $L>x_0$ converge
to zero as $t\to\infty\,$ as long as
$\beta\in\bigl(0,\widehat{\beta}_1(x_0,L)\bigr)$. If this is the case,
we call $\bigl(0,\widehat{\beta}_1(x_0,L)\bigr)$ an interval of stability 
for the integral equation with parameters $L>x_0>0$. We define the
Popov curve associated with $x_0$ and $L$ by
$$
\Gamma_{x_0,L}(\omega)= \bigl( x(\omega),y(\omega)\bigr)=
\Bigl( \operatorname{Re}\bigl[ G_{x_0,L}(i \omega)\bigr],\:
\omega\,\operatorname{Im}\bigl[ G_{x_0,L}(i \omega)\bigr]\Bigr).
$$
By introducing the functional 
$$
F_{q,\beta}(x,y)=y-\frac{x}{q}-\frac{1}{q\,\beta}\,,
$$
the verification of the stability criterion reduces to showing that
suitable choices of the parameters $\beta$ and $q$ lead to
$$
F_{q,\beta}\bigl(\Gamma_{x_0,L}(\omega)\bigr) \leq 0,\:\omega>0.
$$
For arbitrary $x_0>0$ and $L>x_0$ we set 
\begin{equation}\label{beta1hatL}
\widehat{\beta}_1(x_0,L)= \frac{1}{M(x_0,L)\,q(x_0)} 
\end{equation}
where
$$
q(x_0)=\frac{x_0^2}{d_\pi},\,d_\pi=\frac{9\pi^3}{8\pi+\frac{32}{3}}
$$
and
$$
M(x_0,L):=\max\limits_{\omega > 0} \Bigl\{
\omega\,\operatorname{Im}\bigl[ G_{x_0,L}(i \omega)\bigr]
- \frac{1}{q(x_0)}\operatorname{Re}\bigl[ G_{x_0,L}(i \omega)\bigr]
\Bigr\}.
$$
To show that the above maximum exists and that it is positive, observe
that, by definition, the Popov set $\Omega^{\text{Pop}}_{L,x_0}$
contains a minimal element $\omega_1>0$ such that 
$$
\operatorname{Im}\bigl[ G_{x_0,L}(i \omega_1)\bigr]=0 \text{ and
}\operatorname{Re}\bigl[ G_{x_0,L}(i \omega_1)\bigr]<0.
$$ 
This shows that $\widehat{\beta}_1(x_0,L)>0$, if the maximum exists. In
order to obtain the existence of the maximum a simple calculation
yields that
\begin{align}\label{ExplFLx0}
  H(\omega)&:=\omega\,\operatorname{Im}\bigl[ G_{x_0,L}(i \omega)\bigr]
  - \frac{1}{q(x_0)}\operatorname{Im}\bigl[ G_{x_0,L}(i \omega)\bigr]=\notag\\
 &\frac{\sqrt{\omega}}{2} \operatorname{Im} \Bigl[
   \;\overline{\sqrt{i}}\,
   \big\{\cosh(x_0\sqrt{i\omega})\tanh(L\sqrt{i\omega}) -
   \sinh(x_0\sqrt{i\omega})\big\} \Bigr]  \\ 
   &-\frac{d_\pi}{2\,x_0^2\,\sqrt{\omega}}\operatorname{Re} \Bigl[
   \;\overline{\sqrt{i}}\,          
  \big\{\cosh(x_0\sqrt{i\omega})\tanh(L\sqrt{i\omega}) -
   \sinh(x_0\sqrt{i\omega})\big\}\Bigr]. \notag  
\end{align}
Then notice that $\lim_{\omega\to 0^+} H(\omega)$ exists and that
$$
\lim_{\omega\to 0^+} H(\omega)\;<0.
$$
One also has that 
$$
\lim_{\omega\to \infty} H(\omega)=0,
$$
which is verified by using  
$$
\lim_{\omega\to \infty}\tanh(L\sqrt{i\omega})= 1,
$$
and that
$$
\lim_{\omega\to \infty}
\sqrt{\omega}\Bigl[\cosh(x_0\sqrt{i\omega})\tanh(L\sqrt{i\omega}) -
\sinh(x_0\sqrt{i\omega})\Bigr] = 0. 
$$
To study $\lim_{\omega\to 0^{+}} H(\omega)$, one observes that 
$$
\lim_{\omega\to 0^+} \frac{\overline{\sqrt{i}}}{\sqrt{\omega}}
\Bigl[ \cosh(x_0\sqrt{i\omega})\tanh(L\sqrt{i\omega}) -
\sinh(x_0\sqrt{i\omega})\Bigr] >0.
$$
Since $H(\omega)$ is negative for sufficiently small arguments,
converges to zero as $\omega\to\infty$, and has positive values, the
maximum must be attained and be positive. Now for any
$\beta\in\bigl(0,\widehat{\beta}_1(x_0,L)\bigr)$, we obtain the estimate
$$  
F_{q(x_0),\beta}\bigl(\Gamma_{x_0,L}(\omega)\bigr)< 
F_{q(x_0),\widehat{\beta}_1(x_0,L)}\bigl(\Gamma_{x_0,L}(\omega)\bigr),\:\omega>0.
$$
It only remains to verify that 
$$  
F_{q(x_0),{\widehat\beta}_1(x_0,L)}\bigl(\Gamma_{x_0,L}(\omega)\bigr) =
H(\omega) - \frac{1}{q(x_0)\,\widehat{\beta}_1(x_0,L)} \leq 0,\:\omega>0,
$$
which follows from the definition of $\widehat{\beta_1}(x_0,L)$ since
$$
F_{q(x_0),\widehat{\beta}_1(x_0,L)}\bigl(\Gamma_{x_0,L}(\omega)\bigr)
\leq \max\limits_{\omega > 0} H(\omega) -
\frac{1}{q(x_0)\widehat{\beta}_1(x_0,L)} = M(x_0,L) -
\frac{1}{q(x_0)\,\widehat{\beta}_1(x_0,L)}=0. 
$$
Next we present the argument producing the alternative interval of
stability $\bigl( 0,\beta_1(x_0)\bigr)=(0,\frac{c_\pi}{x_0})$ for any
$x_0>0$ and all sufficiently large $L>x_0$. For fixed $x_0>0$, we have
shown in Remark \ref{Ltoinfinity} that
$$
G_{x_0,L}(i\omega)\to G_{x_0}(i\omega) \text{ as } L\to\infty,
$$ 
uniformly for $\omega$ in intervals of the form $(c,\infty)$ with
arbitrary $c>0$. This also implies that
$$
\Gamma_{x_0,L}(\omega) \to \Gamma_{x_0}(\omega) \text{ in } \mathbb{C}
\text{ as } L\to\infty,
$$
and that
$$
F_{q(x_0),\beta_1(x_0)}\bigl(\Gamma_{x_0,L}(\omega)\bigr) \to
F_{q(x_0),\beta_1(x_0)}\bigl(\Gamma_{x_0}(\omega)\bigr) \text{ as }
L\to\infty,
$$ 
uniformly for $\omega$ in $(c,\infty)$ with arbitrary
$c>0$. Hence, for any $\varepsilon>0$ and any $c>0$, there exists
$C(\varepsilon,c)>0$ such that 
$$
F_{q(x_0),\beta_1(x_0)}\bigl(\Gamma_{x_0,L}(\omega)\bigr) 
\leq 
F_{q(x_0),\beta_1(x_0)}\bigl(\Gamma_{x_0}(\omega)\bigr) +
\varepsilon \text{ for }L\geq C(\varepsilon,c),
$$ 
and $\omega\in(c,\infty)$. For any $\beta\in \bigl( 0,\beta_1(x_0)
\bigr)$ we can choose $\delta(\beta):=\frac{\delta}{q(x_0)}\bigl[\frac{1}{\beta} -
\frac{1}{\beta_1(x_0)} \bigr]$, $\delta\in(0,1)$, to obtain
$$
 F_{q(x_0),\beta}\bigl(\Gamma_{x_0,L}(\omega)\bigr) + \delta(\beta)
 \leq F_{q(x_0),\beta_1(x_0)}\bigl(\Gamma_{x_0,L}(\omega)\bigr) \text{
 for } \omega>0. 
$$
Thus, for $L\geq C\bigl( \delta(\beta),c \bigr)$ and for
$\omega\in(c,\infty)$ we have that
$$
F_{q(x_0),\beta}\bigl(\Gamma_{x_0,L}(\omega)\bigr) + \delta(\beta) \leq 
F_{q(x_0),\beta_1(x_0)}\bigl(\Gamma_{x_0,L}(\omega)\bigr)\leq 
F_{q(x_0),\beta_1(x_0)}\bigl(\Gamma_{x_0}(\omega)\bigr) + \delta(\beta).
$$
Consequently, by Proposition \ref{PopCritLim}, it holds that
$$
F_{q(x_0),\beta}\bigl(\Gamma_{x_0,L}(\omega)\bigr) \leq
F_{q(x_0),\beta_1(x_0)}\bigl(\Gamma_{x_0}(\omega)\bigr) \leq 0
$$
for $L\geq C\bigl( \delta(\beta),c \bigr)$ and
$\omega\in(c,\infty)$. Since we know from the first part of the proof
that
$$
\lim_{\omega\to 0^+}
F_{q(x_0),\beta}\bigl(\Gamma_{x_0,L}(\omega)\bigr) <
-\frac{1}{\beta q(x_0)}<0,  
$$
and $c>0$ can be chosen arbitrarily small, the inequality
$$
F_{q(x_0),\beta}\bigl(\Gamma_{x_0,L}(\omega)\bigr)\leq 0, 
$$
holds for $\omega\in(0,\infty)$ and $L\geq C(x_0)$ for a large
enough $C(x_0)>0$.
\hfill{\qed}


\section{Global stability and Hopf bifurcation results for the nonlinear PDE}\label{SecHopfRes}
The stability result obtained in the previous section for the Volterra
integral equation will now be applied to the nonlinear partial
differential equation \eqref{lineTs}. It will be instrumental to
infer the decay of the solutions of the PDE from the decay of the
associated solutions of the integral equation. The following
proposition is proved in \cite[Proposition 2.3.]{GM20} in a slightly
different setting, yet its proof can readily be adapted to the present
situation. 

\begin{prop}\label{ConvEquiv_L}
For fixed parameters $\beta, L\in(0,\infty)$ and $x_0\in(0,L)$
consider orbits $\Phi_\beta(\cdot, u_0)$ of the semiflow
$\bigl(\Phi_\beta,\operatorname{H}_L^1\bigr)$ associated with
\eqref{lineTs}. Then, as $t\to\infty\,$, for any $u_{0}\in
\operatorname{H}_L^1$, it holds that
$$
 \Phi _\beta (t,u_0)\longrightarrow 0 \text{ in } \operatorname{H}_L^1
 \iff  \bigl( \Phi_\beta(t,u_0)\bigr)(x_0) \longrightarrow 0
 \text{ in } \mathbb{R}.
$$
\end{prop}

\begin{proof}
``$\Rightarrow$'': If $\Phi _\beta (t,u_0)\to 0$ as $t\to\infty$, then
the operation of ``taking the trace'' defines a bounded linear
operator $\gamma_{x_0}\in\mathcal{L}(\operatorname{H}_L^1,\mathbb{R})$  
and therefore its continuity implies 
$$
\gamma_{x_0} \bigl( \Phi_\beta(t,u_0)\bigr)=\bigl(
\Phi_\beta(t,u_0)\bigr)(x_0)\to 0 \text{ as } t\to\infty.
$$
  ``$\Leftarrow$'': If $y(t):=\bigl( \Phi_\beta(t,u_0)\bigr)(x_0) \to
  0$ as $t\to\infty$, then by \eqref{vie_L} we have that
  $$
  \lim _{t\to\infty} y(t) = \lim _{t\to\infty}\Bigl[ g_L(t)+\int _0^t
  a_L(t-\tau)f\bigl( \beta u(\tau,x_0)\bigr)\, d\tau \Bigr] \,= \,0,
  $$
which entails 
$$
\lim _{t\to\infty}\int _0^t a(t-\tau)f\bigl( \beta u(\tau,x_0)\bigr)\,
d\tau= 0,
$$
since we know by the properties of the semigroup $e^{-tA_L}$ that
$\lim _{t\to\infty} g_L(t)= 0$. Next notice that, for arbitrary $x\in
(-L,L)$, it holds that
$$
u(t,x):=\Phi _\beta (t,u_0)(x) = \bigl( e^{-tA_L}u_0 \bigr)(x) + \int
_0^t k_L(t-\tau,x)  f\bigl(\beta u(\tau,x_0)\bigr)\, d\tau,
$$
where, by definition \eqref{aRep}, it holds that 
$$
k_L(t,x)=-a_L(t,x)=-\bigl( e^{-tA_L}\delta _0 \bigr)(x).
$$
Again from $\lim _{t\to\infty} \bigl( e^{-tA_L}u_0 \bigr)(x)=0$ we
obtain
\begin{equation}\label{aux_1}
\lim _{t\to\infty} u(t,x)= - \lim _{t\to\infty} \int _0^t
k_L(t-\tau,x)  f\bigl(\beta u(\tau,x_0)\bigr)\, d\tau,\: x\neq x_0.
\end{equation}
Since we know by assumption that $\lim _{t\to\infty} f\bigl(\beta
u(\tau,x_0)\bigr)=0$ and, by inserting the spectral decomposition
\eqref{Fundamental_L} of $k_L$ in \eqref{aux_1}, we conclude similarly
as in \cite{GM20} that for arbitrary $x\in (-L,L)$ and $t\to\infty$
$$
\lim _{t\to\infty} u(t,x)=0\;, 
$$
i.e. we obtain the pointwise convergence of $\Phi _\beta (t,u_0)$ to
the zero function.\\
To prove that convergence to zero also occurs in the topology of
$\operatorname{H}_L^1$ we use \eqref{vcfEq} to derive the equation
satisfied by $\hat u_n(t)$, which is the $n$-th coefficient in the
spectral basis expansion of the solution 
$$
  u(t,x)=\sum_{k=1}^{\infty} \langle u(t,\cdot),\varphi_k \rangle
  \varphi_{k}(x) = \sum_{k=1}^{\infty} \hat u_k(t) \varphi_{k}(x).
$$
Observe that the $\operatorname{H}^1$ norm of a function $u(t,\cdot)$
of $x$ obtained for fixed $t$ is equivalent to
$$
 \| u(t,\cdot) \| _{\operatorname{H}_L^1}^2=\sum _{k=1}^\infty
 (1+k^2)|\hat u_k|^2.
$$
This is seen by extending $u(t,\cdot)$ to a periodic function ${\tilde
  u}(t,\cdot)$ by reflection as described in \eqref{perExtension} and 
noticing the direct relation between the standard Fourier series of 
$\tilde u$ and the spectral basis expansion of $u\,$. 
We also use the fact that $u\in\operatorname{H}_L^1$ if and only if
$\tilde u\in\operatorname{H}^1_\pi(-2L,2L)$, where the index indicates
periodicity.\\   
Next look at the evolution of the single modes of the solution, which
is determined by
$$
 \hat u_n(t)=e^{-t \lambda_{L,n}^2}\, \langle \hat u_0,\varphi_{L,n}
 \rangle -\int_0^t f\bigl(\beta
 u(\tau,x_0)\bigr)e^{-(t-\tau)\lambda_{L,n}^2}\, d\tau.
$$
A simple calculation exploiting the boundedness of $f$ then yields
$$
 (1+n^2)\big | \hat u_n(t)\big |^2\leq c(1+n^2) |\hat u_{0n}|+
 \frac{c}{\lambda_{L,n}^4}(1+n^2),\: n\geq 1.
$$
This, together with the fact that $u_0\in\operatorname{H}^1$ and that
$\lambda_{L,n}^4\sim n^4$ as $n\to\infty$, implies 
that the series 
\begin{equation}\label{seriesh1}
    \sum_{n\geq 1}(1+n^2)|\hat u_n(t)|^2
\end{equation}
converges uniformly in $t\geq 0$. This shows that the tail of the
Fourier representation of the solution can be made smaller than any
given $\varepsilon>0$ independently of $t\geq 0$. For the remaining
finitely many terms, a direct estimate of the integral yields
smallness. It namely follows from the solution representation that
\begin{align*}
  |\hat u_n(t)| &\leq e^{-t \lambda_{L,0}t}|\widehat{u_0}_n|+
  \int_0^{t_\varepsilon} e^{-(t-\tau)\lambda_{L,0}^2}\, d\tau+
  \max_{\tau\geq t_\varepsilon} |f\bigl(\beta u(\tau,x_0)|
  \int_{t_\varepsilon}^t e^{-(t-\tau)\lambda_{L,0}^2}\, d\tau\\
  &\leq e^{-(t-t_\varepsilon)\lambda_{L,0}^2}\Big\{ |\widehat{u_0}_n|
  +\frac{1}{\lambda_{L,0}^2}\Big\}+c \max_{\tau\geq t_\varepsilon}
    |f\bigl(\beta u(\tau,x_0)|\leq \varepsilon
\end{align*}
for $t$ large enough. This shows that
$$
 \sum _{n\geq 1}(1+n^2)|\hat u_n(t)|^2\to 0\text{ as }t\to\infty,
$$
or, in other words that $u(t)\to 0$ in $\operatorname{H}^1$.
\end{proof} 

We can now proceed to summarize the main results of our analysis.

\begin{thm}\label{Main_Result_L}
For arbitrary choice of the paramters $L>x_0>0$ the following assertions hold:
\renewcommand{\labelenumi}{\roman{enumi})}
\begin{itemize}
\item [(i)]
There exsist $\varepsilon>0$ such that for any $\beta\in
\bigl( \beta_1(x_0,L),\beta_1(x_0,L) + \varepsilon \bigr)$, with 
$$
 \beta_1(x_0,L)=-\frac{1}{G_{L,x_0}\bigl(
 i\omega_1(L,x_0)\bigr)}\text{ and }\omega_1(L,x_0)=\min
 \Omega^{Pop}_{L,x_0}, 
$$  
the semiflow $\bigl(\Phi_\beta,\operatorname{H}_L^1\bigr)$ associated
with \eqref{lineTs}$_{L,x_0}$ possesses a non-trivial periodic orbit.
\item[(ii)]
For $\beta\in\bigl(0,\widehat{\beta_1}(x_0,L)\bigr)$ where
$\widehat{\beta}_1(x_0,L)>0$ is defined in \eqref{beta1hatL}, every
(semi-)orbit of the semiflow
$\bigl(\Phi_\beta,\operatorname{H}_L^1\bigr)$ associated with
\eqref{lineTs}$_{L,x_0}$ converges to zero as $t\to\infty$. 
\item[(iii)]
If we assume that $L\geq C(x_0)$ for some sufficiently large constant
$C(x_0)$, then for any
$\beta\in\bigl(0,\beta_1(x_0)\bigr)=\bigl(0,\frac{c_{\pi}}{x_0}\bigr)$, for
$c_\pi=\frac{3\pi}{\sqrt{2}}e^{\frac{3\pi}{4}}$, every (semi-)orbit of
the semiflow $\bigl(\Phi_\beta,\operatorname{H}_L^1\bigr)$ associated
with \eqref{lineTs}$_{L,x_0}$ converges to zero as $t\to\infty$. 
\end{itemize}
\end{thm}
\begin{proof}
The second and third assertion follows directly from our main result
on the Volterra integral equation, Theorem \ref{ConvEquiv_L}, and from
Proposition \ref{vie_L_result}.\\
The first statement is a consequence of the general results obtained in
\cite[Theorem 1]{Ama91},\cite[Theorem I.8.2.]{K12}, or \cite{S95}. 
They can be applied analogously as in \cite{GM97}. 
We have shown in Proposition \ref{spec} and in our discussion of the
Popov set that 
$$
\sigma \bigl( A_{\beta_{1}(x_0,L)}\bigr)\cap i\mathbb{R}=\big\{ \pm
i\omega_1(L,x_0) \big\}.
$$
The non-degeneracy condition for the crossing of the
imaginary axis by the complex conjugate pair of eigenvalues of the
operator $-A_{\beta}$ needs to be checked to conclude the proof. 
We need to verify that  
$$
\frac{d}{d\beta} \operatorname{Re} \bigl[\lambda(\beta) \bigr]\,\Big
|_{\beta=\beta_{1}(x_0,L)} > 0,
$$   
where, for some $\varepsilon>0$, there exists 
$$  
\lambda:\Bigl(\beta_{1}(x_0,L)-\varepsilon,\beta_{1}(x_0,L)+\varepsilon
\Bigr)\to\mathbb{C}\text{ with }\lambda\bigl(\beta_{1}(x_0,L)\bigr)
=+i\omega_1(L,x_0),
$$
i.e. a local parametrization of the eigenvalue's path as it crosses the imaginary axis in 
the complex upper halfplane as $\beta$ increases. 
This follows from Proposition \ref{betaCritical-Lcase} and Remark \ref{L_finite_Rem}.  
\end{proof}

\begin{rem}\label{nogapveri}
Supported by numerical evidence, we conjecture that the definition
\eqref{beta1hatL} leads to
$$
\widehat{\beta_1}(x_0,L)= \beta_1 (x_0,L)=-\frac{1}{G_{L,x_0}\bigl(
  i\,\omega_1(L,x_0)\bigr)}\text{ with }\omega_1(L,x_0)=\min
\Omega^{Pop}_{L,x_0}.
$$
We note that one inequality $\widehat{\beta_1}(x_0,L) \leq \beta_1
(x_0,L)$ follows from the knowledge that the asymptotic stability of
the equilbrium $u=0$ is lost at $\beta_1 (x_0,L)$ due to the Hopf
bifurcation. Thus proving this conjecture reduces to showing
$\widehat{\beta_1}(x_0,L) \geq \beta_1 (x_0,L)$. This, in turn, would
follow if we could verify that the value of $M(x_0,L)$ in the
definition \eqref{beta1hatL} is achieved by $\omega_1(L,x_0)$ as a
maximizer. Clearly, if one were able to prove this statement then
Conjecture \ref{PopDiriConj} would no longer be needed. In that case
the parameter range of global stability
$(0,\widehat{\beta_1}(x_0,L)\bigr)$ would be maximal due to  
$\widehat{\beta_1}(x_0,L) = \beta_1 (x_0,L)$, i.e. the interval of
stability constructed for the Volterra integral equation then extends
up to the critical parameter value where the Hopf bifurcation occurs.
\end{rem}

To discuss the stability of the bifurcating periodic solutions for the
one-parameter family of semiflows
$\bigl(\Phi_\beta,\operatorname{H}_L^1\bigr),\:\beta>0$, observe
that the Ljapunov-Schmidt reduction used in \cite{Ama90b} to discuss
the Hopf bifurcation phenomenon in the finite dimensional case leads
to more precise statements about the structure of the bifurcating periodic
solutions. In particular, the local uniqueness of the bifurcating
solutions can be described in more detail and is made explicit in the
next remark. As highlighted in \cite{Ama90b} the Liapunov-Schmidt
reduction, which the author applies to ODEs, can often be extended
naturally to semi-flows in infinite dimensional phase spaces stemming
from reaction-diffusions problems. We refrain from executing that
approach here and refer to \cite{Ama91}. 
Instead we prefer to apply the results on the existence of a center
manifold in the infinite dimensional situation. In fact, our problem,
for $L<\infty$, formulated in the Sobolev space $\operatorname{H}_L^1$
falls into the rather general class of quasilinear parabolic systems
discussed in \cite{S95}. The possibility to restrict our semiflow to
its finite dimensional center manifold, allows us to discuss the
stability of the bifurcating solutions by studying the ODE that
governs the dynamics on the center manifold.

\begin{rem}\label{Precise_Hopf}
For the stability analysis of the bifurcating periodic solutions the results in 
\cite[Theorems 26.21 and 27.11]{Ama90b} 
provide a more precise description of the local structure at the
bifurcation locus. There exists $\varepsilon>0$ and a map
$$
\Big[ s \mapsto (u(s),T(s),\beta(s)) \Big] 
\in  C^{\infty}
\Big( 
(-\varepsilon,\varepsilon) \;,\;\delta\mathbb{B}_{H^1_L}(0)  
\times  \delta\mathbb{B}_{\mathbb{R}}(\frac{2\pi}{\omega_1(x_0,L)})  
\times \delta\mathbb{B}_{\mathbb{R}}(\beta_1(x_0,L)) 
\Big)
$$
with 
$$
\bigl(u(0),T(0),\beta(0)\bigr)=\bigl(0,\frac{2\pi}{\omega_1(x_0,L)},\beta_1(x_0,L)\bigr)
\in H^1_L \times(0,\infty)\times(0,\infty),
$$ 
for some suitably chosen factor $\delta >0$, that shrinks the open
unit balls appropriately. The above map has the property that, for
$0<s<\varepsilon$, the orbit of $u(s)$ under $\Phi_{\beta(s)}$ denoted by 
$$
\gamma(s):=\Big\{ \Phi_{\beta(s)}(t,u(s)) \,\big |\, t\geq 0 \Big\}
$$
is a noncritical periodic orbit of the semiflow
$\bigl(\Phi_{\beta(s)},\operatorname{H}_L^1\bigr)$ with period $T(s)$
passing through the point $u(s)\in \delta\mathbb{B}_{H^1_L}(0)$ and with
\begin{equation}\label{injective}
\gamma(s_1)\neq \gamma(s_2),
\end{equation}
for $0<s_1<s_2<\varepsilon$.
Every noncritical periodic orbit of the semiflow
$\bigl(\Phi_{\beta(s)},\operatorname{H}_L^1\bigr)$ in a sufficiently
small neighbourhood of
$\bigl(0,\frac{2\pi}{\omega_1(x_0,L)},\beta_1(x_0,L)\bigr)$ in the
cartesian product $H^1_L \times(0,\infty)\times(0,\infty)$ is
contained in the family
$$
\big\{ \gamma (s) \,\big |\, 0<s<\varepsilon \big\}.
$$
The map  
$$
\big[ s \mapsto \beta(s) \big] \,:\, (0,\varepsilon)
\rightarrow(0,\infty)
$$ 
is injective. This follows directly from \eqref{injective}, since
otherwise identical noncritical periodic orbits  for $s_1 \neq s_2$
could be obtained from $\beta(s_1)=\beta(s_2)$ and the identity of the
semiflows $\bigl(\Phi_{\beta(s_i)},\operatorname{H}_L^1\bigr)$, $i=1,2$.
\end{rem}

\begin{thm}\label{Stability_Periodic_L}
Fix arbitrary $L>x_0>0$ and assume that, for any 
$\beta \in \bigl(0,\beta_1 (x_0,L)\bigr)$, the trivial solution of the
semiflow $\bigl(\Phi_{\beta},\operatorname{H}_L^1\bigr)$ is globally
asymptotically stable. Then the noncritical periodic orbits of the semiflow 
$\bigl(\Phi_{\beta},\operatorname{H}_L^1\bigr)$ originating from the
Hopf bifurcation at $\beta_1(x_0,L)>0$ are (orbitally) stable for any  
$\beta\in \bigl(\beta_1(x_0,L),\beta_1(x_0,L)+\delta\bigr)$ for some
$\delta>0$.  
In fact, using the map $\big[ s \mapsto \beta(s) \big]$ discussed in
Remark \ref{Precise_Hopf}, it holds that 
$$
\dot{\beta}(s)>0,
$$
for $0<s<\varepsilon(\delta)$, which means that the Hopf bifuration at
$\beta_1(x_0,L)$ is supercritical.       
\end{thm}
\begin{proof}
The map  $\big[ s \mapsto \beta(s) \big]: (0,\varepsilon)
\rightarrow(0,\infty)$ is continuous and injective for
$0<s<\varepsilon$. Hence it is strictly monotone on $(0,\varepsilon)$.  
Since $\beta(\cdot)$ is differentiable either $\dot{\beta}(s)<0$ or
$\dot{\beta}(s)>0$ must hold for $s\in(0,\varepsilon)$. The case
$\dot{\beta}(s)<0$ can be excluded since it implies the existence of a
noncritical periodic orbit of the semiflow
$\bigl(\Phi_{\beta},\operatorname{H}_L^1\bigr)$ for
$\beta<\beta_1(x_0,L)$.  Since this cannot happen by Theorem
\ref{Main_Result_L} and by the assumption that $\bigl(0,\beta_1
(x_0,L)\bigr)$ is an interval of global stability for the trivial
equilibrium, we conclude that $\dot{\beta}(s)>0$ and that 
the bifurcating noncritical periodic orbits are stable. 
\end{proof}

\begin{rem}\label{DropStabAssump}
The assumption that the trivial solution of the semiflow
$\bigl(\Phi_{\beta},\operatorname{H}_L^1\bigr)$ is globally
asymptotically stable for any $\beta \in \bigl(0,\beta_1 (x_0,L)\bigr)$  
can be dropped if either Conjecture \ref{PopDiriConj} were shown to be
true or if the condition $\widehat{\beta_1}(x_0,L) = \beta_1 (x_0,L)$
discussed in Remark \ref{nogapveri} were shown to hold.
\end{rem} 

The next result settles a conjecture formulated in \cite[Remarks
4.4. (c)]{GM97} for the problem \eqref{tsEq}. It was not stated in
\cite{GM20} even though, in the light of the above results, it is an
immediate corollary to \cite[Theorem 5.1.]{GM20}. We add the result
here for the sake of completeness and due to the fact that the
conjecture in \cite{GM97} provided the initial motivation for both
\cite{GM20} and the present paper. 

\begin{thm}\label{Stability_Periodic_Neumann}
The noncritical periodic orbits of the semiflow
$\bigl(\Phi_{\beta},\operatorname{H}^1(0,\pi)\bigr)$ associated with
\eqref{tsEq} that originate from the Hopf bifurcation at $\beta_0\approx 5.6655$ are (orbitally) stable for 
$\beta\in (\beta_0,\beta_0 +\delta)$ for some $\delta>0$. In other
words, the Hopf bifuration from the trivial solution at
$\beta=\beta_0$ is supercritical.       
\end{thm}
\begin{proof}
The statement follows analogously as in the proof of Theorem \ref{Stability_Periodic_L} since 
the trivial critical point of the semiflow
$\bigl(\Phi_{\beta},\operatorname{H}^1(0,\pi)\bigr)$ is shown to be
globally attractive on the maximal interval of stability, i.e. for
$\beta\in (0,\beta_0)$ in \cite[Theorem 5.1.]{GM20}.
\end{proof}

\section{Implementation used in the numerical calculations}\label{SecNumerical}
In order to generate Figures \ref{1stEfct}, \ref{efctT}, and
\ref{crossing}, we made used of a discretization of the operator
$A^L_\beta$ which is described in this section. As
for Figures \ref{L2infty} and \ref{merging}, the computations are
based on the zeros of the spectrum determining functions $z_L$
found in \eqref{z-fct} inside the proof of Proposition
\ref{betaCritical-Lcase}, and on \eqref{eigEq}, respectively.
Since $A_{L,\beta}=A_L+\beta \delta_0\delta_{x_0}^\top$ and
$A_L=A_{L,0}$ has an explicit spectral resolution in 
terms of its eigenvalues $\mu^k_{L,0}=\frac{\pi^2k^2}{4L^2}$, $k\in
\mathbb{N}$, and eigenfunctions $\varphi_{k,L}=\frac{1}{\sqrt{L}}\sin
\bigl(k\pi\,\frac{x+L}{2L}\bigr)$, we opt for a spectral
discretization. In order to obtain it, we introduce the grid of
equidistant points 
$x^m=(x^m_k)_{k=1,\dots 2^m-1}$ given by
$$
 x^m_k=-L+k \frac{2L}{2^m},\: k=1,\dots, 2^m-1,
$$
and the discrete sine transform matrix $S_m$ with entries
$$
 S_m(k,j)=\varphi_{k,L}(x^m_j).
$$
Then we approximate $A_L$ spectrally by
$$
 A_L^m= \frac{2L}{2^m}\, S_m^\top
 \operatorname{diag}\bigl [\frac{k^2\pi^2}{4L^2}\bigr] _{k=1,\dots,2^m-1}
 \, S_m,
$$
where $S_m^\top=S_m^{-1}$ and the scalar factor amounts to the
application of the quadrature rule (the trapezoidal rule in this case)
required in the discrete transform to approximate the corresponding
continuous integral. The Dirac distribution supported at $y\in (-L,L)$ is also
discretized spectrally as
$$
 \delta^m_y=\sum _{k=1}^{2^m-1}\varphi _{k,L}(y) \varphi _{k,L}(x^m).
$$
This yields a spectral approximation through
$$
 \langle \delta_y, u\rangle
 _{\operatorname{H}^{-1}_L,\operatorname{H}^1_L} \simeq
 \frac{2L}{2^m}\bigl( \delta^m_y \bigr) ^\top u^m,
$$
if $u^m$ is the vector approximating $u\in
\operatorname{H}^1_L$. Again the scalar factor is dictated by the
quadrature rule used to approximate the duality pairing. Finally the
operator $A_{L,m}$ of interest is approximated by
$$
 A^m_{L,\beta}=A^m_L+\beta\frac{2L}{2^m}\delta^m_0 \bigl( \delta
 ^m_{x_0}\bigr)^\top,
$$
and its adjoint by the transpose $\bigl( A^m_{L,\beta}\bigr)
^\top$. Again, this discretization is used for the spectral calculations
leading to Figures \ref{1stEfct}, \ref{efctT}, and \ref{crossing}.
\bibliography{lite}
\end{document}